\documentclass[12pt,reqno]{amsart}
\usepackage[letterpaper]{geometry}
\usepackage{geometry}
\usepackage{amsmath,amsfonts,amsthm,amscd,amssymb,mathrsfs,amssymb, multirow, pst-node, comment,enumitem}
\usepackage{stmaryrd}
\usepackage{graphicx}
\usepackage{verbatim}
\usepackage{mathtools}
\usepackage{tikz}
\usepackage{tikz-cd} 
\newtheorem{theorem}{Theorem}

\newtheorem{proposition}[theorem]{Proposition}
\newtheorem{lemma}[theorem]{Lemma}

\newtheorem{corollary}[theorem]{Corollary}

\numberwithin{figure}{section}
\theoremstyle{definition}
\newtheorem{definition}[theorem]{Definition}
\newtheorem{remark}[theorem]{Remark}

\renewcommand{\H}{\mathcal{H}}
\newcommand{\X}{\mathcal{X}}
\newcommand{\Y}{\mathcal{Y}}
\newcommand{\Z}{\mathcal{Z}}

\newcommand{\PP}{\mathbb{P}}
\newcommand{\CC}{\mathbb{C}}

\newcommand{\RR}{\mathbb{R}}
\newcommand{\ZZ}{\mathbb{Z}}
\newcommand{\QQ}{\mathbb{Q}}

\renewcommand{\i}{i}

\newcommand{\oo}{\mathcal{O}}

\newcommand{\JJ}{\mathcal{J}}
\newcommand{\KK}{\mathcal{K}}

\newcommand*{\sheafhom}{\mathscr{H}\kern -.5pt om}
\newcommand*{\sheafext}{\mathscr{E}\kern -.5pt xt}
\numberwithin{equation}{section}
\numberwithin{theorem}{section}
\numberwithin{table}{section}
\numberwithin{table}{section}

\begin{document}
\bibliographystyle{amsalpha}
\title[Existence and compactness theory]{Existence and compactness theory for \\ ALE scalar-flat K\"ahler surfaces}
\author{Jiyuan Han}
\address{Department of Mathematics, Purdue University, West Lafayette,
IN, 47907}
\email{han556@purdue.edu}
\author{Jeff A. Viaclovsky}
\address{Department of Mathematics, University of California, Irvine, 
CA, 92697}
\email{jviaclov@uci.edu}
\thanks{The second author was partially supported by NSF Grant DMS-1811096.}
\date{April 18, 2019}
\begin{abstract}
Our main result in this article is a compactness result which states that a noncollapsed sequence of  asymptotically locally Euclidean (ALE) scalar-flat K\"ahler metrics on a minimal K\"ahler surface whose K\"ahler classes stay in a compact subset of the interior of the K\"ahler cone must have a convergent subsequence. 
As an application, we prove the existence of global moduli spaces of scalar-flat K\"ahler ALE metrics
for several infinite families of K\"ahler ALE spaces. 
\end{abstract}
\maketitle
\setcounter{tocdepth}{1}
\tableofcontents
\section{Introduction}
\label{intro}
\begin{definition}
An ALE K\"ahler surface $(X, g,J)$ is a K\"ahler manifold of complex dimension $2$ with the following property. There exists a compact subset $K\subset X$ and a diffeomorphism 
$\Psi: X \setminus K\rightarrow (\RR^4\setminus \overline{B})/\Gamma$, such that for each multi-index $\mathcal{I}$ of order~$|\mathcal{I}|$
\begin{align}
\partial^{\mathcal{I}}(\Psi_*(g)-g_{Euc}) = O(r^{-\mu-|\mathcal{I}|}),
\end{align}
as $r \rightarrow \infty$, where $\Gamma$ is a finite subgroup of ${\rm{U}}(2)$ containing no complex reflections, $B$ denotes a ball 
centered at the origin, and $g_{Euc}$ denotes the Euclidean metric. The real number $\mu$ is called the order of $g$.
\end{definition}
\begin{remark}
In this paper, henceforth $\Gamma$ will always be a finite subgroup of ${\rm{U}}(2)$ containing no complex reflections.
\end{remark}
Any ALE K\"ahler surface can be blown-down to a smooth minimal complex surface in its birational class, minimal in the sense that there is no
rational curve of self-intersection $-1$. Our interest lies in building canonical metrics on minimal ALE K\"ahler surfaces. Specifically,
we are interested in constructing a smooth family of ALE SFK (scalar-flat K\"ahler) metrics that corresponding to the versal deformation family of $\CC^2/\Gamma$. 
Before we discuss existence results, we will present our main theorem in this paper, which is a compactness result. 

In the following, if $(X,g)$ is an ALE metric and $\varphi$ is a smooth tensor of any type, we say that $\varphi \in C^{\infty}_{\delta}(X,g)$ if $\phi$ is smooth and 
$\nabla_g^{\mathcal{I}} \varphi = O(r^{\delta - |\mathcal{I}|})$ as $r \rightarrow \infty$, 
where $\mathcal{I}$ is any multi-index of length $|\mathcal{I}|$. 

\begin{definition}
\label{kcdef}
 Let $(X,J)$ be a K\"ahler surface with a smooth ALE K\"ahler metric $g_0$, with K\"ahler form $\omega_0$. For $-2 < \delta_0 < -1$, define
\begin{align}
\mathcal{P}(X,J, \omega_0, \delta_0) = \{ \omega | \ \omega {\text{ is K\"ahler form satisfying \ }} \omega - \omega_0 \in C^{\infty}_{\delta_0}(X,g_0)  \}.
\end{align}
The  {\textit{K\"ahler cone}} of $(X,J)$ with respect to $\omega_0$ and $\delta_0$ is 
\begin{align}
\label{cone}
\KK(X,J,\omega_0, \delta_0) := \{[\omega] \ | \ \omega \in \mathcal{P}(X,J, \omega_0, \delta_0) \},
\end{align}
where $[\omega]$ denotes the class of $\omega$ in $H^2(X,\RR)$. 
\end{definition}

Clearly, $\KK(X,J,\omega_0, \delta_0)$ is a convex subspace in the de Rham cohomology group $H^2(X,\RR)$.
We remark that if $J$ is Stein, then $\KK(X,J, \omega_0, \delta_0)$ is the entire space $H^2(X,\RR)$, but if there exist any holomorphic curves, then it is a proper subset. 
This is because the integral of the K\"ahler form over a holomorphic curve must be strictly positive since it is the area, but if there are no holomorphic 
curves, then there are no constraints. See the discussion in Remark~\ref{KahlerStein} for details.

\begin{definition} The {\textit{lower volume growth ratio}}  of $(X,g)$ is 
\begin{align}
\mathcal{V}(g) \equiv \inf_{x \in X} \inf_{0 < r < 1} \frac{Vol ( B_r(x,g))}{r^4}.
\end{align}
\end{definition}

The following is our main compactness theorem dealing with sequences of ALE SFK metrics with respect to a {\em{fixed}} complex structure. 
\begin{theorem}
\label{compactness}
Let $(X,J,g_0)$ be an ALE minimal K\"ahler surface, associated with an ALE coordinate of 
asymptotic rate $O(r^{\delta_0})$, $(-2<\delta_0<-1)$.
Let $\kappa_i \in \mathcal{K}(X,J, \omega_0, \delta_0)$ be a sequence with 
$\kappa_i \to \kappa_{\infty} \in  \mathcal{K}(X,J, \omega_0, \delta_0)$ as $i \to \infty$. 
If $g_i$ is a sequence of 
ALE SFK metrics with $\omega_i  \in  \mathcal{P}(X,J, \omega_0, \delta_0)$ 
satisfying 
\begin{enumerate}
\item $[\omega_i] = \kappa_i$,
\item there exists a constant $v>0$, independent of $i$,
such that $\mathcal{V}(g_i) > v$,
\end{enumerate}
then there exists a subsequence $\{j\} \subset \{ i \}$ and $\omega_{\infty} \in  \mathcal{P}(X,J, \omega_0, \delta_0)$ such that 
$\omega_j \to \omega_{\infty}$ in $C^{k,\alpha}_{\delta_0}(X,g_0)$ norm 
for any $k \geq 0$, $0 < \alpha < 1$, as $j \to \infty$, 
where $g_{\infty}$ is an ALE SFK metric satisfying $[\omega_{\infty}] = \kappa_{\infty}$.
\end{theorem}
For the definition of the weighted norm, see Section \ref{notations} below. A brief outline of the proof of Theorem \ref{compactness} is follows. First, we apply the compactness result of Tian-Viaclovsky \cite{TV2} to obtain an ALE SFK {\textit{orbifold}} limit $X_{\infty}$,  
in the pointed Cheeger-Gromov sense. In Section \ref{uniform} we will also show that the limit $X_{\infty}$ is birationally equivalent to $(X,J)$. Then, in Section \ref{remove} we will show that the limit space $X_{\infty}$ 
is moreover birationally dominated by $X$, that is, $X_{\infty}$ is a blow-down of $X$. The key point in this step is to show that there are no $(-1)$ curves in the minimal resolution of $X_{\infty}$, the proof of which uses crucially the minimality assumption on $X$.  Then in Section~\ref{bubbles}, using some key results of Lempert, we will show that in the ``bubble tree'' of each orbifold singularity in the limit space,
each bubble is biholomorphic to a {\textit{resolution}} of an orbifold singularity in the previous bubble.
This, together with a result of Laufer, implies the area contraction of a holomorphic curve, 
which contradicts with the non-degeneracy of the limiting K\"ahler class, and 
therefore the limit space must be a smooth ALE SFK metric. 
We remark that Theorem~\ref{compactness} in some sense can be viewed as a 
non-compact analogue 
of the main result in \cite{CLW}.

\begin{definition} For $(X,J,g_0)$ an ALE SFK K\"ahler surface, let
\begin{align}
\mathcal{V}(\mathcal{P}(J)) = \inf_{\substack{g \in \mathcal{P}(J) \\ R_g  =0}} \mathcal{V}(g),
\end{align}
where $\mathcal{P}(J) = \mathcal{P}(X,J, \omega_0, \delta_0)$.
\end{definition}
Our main existence result is the following. 
\begin{corollary} 
\label{corexist}
Let $(X,J,g_0)$ be as in Theorem \ref{compactness},
and assume that $g_0$ is SFK. If 
\begin{align}
\label{Vassump}
\mathcal{V}(\mathcal{P}(J)) > 0,
\end{align} 
then 
for any $\kappa \in \KK(X,J,\omega_0, \delta_0)$, there exists an ALE SFK metric $\omega \in \mathcal{P}(J)$
with $[\omega] = \kappa$. 
\end{corollary}
This theorem is proved by using the continuity method.
Openness in the continuity method follows from the same method in \cite[Section~8]{HanViaclovsky}. Closedness follows from Theorem \ref{compactness}. 

\begin{remark}
The family of ALE SFK metrics constructed by the continuity method 
depends upon the initial metric we choose, but otherwise
does not depend upon the specific value of $\delta_0$ for $-2 < \delta_0 < -1$.
\end{remark}

\begin{remark} In certain examples, we can prove the non-collapsing condition required in Corollary \ref{corexist} by using a topological argument; we will discuss these examples in Subsection~\ref{introex} below.
\end{remark}
\subsection{General  existence results}In order to state our next result, we need to recall some theory regarding the deformations of $\CC^2/\Gamma$. 
By a classical theorem of Grauert \cite{Grauert1972}, (and see \cite{Elkik} for the algebraic version) there exists a (mini)versal deformation
$\Y\rightarrow Der(\Y_0)$ of $\CC^2/\Gamma$, 
such that any deformation of $\CC^2/\Gamma$ 
over a complex space germ can be obtained by a pull-back morphism
from the versal deformation, on the level of germs (see \cite{Greuel} for the complete definition of versality). 
Furthermore, there is a natural $\CC^*$-action on $Der(\Y_0)$,  which lifts to a $\CC^*$-action on $\Y$ (which is of negative weight, see \cite[Section 2]{Pinkham78}). 
The complex space germ $Der(\Y_0)$ can be reducible in general.
Let $r+1$ denote the number of irreducible components, 
and denote each irreducible component by $Der_k(\Y_0)$, $k\in \{0,\ldots,r\}$.
By \cite{KollarShepherd} and \cite{BehnkeChristophersen}, for each irreducible component,
there exists a unique $P$-resolution $Z^P_k\rightarrow \Y_0$, a unique $M$-resolution $Z^M_k\rightarrow Z^P_k$, and finite base changes
$Der'(Z^M_k)\rightarrow Der'(Z^P_k)\rightarrow Der_k(\Y_0)$.
Using the $\CC^*$-action, we can extend $Der'(Z^M_k), Der'(Z^P_k), Der_k(\Y_0)$ 
to global analytic spaces $\JJ^M_k,\JJ^P_k, \JJ_k$, which are
bases spaces of deformations $\X_k, \Z_k, \Y_k$, respectively, 
and the total spaces admit $\CC^*$-actions such that the 
following diagram is $\CC^*$-equivariant
\begin{equation}
\begin{tikzcd}
\label{JJ_k}
\X_k \arrow{r}{} \arrow[swap]{d}{} & \Z_k \arrow{r}{} \arrow[swap]{d}{} & {\Y}_k \arrow{d}{} \\
\JJ^M_k \arrow{r}{} & \JJ^P_k \arrow{r}{} & \JJ_k.
\end{tikzcd}
\end{equation}
Define global base spaces 
\begin{align}
\label{bases}
\JJ^M = \cup_{0\leq k\leq r} \JJ^M_k, \ \JJ^P = \cup_{0\leq k\leq r}\JJ^P_k, \ \JJ = \cup_{0\leq k\leq r} \JJ_k.
\end{align}
Note that while $\JJ$ is connected, the spaces $\JJ^P,\JJ^M$ have $r+1$ connected components. 
We also note that $\JJ^M_0$ is the simultaneous resolution of the Artin component, up to a base change. Further details of this construction can be found in Section~\ref{versal}.

In a recent work of \cite{HRS16}, it is shown that any ALE K\"ahler surface is birationally equivalent to a deformation of $\CC^2/\Gamma$.
Their work indicates that the space of minimal ALE
K\"ahler surfaces is essentially parameterized by $Der(\Y_0)$.
In Lemma \ref{minimal_ALE_moduli} below, we show that
any minimal ALE K\"ahler surface $(X,J)$ is biholomorphic to an element in $\JJ^M$.
For this reason, it is reasonable to first restrict our attention to complex structures parametrized by the base space $\mathcal{J}^M$ (or $\mathcal{J}^P$). 

\begin{theorem}
\label{background_metric}
There exists a smooth family of background ALE K\"ahler metrics
$\omega_{b,J}$, for all smooth fibers over $J \in \JJ^M$ (similarly for $J \in \JJ^P$ away from the discriminant locus).
\end{theorem}
This will be proved in Section~\ref{non-Artin} below. 
Our main interest is therefore in constructing ALE SFK metrics in these ALE K\"ahler classes. 
We emphasize that in all the following results, the K\"ahler cone is defined with respect to the background ALE K\"ahler metric $\omega_{b,J}$.
Thus in the following when there is no ambiguity, we will abbreviate $\mathcal{P}(X,J,\omega_{b,J}, \delta_0)$ and $\KK(X,J,\omega_{b,J}, \delta_0)$ as $\mathcal{P}(J)$ and $\KK(J)$, respectively. 

Recall from above that for each irreducible component $\JJ_k$ in the moduli space $\JJ$ associated to the versal deformation of 
$\CC^2/\Gamma$, there corresponds a $P$-resolution $Z^P_k$ and a $M$-resolution $Z^M_k$.
The space $Z^P_k$ is an orbifold with singularities of type $T$, and 
the space $Z^M_k$  is an orbifold with only type $T_0$ singularities. 
\begin{theorem}
\label{classification}
Let $\JJ$ be the moduli space associated to the versal deformation of $\CC^2/\Gamma$ as defined in the previous paragraphs. Let $\JJ_k$ be an irreducible component.
\begin{enumerate}
\item[(a)] If $\JJ_k = \JJ_0$ is the Artin component, 
then for any complex structure  $J\in \JJ^M_0$
there exists an ALE SFK metric in some K\"ahler class in $\KK({J})$.

\item[(b)]
For $k>0$, 
if there exists an ALE SFK orbifold metric on the orbifold 
$Z^M_k$,  then for any complex structure  
$J\in \JJ^M_k$ away from the central fiber,
there exists an ALE SFK metric in 
some K\"ahler class in $\KK(J)$.

\item[(c)]
For $k>0$, 
if there exists an ALE SFK orbifold metric on the orbifold 
$Z^P_k$, then for any complex structure $J\in \JJ^P_k$ away from the discriminant locus,
there exists an ALE SFK metric for 
some K\"ahler class in $\KK(J)$.
\end{enumerate}
\end{theorem}
Case (a) follows easily from \cite[Theorem 1.4]{HanViaclovsky}.
Cases (b) and (c) are obtained by applying a generalization of a result of Biquard-Rollin to the 
ALE case \cite{BiquardRollin}. For the precise statement, see Theorem~\ref{SFK_non_Artin} below.

Recall that for integers $p,q$ satisfying $(p,q) = 1$, the cyclic action 
$\frac{1}{p}(1,q)$ is that generated by $(z_1, z_2) \mapsto ( \zeta_p z_1, \zeta_p^q z_2)$ 
where $\zeta_p$ is a primitive $p$th root of unity.

\begin{corollary} 
\label{cyclicthm1}
Let $\Gamma = \frac{1}{p}(1,q)$ be any cyclic group with $(p,q) = 1$,
and let $\mathcal{J}^M_k$ be any component of $\mathcal{J}^M$.
Then for any $J \in \mathcal{J}^M_k$ ($J$ is away from the central fiber if $k>0$), 
there exists a scalar-flat K\"ahler metric $\omega_J$ in some 
K\"ahler class. 
\end{corollary}
This is obtained by using the Calderbank-Singer construction from \cite{CalderbankSinger}, together with Theorem~\ref{classification}. 

\subsection{Global existence results} 
\label{introex}
 We now turn our attention to existence 
of global moduli spaces of ALE SFK metrics for certain groups $\Gamma$. 
The following theorem is an application of Case~(a) in Theorem~\ref{classification} together with 
Corollary \ref{corexist}. 
\begin{theorem}
\label{theorem1}
Let $\Gamma \subset {\rm{U}}(2)$ be any of the following groups:
\begin{align}
\label{validgroups}
\frac{1}{3}(1,1), \ \frac{1}{5}(1,2), \ \frac{1}{7}(1,3).
\end{align}
Note that for these groups, the versal deformation space of $\CC^2 / \Gamma$ 
has only the Artin component $\mathcal{J}$, which has $b_2(X) =1,2,3$, respectively,
where $b_2$ denotes the second Betti number.   
Then for any complex structure $J \in \mathcal{J^M}$, and any 
K\"ahler class 
$[\omega] \in \mathcal{K}(J)$, there exists a scalar-flat K\"ahler ALE metric $g$ satisfying $[\omega_g] = [\omega]$.
\end{theorem}

\begin{remark} 
Our method also proves an analogous global existence result for the 
case $\Gamma \subset {\rm{SU}}(2)$. However, this case was explicitly constructed by Kronheimer using the hyperk\"ahler quotient construction \cite{Kronheimer}, so 
we do not devote any extra attention to this case. Note also that the $\QQ$-Gorenstein smoothings of the type $T$ cyclic singularities  
admit Ricci-flat K\"ahler metrics which are just quotients of the $A_k$-type hyperk\"ahler metrics by finite groups of isometries \cite{Suvaina_ALE, Wright_ALE}. These metrics play a crucial role in our analysis of non-Artin components. 
\end{remark}

\begin{remark}
 A drastic difference between the ADE cases and the non-ADE cases, is that the global moduli spaces in the latter cases can 
have ``holes'' which can only be filled in by certain smoothings of orbifolds which have non-minimal resolutions. This phenomenon arises already in the case of $\mathcal{O}(-n)$ for $n \geq 3$. 
See Section~\ref{conclusion} below for details of these examples. 
\end{remark}

The groups in Theorem \ref{theorem1} have only Artin components. 
The next result deals with five infinite families of non-Artin 
components, and is an application of Case (b) in Theorem~\ref{classification}, 
together with Corollary \ref{corexist}.  
\begin{theorem}
\label{exthm2}
Let $\Gamma \subset {\rm{U}}(2)$ be any of the following groups for $r\geq 2$
\begin{align}
\tag{1} \Gamma &= \frac{1}{r^2+r+1}(1,r),\\
\tag{2} \Gamma &= \frac{1}{r^2+2r+2}(1,r+1) \mbox{\ or \ } \Gamma = \frac{1}{2r^2+2r+1}(1,2r+1),\\
\tag{3} \Gamma &= \frac{1}{r^2 + 3r + 3} (1, r+2) \mbox{\ or  \ } \Gamma = \frac{1}{3r^2+3r+1}(1,3r+2).
\end{align}
There is a non-Artin component $\mathcal{J}(i)$ of the versal deformation space of $\CC^2/\Gamma$ with $b_2(X) = i$ in Case $(i)$, $i = 1 ,2, 3$. 
For any complex structure $J \in\mathcal{J}^M(i)$ away from the central fiber, and any K\"ahler class $[\omega] \in \mathcal{K}(J)$, 
there exists a scalar-flat K\"ahler ALE metric $g$ satisfying $[\omega_g] = [\omega]$.
\end{theorem}
Finally, we conjecture that the assumption on the lower volume growth ratio 
is redundant, and that for any group $\Gamma$, there exists ALE SFK metrics in 
all K\"ahler classes for all complex structures in the versal family. 

\subsection{Acknowledgements} 
The authors would like to thank Simon Donaldson and Gang Tian 
for providing motivating comments during the early stages of this project. 
The authors had helpful discussions on the deformation theory of ALE K\"ahler surfaces with Mao Li, Rares Rasdeaconu, and Song Sun.
Hans-Joachim Hein provided assistance on numerous occasions throughout the preparation of this article. Finally, the authors owe a huge debt of gratitude to Claude LeBrun for invaluable remarks on an early draft of this article, and for many other insightful comments.  

\section{Preliminaries}

\subsection{Notation}
\label{notations}
In this section, we record some symbols and notations that will be used in this article.
Weighted H\"older spaces are defined as follows. 
\begin{definition}
\label{weighted_norm}
Let $E$ be a tensor bundle on $X$, with Hermitian metric $\Vert \cdot\Vert_h$. 
Let $\varphi$ be a smooth section of $E$. We fix a point $p_0\in X$, and define $r(p)$ to be the distance between $p_0$ and $p$. Then define
\begin{align}
\Vert \varphi\Vert_{C^{0}_\delta} &:= \sup_{p\in X}\Big\{\Vert\varphi(p)\Vert_h\cdot (1+r(p))^{-\delta}\Big\}\\
\Vert \varphi\Vert_{C^{k}_\delta} &:= \sum_{|\mathcal{I}|\leq k}\sup_{p\in X}
\Big\{\Vert\nabla^{\mathcal{I}} \varphi(p)\Vert_h\cdot (1+r(p))^{-\delta+|\mathcal{I}|}\Big\},
\end{align}
where $\mathcal{I} = (i_1,\ldots,i_n),|\mathcal{I}|=\sum_{j=1}^n i_j$.
When there is no ambiguity, if $|\mathcal{I}|=d$, we will abbreviate 
$\nabla^{\mathcal{I}} \varphi$ by $\nabla^{(d)}\varphi$.
Next, define
\begin{align}
[\varphi]_{C^{\alpha}_{\delta-\alpha}} &:= \sup_{0<d(x,y)<\rho_{inj}}\Big\{\min\{r(x),r(y)\}^{-\delta+\alpha}
\frac{\Vert\varphi(x)-\varphi(y)\Vert_h}{d(x,y)^\alpha}\Big\},
\end{align}
where $0<\alpha<1$, $\rho_{inj}$ is the injectivity radius, and $d(x,y)$ is the distance between $x$ and $y$. 
The meaning of the tensor norm is via parallel transport along the unique minimal geodesic from  $y$ to $x$, and then take the norm of the difference
at $x$.
The weighted H\"older norm is defined by
\begin{align}
\Vert\varphi\Vert_{C^{k,\alpha}_\delta} &:= \Vert \varphi\Vert_{C^{k}_\delta}+\sum_{|\mathcal{I}|=k}
[\nabla^{\mathcal{I}} \varphi]_{C^{\alpha}_{\delta-k-\alpha}},
\end{align}
and the space $C^{k,\alpha}_{\delta}(X, E)$ is the closure of $\{\varphi \in C^{\infty}(X,E): \Vert \varphi\Vert_{C^{k,\alpha}_\delta}<\infty\}$.
\end{definition}

\begin{itemize}
\item
$\epsilon(i\;|\;\delta)$: The symbol $\epsilon(i\;|\;\delta)$ represents a small positive number, and for any fixed 
$\delta>0$, $\epsilon(i\;|\;\delta)\to 0 $ as $i\to\infty$.
\item
$\Lambda^\cdot, \Lambda^{\cdot,\cdot},\Omega^\cdot$: $\Lambda^p$ stands for the space of real $p$-forms, $\Lambda^{p,q}$ stands for the space of complex $(p,q)$-forms,
$\Omega^p$ stands for the space of complex $(p,0)$-forms.
\item
$\widetilde{X}$: For a complex variety $X$ of complex dimension $2$, $\widetilde{X}$ stands for the minimal resolution of $X$.
\item
$\overline{V}$:
For a topological space $V$, $\overline{V}$ stands for its universal cover.
\item $g, \omega$: We will denote the Riemannian metric by $g$ and $\omega = g(J\cdot, \cdot)$ as the corresponding K\"ahler form. 
But on occasion when there is no ambiguity, we will use these two symbols alternatively for convenience.
\end{itemize}

\subsection{Facts about ALE K\"ahler surfaces}
We list some facts about ALE K\"ahler surfaces which we will use later. We will always assume the asymptotic rate $-\mu<-1$.

By applying Hodge index theorem as shown in \cite[Proposition 4.2]{HL15},  an ALE K\"ahler surface has only one ALE end. 
As pointed out by Hein-LeBrun, for an ALE K\"ahler metric $(X,g,J)$ of order $\mu$, the complex structure has an asymptotic rate of
\begin{align}
\label{HLJ}
\partial^{\mathcal{I}} (J - J_{Euc})= O(r^{-\mu-|\mathcal{I}|}),
\end{align}
for any multi-index $\mathcal{I}$ as $r \rightarrow \infty$, where $J_{Euc}$ is the standard complex structure on Euclidean space. 
This is because, $\nabla_{g_{Euc}}J = (\nabla_{g_{Euc}}-\nabla_{g})J = O(r^{-\mu-1})$. The integral along each $g_{Euc}$-geodesic ray implies the ALE asymptotic
rate of $J$ as above.

\begin{remark}Although our proof will not require the following, we make a remark on the optimal decay rates of the metric and complex structure. For any ALE SFK metric, there exists an ALE coordinate with optimal metric asymptotic rate of $O(r^{-2})$, see \cite{AV12, LebrunMaskit,Streets}. Furthermore, by \cite[Proposition 4.5]{HL15}, 
for $(X,g,J)$ of order $\mu$, there exists an ALE coordinate which is still at least of order $\mu$, and for which $J$ converges to the Euclidean complex structure $J_{Euc}$ at the rate of $O(r^{-3})$.
Therefore, if $g$ is ALE SFK, there always exists an ALE coordinate so that the metric $g$ converges to $g_{Euc}$ at the rate of $O(r^{-2})$
and $J\sim J_{Euc} + O(r^{-3})$ as $r \to \infty$. 
\end{remark}

For an ALE K\"ahler surface $X$, $\H_{-3}(X,\Lambda^{1,1}_\RR)$ stands for the space of decaying real harmonic $(1,1)$-forms. Note that any decaying real harmonic $(1,1)$-form has a decay rate at least $O(r^{-3})$, 
and $H^2(X,\RR) \cong \H_{-3}(X,\Lambda^{1,1}_\RR)$ (for details see \cite[Section 7]{HanViaclovsky} and \cite[Sections~8.4 and 8.9]{Joyce2000}).
We have the following which is a consequence of a $\partial\bar\partial$-lemma for K\"ahler forms as shown  in \cite[Lemma 8.3]{HanViaclovsky}.
\begin{lemma}
\label{d-dbar}
For any two smooth K\"ahler metrics $\omega_1,\omega_2$ over an ALE K\"ahler surface $(X,J)$, if $\omega_1 - \omega_2 = O(r^{\nu-2})$, $(0<\nu<1)$, and 
$\int_X (\omega_1-\omega_2)\wedge h = 0$ for any $h\in \H_{-3}(X,\Lambda^{1,1}_\RR)$,
then there exists  $\phi\in C^\infty_{\nu}(X, \RR)$, such that $\omega_2 = \omega_1 + \sqrt{-1}\partial\bar\partial\phi$.
\end{lemma}
In particular, this shows that our definition of the K\"ahler cone in Definition \ref{kcdef} is the ``correct'' one: any two K\"ahler forms whose difference decays and is zero in the de Rham cohomology group $H^2(X,\RR)$, must differ by $\sqrt{-1} \partial\bar{\partial} \phi$, where $\phi$ is of sub-linear growth rate. 

Another important fact about ALE K\"ahler surfaces is that they are one-convex, which we define next.

\begin{definition}[One-convex surface] A one-convex surface $X$ is 
a noncompact complex surface carrying a $C^{\infty}$-exhaustion function 
$f: X \rightarrow [0, \infty)$ which is strictly plurisubharmonic outside a compact set.
\end{definition}
To see that an ALE K\"ahler surface is one-convex: using an ALE coordinate system,
 extend the pullback of the function $r_{Euc}^2$ to a smooth 
non-negative function on all of $X$,
and this will be the required function $f$. 
Any one-convex surface $X$ is a point modification of a Stein space $Y$, that is, 
$X$ is obtained from $Y$ by substituting some points with compact analytic sets,
for more details, see  \cite[Theorem~2.1]{Peternell}.
On a one-convex surface $X$, any holomorphic function defined outside of a compact set can be extended to a holomorphic function on $X$.
This is because a holomorphic function defined outside of a compact set on the Remmert reduction $Y$ can be extended to a holomorphic function on $Y$ by \cite[Theorem 6.1]{Rossi1963}, and then can be lifted up to a holomorphic function on $X$.

\subsection{Versal deformation of $\CC^2/\Gamma$}
\label{versal}
In this subsection, we will provide more details of the versal family, and the deformation to the normal cone construction.

By Artin \cite{Artin1974} and Wahl \cite{Wahl1979}, there exists an irreducible component $Der_0(\Y_0)\subset Der(\Y_0)$,
with a finite base change (which is a Galois cover) $Res\rightarrow Der_0(\Y_0)$, such that there exists a simultaneous resolution
$\X$ that satisfies the commutative diagram:
\begin{equation}
\label{JJ_0}
\begin{tikzcd}
\X \arrow{r}{} \arrow[swap]{d}{} & {\Y} \arrow{d}{} \\
{Res} \arrow{r}{} & {Der_0(\Y_0)}
\end{tikzcd}
\end{equation}
The base $Der_0(\Y_0)$ is called the Artin component of the versal deformation.
The Artin component is the only irreducible component which admits a simultaneous resolution.
According to Wahl, $Der_0(\Y_0) = Res/W$, where $W$ is the Weyl group action.
Since the $\CC^*$-action is preserved under the finite base change, 
we can apply the $\CC^*$-action on $Res$.
Then we obtain a global analytic space $\JJ_0$ and a family $\X\rightarrow \JJ_0$. Each fiber 
$\X_t$ is smooth.

 We recall some facts from \cite{KollarShepherd}.
There exists a one-parameter $\QQ$-Gorenstein smoothing of $\CC^2/\Gamma$ 
if and only if 
$\Gamma\subset SU(2)$, or $\CC^2/\Gamma$ is a type $T$ singularity, that is, 
$\Gamma$ is cyclic of type 
$ \frac{1}{r^2 s} (1, r s d -1)$ where $r \geq 2, s \geq 1, (r,d) = 1$.
See Section~\ref{existsec} below for more details about type $T$ singularities.
For each non-Artin component $Der_k(\Y_0)$ $(k>0)$,  
there exists a $P$-resolution $Z^P_k$ with only type $T$ singularities, which has a local moduli space of $Der'(Z^P_k)$ which is the component 
corresponding to $\QQ$-Gorenstein smoothings. 
Furthermore, there exists a finite base change $Der'(Z^P_k)\rightarrow Der_k(\Y_0)$.

Next, we recall some facts from \cite{BehnkeChristophersen}. 
There exists an $M$-resolution $Z^M_k\rightarrow Z^P_k$ with only
type $T_0$ singularities (type $T$ singularities with $s=1$), 
which has a local moduli space $Der'(Z^M_k)$, where all
nearby fibers are smooth; here $Der'(Z^M_k)$ denotes the component 
corresponding to $\QQ$-Gorenstein smoothings. 
There exists a finite base change $Der'(Z^M_k)\rightarrow Der'(Z^P_k)$.
All together, we have the commutative diagram
\begin{equation}
\label{JJ_k2}
\begin{tikzcd}
\X \arrow{r}{} \arrow[swap]{d}{} & {\Z} \arrow{r}{} \arrow[swap]{d}{} & {\Y} \arrow{d}{} \\
{Der'_k(Z^M_k)} \arrow{r}{} & Der'(Z^P_k) \arrow{r}{} & {Der_k(\Y_0)}.
\end{tikzcd}
\end{equation}

Each fiber $\Z_t$ is smooth away from the discriminant locus.
Each fiber $\X_t$ is smooth save the central fiber.
For $t''\in Der'(Z^M_k),t'\in Der'(Z^P_k),t\in Der_k(\Y_0)$ with $t''$ mapped to $t'$, $t'$ mapped to $t$, there exists resolutions
$\X_{t''}\rightarrow \Z_{t'}\rightarrow \Y_t$, and $\X_{t''}$ is minimal when $t''\neq 0$.
$\JJ^M_k, \JJ^P_k$ are generated by applying the $\CC^*$-action on $Der'(Z^M_k)$, $Der'(Z^P_k)$ respectively. 

By \cite{HRS16}, any ALE K\"ahler surface is birationally equivalent to an element in the versal deformation of $\CC^2/\Gamma$.
We will review some details of the construction in \cite{HRS16} which will be needed in our proof.  For an ALE K\"ahler surface $X$ under our consideration, the asymptotic rate of the complex structure is faster than
$O(r^{-1-\epsilon})$. By \cite{HL15}, the $O(r^{-1-\epsilon})$ asymptotic rate of the complex structure implies that
$X$ can be compactified analytically to a compact orbifold $\hat{X} = X\cup D$, where $D$ is isomorphic to $\PP^1$
quotient by a finite group (see \cite{Li14} for the more general asymptotically conical case).
There exists a positive integer $m\in \ZZ_+$ such that $m\cdot D$ is a Cartier divisor,
which induces a line bundle $L$ in $\hat{X}$.
By a Nakai-Moishezon type argument, it is shown in \cite{HRS16} that for some $k\in \ZZ_+$ large enough,  $H^0(\hat{X},L^k)\rightarrow H^0(D,L^k)$
is surjective and $L^k\rightarrow \hat{X}$ is globally generated.
As a result, there exist holomorphic sections $s_0,\ldots,s_N$ in $H^0(\hat{X},L^k)$, where $s_0$ is the defining section of $km\cdot D$,
i.e., $s_0$ vanishes exactly on $D$, such that images of $s_1,\ldots,s_N$ in $H^0(D,L^k)$ are generators.
Then the linear system $|H^0(\hat{X},L^k)|$ maps $\hat{X}$ to $\PP^N$ by $[s_0,\ldots,s_N]$, where the image ${\hat{X}}'$ is 
birationally equivalent to $\hat{X}$.
Furthermore, $u^1=\frac{s_1}{s_0},\ldots,u^N=\frac{s_N}{s_0}$ can extend to holomorphic functions on $X$, and 
$u = (u^1,\ldots,u^N)$ maps $X$ to $X'$ in $\CC^N$. 
Define the graded ring
\begin{align}
R = \bigoplus_{\substack{n\geq 0}}H^0(\hat{X},\oo(n\cdot D))
\end{align}
which is finitely generated.
Let $R[z]$ be a graded ring where $z$ is a free variable of the degree $1$ and is defined as
\begin{align}
R[z] = \bigoplus_{\substack{n\geq 0}}(\bigoplus_{\substack{0\leq j\leq n}}H^0(\hat{X},\oo(j\cdot D))\cdot z^{n-j})
\end{align}
The deformation to the normal cone is defined by 
\begin{align}
\label{defnc}
\hat{\X}' = \{s-t z=0\} \subset Proj(R[z])\times \CC,
\end{align}
where $s$ is the defining section of $D$, 
$t\in \CC$,
$\hat{\X}'_1$ is identified with ${\hat{X}}'$, and $C(D):= \hat{\X}'_0\setminus D$ is the normal cone of $D$.
This implies that $X'$ is a deformation of $\CC^2/\Gamma$.
By versality, the deformation to the normal cone can be considered as a pull-back of the versal deformation of $\CC^2/\Gamma$.

We next have the following proposition which parameterizes all {\textit{minimal}} ALE K\"ahler surfaces.
\begin{proposition}
\label{minimal_ALE_moduli}
Each minimal ALE K\"ahler surface is biholomorphic to an element in $\JJ^M$.
\end{proposition}
\begin{proof}
Let $X$ be a minimal ALE K\"ahler surface with an end asymptotic to $\CC^2/\Gamma$. Then there exists no $(-1)$-curve in $X$.
By the result of \cite{HRS16}, $X$ is birationally equivalent to $Y$, which is a deformation of $\CC^2/\Gamma$.
By the commutative diagram \eqref{JJ_k2}, there exists an element $X'$ in $\JJ^M$, which is the minimal resolution of $Y$.
Since $X, X'$ are one-convex spaces and they are birationally equivalent with each other, there exist compact subsets $K\subset X,K'\subset X'$,
and a biholomorphic map $\Phi: X\setminus K\rightarrow X'\setminus K'$.
Furthermore, by choosing $K$ large enough,
there exist holomorphic functions $u^1,\ldots,u^N$ on $X\setminus K$, which embed $X\setminus K$ into $\CC^N$ by $u = (u^1,\ldots,u^N)$.
Since $X$ is one-convex, $u$ can be extended to a holomorphic map on $X$. Meanwhile $u' = u\circ\Phi^{-1}$ embeds $X'\setminus K'$ into $\CC^N$ and can be
extended to a holomorphic map on $X'$. The image  $u(X\setminus K)$ in $\CC^N$ conincides with the image $u'(X'\setminus K')$,
which is denoted by $V$. The boundary of $V$ is a strictly pseudoconvex manifold ($V$ itself is called strictly pseudoconcave). By \cite[Theorem~10.4]{HarveyLawson},
there exists a unique Stein space $W$ in $\CC^N$, which extends from $V$ through its boundary smoothly.
By uniqueness of analytic extension, $u(X),u'(X')$ concide with $W$, and thus $W$ is the Remmert reduction of $X, X'$. Since each isolated $2$-dimensional quotient
singularity, there exists a unique minimal resolution, then $W$ has a unique minimal resolution.
Then by the minimality of $X,X'$, they are both biholomorphic to the minimal resolution of $W$.
\end{proof}

\subsection{Volume local non-collapsing}
\label{volume}

Let $(X,g)$ be an ALE SFK metric, with the complex orientation so that $W^+_g \equiv 0$,
and group $\Gamma$ at infinity.  
Let $(M, [\hat{g}])$ be the orbifold conformal compactification, with the reversed orientation so that the group at the orbifold point is also $\Gamma$ \cite{ViaclovskyFourier}. 
Since the orientation is reversed, we have that $W^-_{\hat{g}} \equiv 0$. 
Note that $[\hat{g}]$ is a priori a self-dual \textit{conformal} structure, but 
by \cite[Proposition~12]{CLW}, we can assume that there is a metric representative 
$\hat{g} \in [\hat{g}]$ which is moreover a {\textit{smooth}} Riemannian orbifold.  

The Hirzebruch signature theorem for orbifolds \cite{Kawasaki} states that,  
\begin{align}
\label{taum}
\tau(M)= \frac{1}{12\pi^2}\int_{M} ||W^+||^2dV_{g}  - \eta(S^3/\Gamma),
\end{align}
and $\eta(S^3/\Gamma)$ is the $\eta$-invariant of the signature complex, which for a finite subgroup $\Gamma \subset {\rm SO}(4)$ acting freely on $S^3$, is given by
\begin{align}
\label{eta-invariant}
\eta( S^3/\Gamma) = \frac{1}{|\Gamma|} \sum_{\gamma \neq Id\in \Gamma}  \cot \Big(\frac{ r(\gamma)}{2}\Big) 
\cot\Big(\frac{s(\gamma)}{2}\Big),
\end{align}
where $r(\gamma)$ and $s(\gamma)$ denote the rotation numbers of $\gamma\in\Gamma$.

The Chern-Gauss-Bonnet theorem for orbifolds \cite{Kawasaki} states that
\begin{align}
\label{chim}
\chi(M)= \frac{1}{8\pi^2}\int_{M}\Big( ||W||^2-\frac{1}{2} |E|^2 + \frac{1}{24}R^2\Big)dV_{g}  + \Big( 1 - \frac{1}{|\Gamma|} \Big),
\end{align}
where $E$ denotes the traceless Ricci tensor, and $R$ denotes the scalar curvature. 

Using \eqref{taum} and \eqref{chim}, we obtain 
\begin{align}
2 \chi(M) - 3 \tau(M) =  \frac{1}{4\pi^2}\int_{M} \Big( -\frac{1}{2} |E|^2 + \frac{1}{24}R^2\Big)dV_{\hat{g}} + 2\Big (1 - \frac{1}{|\Gamma|} \Big) + 3  \eta(S^3/\Gamma).
\end{align}
Define the quantity
\begin{align}
\mathcal{C}(X) = 2 \chi(M) - 3 \tau(M) - 2\Big (1 - \frac{1}{|\Gamma|} \Big) - 3  \eta(S^3/\Gamma).
\end{align}
Then we obtain 
\begin{align}
\mathcal{C}(X) \leq  \frac{1}{96 \pi^2}\int_{M} R^2 dV_{\hat{g}}.
\end{align}
We note that the conformal class is of positive type, that is, 
$Y(M,  [\hat{g}]) >0$ \cite{AB2,CLW}.
If there exists a minimizing solution of the Yamabe problem on the 
orbifold $(M, [\hat{g}])$ then since the scalar curvature is constant we obtain
the lower estimate on the Yamabe invariant.
\begin{align}
\label{leyi}
Y(M,  [\hat{g}]) \geq 4 \sqrt{6} \pi \sqrt{\mathcal{C}(X)}.
\end{align}
If there does not exist a Yamabe minimizer, then the estimate of 
Akutagawa-Botvinnik \cite{Akutagawacoyi, AB2} says that the Yamabe invariant must be maximal
\begin{align}
Y(M,  [\hat{g}]) = \frac{8 \sqrt{6} \cdot \pi}{\sqrt{|\Gamma|}}.
\end{align}
In either event, if $\mathcal{C}(X) > 0$ we have that the Yamabe invariant is 
strictly bounded below by a positive constant. From \eqref{leyi}, we have
\begin{align}
\int_M u \square_{\hat{g}} u dV_{\hat{g}} \geq  4 \sqrt{6} \pi \sqrt{ \mathcal{C}(X)} 
\left\{ \int_M u^4 dV_{\hat{g}} \right\}^{1/2}
\end{align}
for any $u \in C^{\infty}(M)$, where 
\begin{align}
\square_{\hat{g}} = - 6 \Delta_{\hat{g}} + R_{\hat{g}} 
\end{align}
is the conformal Laplacian. 

Writing $\tilde{g} = v^{2} g$, 
we have the transformation formula
\begin{align}
\label{cltf}
\square_{\tilde{g}} ( u) = v^{-3} \square_g ( uv). 
\end{align}
This yields
\begin{align}
\int_M f \square_{g} f dV_{g} \geq  4 \sqrt{6} \pi \sqrt{\mathcal{C}(X)} 
\left\{ \int_M f^4 dV_{g} \right\}^{1/2}.
\end{align}
Since $g$ is scalar flat, $\square_g = -6 \Delta_g$, so 
we obtain the $L^2$-Sobolev inequality
\begin{align}
\left\{ \int_X f^4 dV_g \right\}^{1/2} \leq \frac{\sqrt{6}}{ 4 \pi \sqrt{\mathcal{C}(X)}} \int_X |\nabla f|^2 dV_g,
\end{align}
for all $f \in C^{\infty}_c(X)$. 

Note that since $M = X \cup \{pt\}$, we have $\chi(M) = \chi(X) + 1$. 
Also, since the orientation is reversed, we have $\tau(M)= - \tau(X)$. 
Since $(X,g)$ is K\"ahler ALE, we have $b_1(X) = 0$. Therefore 
\begin{align}
\mathcal{C}(X) = 2 - b_2(X) + \frac{2}{|\Gamma|} - 3  \eta(S^3/\Gamma).
\end{align}
Therefore, we have the following:
\begin{proposition}
If $(X,J,g)$ is an ALE SFK metric with $\mathcal{C}(X) > 0$, then 
there exists a constant $v>0$, depending only upon $X$, such that $\mathcal{V}(g) > v$. 
\end{proposition}
\begin{proof} The above argument shows that there is a uniform $L^2$-Sobolev inequality. The lower volume growth estimate follows from this by a standard argument, see \cite[Lemma~3.2]{Hebey}. 
\end{proof}
For any component $ \JJ^M_k$, we define $\mathcal{C}(\JJ^M_k)$ to be 
$\mathcal{C}(X)$, where $X$ is diffeomorphic to a smooth fiber of the 
component $\JJ^M_k$ (noting that any two such fibers are diffeomorphic).

\subsection{Cheeger-Gromov convergence}
\label{CGsubsec}
We begin this subsection with the following notion of convergence. 
\begin{definition}[Pointed Cheeger-Gromov convergence]
A sequence of K\"ahler manifolds $(X_i,g_i,J_i,x_i)$ converges to a K\"ahler orbifold space 
$(Z, g, J, z)$ in the pointed Cheeger-Gromov sense if 
$(X_i,g_i,x_i)$ converges to $(Z,g,z)$ in the pointed Gromov-Hausdorff sense,
and there exists a subset $S = \{p_1,\ldots,p_m\}\subset Z$ which contains the singular set of $Z$, 
for any compact subset $K\subset Z\setminus S$ containing $z$, there exists diffeomorphisms $\psi_i: K\rightarrow X_i$, such that $\psi_i^*g_i, \psi_i^* J_i$ converges to $g,J$ in $C^{k, \alpha}(K)$-sense, for some $k, \alpha$. 
\end{definition}
We refer to \cite{Anderson, Bando, BKN, Nakajimaconv, TV2, Tian} for more details on this type of convergence.

First recall the $\epsilon$-regularity theorem proved in \cite{TV,TV3}.
Let $(X,g)$ be a complete scalar-flat K\"ahler $4$-dimensional manifold, with a local volume ratio lower bound
$v>0$, i.e., $vol(B_r(x))>v\cdot r^4$ for any $|r|<1$.
In \cite[Theorem 1.1]{TV}, by studying the PDE system with a Moser-iteration type argument,
\begin{align}
\Delta_g Ric &= Rm \,*\, Ric \\ 
\Delta_g Rm &= L(\nabla^2_g Ric) + Rm \,*\, Rm
\end{align}
the authors proved that there exists an $\epsilon_0 = \epsilon_0(v)>0$, such that if
$\int_{X\setminus B_R(x_0)}\|Rm(g)\|^2 dV_g <\epsilon_0$,
then there exists $C = C(v)>0$, such that $\|Rm(g)\|< C\cdot r^{-2}$ on $X\setminus B_R(x_0)$,
where $x_0$ is a point in $X$, $B_R(x_0)$ is the geodesic ball centered at $x_0$ with a radius of $R$.
Note that the argument in \cite{TV} required a Sobolev constant bound, but this was weakened to only a lower volume growth assumption in \cite{TV3}. 
Furthermore, by Kato's inequality and a further analysis of the connection form, for any $-2<-\mu<-1$,
for any positive integer $k$, there exists $C' = C'(v,k)>0$, such that, on $X\setminus B_R(x_0)$,
$\|\nabla^{(k)} Rm(g)\| < C'\cdot r^{-2-\mu-k}$. 
We call the $\epsilon_0$ above the ``energy threshold''.

By the proof of \cite[Theorem 1.1]{BKN}, there exists a harmonic coordinate on the universal cover of
$X\setminus B_R(x_0)$, which provides an ALE coordinate
\begin{align}
\label{BKN}
H: {X\setminus B_R(x_0)}\rightarrow \RR^4/\Gamma,
\end{align}
and constants $C'' = C''(v,k)>0$, such that
\begin{align}
|\partial^{(k)}(H_*g - g_{Euc})| < C''\cdot r^{-\mu-k}.
\end{align}
Note that the harmonic coordinates are technically defined on the universal cover $\overline{X\setminus B_R(x_0)}$, which is a mapping 
$\overline{H}: \overline{X\setminus B_R(x_0)}\rightarrow \RR^4$ defined by harmonic functions of ``linear growth''. 
However, by the rigidity of harmonic coordinates proved in  \cite[Corollary 3.2]{Bartnik}, for any $\gamma\in \Gamma$, 
$\gamma^* \overline{H} = \gamma \cdot \overline{H}$, where in the latter formula $\gamma$ is considered as a linear map in
${\rm{SO}}(4)$. This implies that $\overline{H}$ is $\Gamma$-equivariant and can descend to a map 
$H: X\setminus B_R(x_0)\rightarrow \RR^4/\Gamma$. 

\begin{definition}  An energy concentration point $x_{\infty} \in X_{\infty}$ is a point such that 
for any $\delta > 0$, there exists $x_i \in X_{i}$ with $x_i \to x_{\infty}$ (in the Gromov-Hausdorff distance), 
and such that 
\begin{align}
 \int_{B_{\delta}(x_i)} \|Rm(g_i)\|^2 dV_{g_i} \geq \epsilon_0,
\end{align}
where $\epsilon_0$ is the energy threshold. 
\end{definition}

We next define a stronger notion of pointed Cheeger-Gromov convergence in the ALE setting which includes the convergence near $\infty$. 
\begin{definition}
\label{uniform_asymptotic_rate}
Let $(X_i,J_i,g_i,x_i)$ be a sequence of ALE K\"ahler surfaces, where each $g_i$ is asymptotic to $g_{Euc}$ of order
$O(r^{-\mu})$ $(-2<-\mu<-1)$ with respect to a fixed ALE coordinate. 
We say the sequence $\{(X_i,J_i,g_i,x_i)\}$ converges in the sense of 
``pointed Cheeger-Gromov with a uniform ALE asymptotic rate of order $O(r^{-\mu})$''
if there exists an ALE K\"ahler orbifold $(X_\infty,J_\infty,g_\infty,x_\infty)$, where $p_1,\ldots,p_m$
are ``energy concentration '' points in $X_\infty$, such that
\begin{align*}
(X_i,J_i,g_i,x_i)\xrightarrow{\mathrm{pointed \; Cheeger-Gromov}} (X_\infty,J_\infty,g_\infty,x_\infty)
\end{align*}
for any $k\in \ZZ_{\geq 0}, 0<\alpha<1$,
and for any $\delta>0$, when $i$ is sufficiently large, there exists a diffeomorphism
\begin{align*}
\psi_i: X_\infty\setminus \sqcup_{1\leq j\leq m}B_\delta(p_j)\rightarrow X_i
\end{align*}
such that
$\|\psi_i^*g_i -g_\infty\|_{C^{k,\alpha}_{-\mu}(g_\infty)}< \epsilon(i \;|\; k,\delta)$, 
$\|\psi_i^*J_i -J_\infty\|_{C^{k,\alpha}_{-\mu}(g_\infty)}< \epsilon(i \;|\; k,\delta)$.
\end{definition}
Note that if a sequence converges in the above sense, then $X_\infty$ has end diffeomorphic to $\RR^4/\Gamma$ with the same group $\Gamma$
as for $X_i$. Also, for each ``energy concentration'' point $p$ above, there exists a sequence of points $p_i\in X_i$, where
$\lim\limits_{i\to\infty} \|Rm(p_i)\|_{C^0(g_i)}\to \infty$.
We also remark that $p$ may not strictly be an orbifold point, since the ``bubble'' appearing at $p$ could be $\oo_{\CC P^1}(-1)$ with the Burns metric \cite{Burns,Calabi1979}.

\begin{lemma}
\label{uniform_diffeom}
Consider a sequence of ALE SFK metrics $(X_i,J_i,g_i,x_i)$ which are ALE of asymptotic rate 
$O(r^{-\mu})$ with respect to a fixed ALE coordinate, where $-2<-\mu<-1$.
Assume that
\begin{enumerate}
\item[(1)] the spaces $X_i$ are diffeomorphic to a fixed space $X$, 
\item[(2)] there exists a constant $v>0$, independent of $i$, such that $\text{Vol}(B_r(x,g_i))>v\cdot r^4$ for each $x\in X$ and $0<r \leq 1$,
\item[(3)] there exists $R>0$, 
such that $\int_{X_i\setminus B_R(x_i,g_i)}\|Rm(g_i)\|_{C^0}^2 dV_{g_i}<\epsilon_0/2$, where $B_R(x_i,g_i)$ is a geodesic ball with respect to the metric $g_i$.
\end{enumerate}
Then 
up to a subsequence, $(X_i,J_i,g_i,x_i)$ converges 
to an ALE SFK  orbifold $(X_\infty,J_\infty,g_\infty,x_\infty)$
in the sense of pointed Cheeger-Gromov convergence with a uniform ALE asymptotic rate of order $O(r^{-\mu})$.
\end{lemma}
\begin{proof}
For convenience, in the following of the proof, $C$ is denoted as a positive constant
with value that may vary line by line.
If $C$ depends on the subscript $i$ (index of the sequence) (or the superscript $k$ (degree of regularity)),
we will specify it as $C = C(i)$  (or $C(k)$).

The Hirzebruch signature theorem for an ALE SFK metric states that,  
\begin{align}
\label{taum2}
\tau(X)= - \frac{1}{12\pi^2}\int_{X} \Vert W^- \Vert^2dV_{g}  + \eta(S^3/\Gamma),
\end{align}
and the Chern-Gauss-Bonnet theorem in this setting  \cite{Hitchin, Nakajima} states that
\begin{align}
\label{chim2}
\chi(X)= \frac{1}{8\pi^2}\int_{X}\Big( \Vert W^- \Vert^2-\frac{1}{2} |E|^2\Big)dV_{g}  + \frac{1}{|\Gamma|} \Big).
\end{align}
Consequently, if the group $\Gamma$ is fixed, and all of the spaces are diffeomorphic, then there exists a constant $C$ so that 
\begin{align}
\int_{X_i} \Vert Rm \Vert_{g_i}^2 dV_i \leq C. 
\end{align}
By \eqref{BKN}, there exists an ALE coordinate
$H_i: {X_i\setminus B_R(x_i)}\rightarrow \RR^4/\Gamma$, such that 
\begin{align}
\label{harmonic_coordinate_decay_rate}
|\partial^{(k)}({H_i}_*g_i-g_{Euc})| < C(k)\cdot r^{-\nu-k},
\end{align}
where we can choose $-\nu$ between $-2<-\nu<-\mu$.
By our assumption of lower volume growth, by \cite[Theorem 1.1]{TV2} and 
\cite[Theorem~1.3]{TV3}, 
up to a subsequence, $(X_i,g_i,x_i)$ converges to $(X_\infty,g_\infty,x_\infty)$
in the pointed Gromov-Hausdorff sense. Since $J_i$ is parallel with respect to $g_i$, 
it is easy to see that there is a limiting complex structure $J_{\infty}$. 
Moreover, using \cite[Theorem~6.1]{TV2}, the limit $(X_\infty,J_{\infty}, g_\infty,x_\infty)$
is an ALE SFK orbifold. Without loss of generality, assume $x_\infty$ is the only energy concentration point in $X_\infty$.
Then for any $\delta>0$, $R>\delta>0$, there exists a diffeomorphism
\begin{align}
\psi'_i: A_{\delta,2R}(x_\infty)\rightarrow X_i
\end{align}
such that ${\psi'_i}^*g_i\xrightarrow{C^\infty} g_\infty$.
For a $R$ large enough (with its specific value to be determined later), there exists an ALE coordinate
\begin{align}
\pi: {X_\infty\setminus B_R(x_\infty)}\rightarrow \RR^4/\Gamma
\end{align}
such that $|\partial^{(k)}(\pi_* g_\infty-g_{Euc})| < C(k)\cdot r^{-\nu-k}$, where $r$ is the Euclidean distance to the origin.
Since on $A_{\delta,2R}(x_\infty)$, ${\psi'_i}^*g_i$ converges to $g_\infty$ smoothly, for each $\epsilon'>0$,
by choosing $R$ large enough, and when $i$ is sufficiently large, 
\begin{align}
H_i\circ{\psi'_i}\circ\pi^{-1} = A_i + Q_i
\end{align}
where $A_i$ is induced from a subgroup of $SO(4)$ acting on the universal cover of $\RR^4/\Gamma$, $|Q_i|<\epsilon'$. 
Since $A_i$ is induced from a subgroup of $SO(4)$,
$A_i^{-1}\circ H_i: {X_i\setminus B_R(x_i)}\rightarrow \RR^4/\Gamma$ is still an ALE coordinate with the same 
asymptotic rate.
Then we can extend ${\psi'_i}$ to a diffeomorphism ${\psi_i}$ from 
${X_\infty\setminus B_\delta(x_\infty)}$ to ${X_i\setminus B_\delta(x_i)}$ by defining
\begin{align}
\label{CG_diffeomorphism}
{\psi_i} = 
\begin{cases}
{\psi'_i}  &\text{ on $A_{\delta,R}(x_\infty)$} \\
H_i^{-1}\circ A_i\circ \pi &\text{ on $X_\infty\setminus B_{2R}(x_\infty)$}\\
(1-\chi(\frac{r(x)}{R})){\psi'_i} + \chi(\frac{r(x)}{R}) H_i^{-1}\circ A_i\circ \pi &\text{ on $A_{R,2R}(x_\infty)$},
\end{cases}
\end{align}
where $\chi:\RR_{\geq 0}\rightarrow \RR_{\geq 0}$ is a non-decreasing smooth function,
$\chi(t) = 0$ if $t\leq 1$, $\chi(t) = 1$ if $t\geq 2$, $r(\cdot)$ is the distance to $x_\infty$ with respect to 
the metric $g_\infty$.
Since $-\nu<-\mu$, for any $\epsilon'>0$, we can fix a constant $R>0$ large enough, such that, when
$i$ is sufficiently large, 
$\|\psi_i^*g_i - g_\infty\|_{C^{k,\alpha}_{-\mu}(X_\infty(B_\delta(x_\infty)))} < \epsilon'$.
The convergence of the complex structure follows from the convergence of the Riemannian metric, using the same argument as in \eqref{HLJ}. 
\end{proof}

\subsection{Bubble trees} The degeneration of convergence at ``energy concentration points'' can be understood through a process called ``bubbling''. 
The sequence $(X_i,g_i,x_i)$ in Lemma \ref{uniform_diffeom} converges to an orbifold limit $(X_\infty,g_\infty,x_\infty)$.
By studying different scales of convergence toward the energy concentration point $x_\infty$, there is a ``bubble tree''
structure which captures the topological information that ``disappears'' in the orbifold limit.

At any energy concentration point,  we choose the smallest fixed $\delta>0$, and $r_i\to 0$, such that in $B_\delta(x_i)$
\begin{align}
\int_{B_\delta(x_i)\setminus B_{r_i}(x_i)}\|Rm(g_i)\|^2 dV_g= \frac{\epsilon_0}{2}.
\end{align}
The rescaled sequence 
\begin{align}
(Y_i, g'_i, y_i) = \Big(B_{\delta}(x_i),\frac{1}{r_i^2}g_i,x_i \Big)
\end{align} 
then converges to an ALE orbifold limit $(Y_\infty,g'_\infty,y_\infty)$
in the pointed Cheeger-Gromov sense, where the limit is called the ``first bubble''.
For any energy concentration point $p\in Y_\infty$ in the rescaled limit, 
there exists a sequence of points $p_i\in Y_i$ that converges to $p$,
and high curvature regions $B_{\delta'}(p_i)\subset Y_i$ for some $\delta'>0$, and $r'_i\to 0$, such that
\begin{align}
\int_{B_{\delta'}(p_i)\setminus B_{r'_i}(p_i)} \|Rm(g'_i)\|^2 dV_{g'_i} = \frac{\epsilon_0}{2},
\end{align}
and the rescaled sequence
\begin{align}
(Z_i, g''_i, z_i) = \Big(B_{\delta'}(p_i),\frac{1}{{r'_i}^2}g'_i, p_i \Big)
\end{align}
converges to an ALE orbifold $(Z_\infty,g''_\infty,z_\infty)$. The limit
$(Z_\infty,g''_\infty,z_\infty)$ is called a ``deeper bubble'' to the previous bubble $(Y_\infty,g'_\infty,y_\infty)$.
Iteratively, for each energy concentration point in a bubble, we can consider the rescaled limit (by energy scale)
and obtain an ALE orbifold limit as a deeper bubble.
Since the total energy is finite and each deeper bubble loses a definite amount of energy, there are at most finite iteration steps.
The smooth bubbles with no energy concentration points are called the ``deepest bubbles''.
By gluing each deeper bubble to the corresponding singularity in the previous bubble, we obtain a topological space
which is called the ``bubble tree''. The bubble tree is homeomorphic to $B_\delta(x_i)$ for $i$ sufficiently large. We refer the reader to \cite{Bando} for a more detailed description of the bubbling process in the Einstein case, and \cite{TV2} for the SFK case. 

If the bubble tree
has only 1 branch, then the original manifold $X_i$ for $i$ sufficiently large 
is diffeomorphic to $X_{\infty} \# Y_1 \# Y_2 \#\ldots \# Y_r$, where
$Y_1$ is the first bubble, and $Y_r$ is the deepest bubble.
The notation $\#$ stands for a generalized connected sum, which is obtained by attaching the boundary of a truncated ALE space onto the boundary of a punctured neighborhood of an orbifold point.  By the Mayer-Vietoris sequence, it follows that 
\begin{align}
\label{b2sum}
b_2(X_i) = b_2(X_{\infty}) + \sum_{i = 1}^r b_2(Y_{i}).
\end{align}
A similar formula holds in the case of several branches.

 In general, there can be energy concentration points which are smooth points of the limit space. In this case, the first bubble will be an asymptotically flat (AE) orbifold, i.e., an ALE space with $\Gamma = \{e\}$. 
While these types of bubbles can certainly appear in general, one can rule out such bubbles which are topologically trivial.  
\begin{lemma}
\label{tblemma}
 If $(X,g,J)$ is a AE SFK orbifold with $b_2(X) = 0$, then $(X,J)$ is biholomorphic to $\CC^2$ and $g$ is the flat metric. 
\end{lemma}
\begin{proof}
Consider the minimal resolution of $(\widetilde{X}, \widetilde{J})$ of $(X,J)$. 
By a basic local gluing argument on the level of K\"ahler potentials (see \cite{Arezzo06} and also \cite{ALM})
we can glue on Lock-Viaclovsky ALE metrics (see \cite{LV14}) on resolutions at the orbifold points to show that this resolution admits an ALE K\"ahler metric. By \cite[Proposition 4.3]{HL15}, $(\widetilde{X}, \widetilde{J})$ is biholomorphic to $\CC^2$ blown-up at finitely many points. 
Since $b_2(X) = 0$, this implies that $X$ is obtained from $\widetilde{X}$ by blowing down all possible holomorphic curves, 
and is therefore biholomorphic to $\CC^2$.  The Hirzebruch signature theorem for 
an AE SFK metric states that,  
\begin{align}
\label{taum3}
\tau(X)= - \frac{1}{12\pi^2}\int_{X} \Vert W^- \Vert^2dV_{g},
\end{align}
since $\tau(\CC^2) = 0$, this implies that $W^{-} \equiv 0$. The Chern-Gauss-Bonnet theorem in this setting  states that
\begin{align}
\label{chim3}
\chi(X)= \frac{1}{8\pi^2}\int_{X}\Big( \Vert W^- \Vert^2-\frac{1}{2} |E|^2\Big)dV_{g}  + 1,
\end{align}
and since $\chi(\CC^2) = 1$, this implies that $E \equiv 0$, and consequently $g$ is flat. 
\end{proof}

\section{Compactness I. Convergence of birational structure}
\label{uniform}
In this section, we will investigate more closely the pointed Cheeger-Gromov convergence of the sequence of metrics in Theorem~\ref{compactness}. 
By results of Tian-Viaclovsky discussed above in Section~\ref{CGsubsec},  a subsequence converges to an ALE SFK metric. The main issue here is there could be a ``jump'' of complex 
structure at the limit, or a ``jump'' of birational type of the limit, even if every metric in the sequence is biholomorphic.  
For example, if we rescale down an ALE SFK metric on a Stein surface $X$ by $r_i^2\cdot g$, $r_i\to 0$, 
the pointed Cheeger-Gromov limit is the flat cone $\CC^2/\Gamma$. This limit is not birationally equivalent to $X$ since $X$ is Stein and smooth.
However, note that in the setting of Theorem~\ref{compactness} with fixed complex structure 
and varying K\"ahler classes, such rescaling is excluded.  Note also that as of yet,  we do not know that the convergence is uniform at infinity,
which is what we will prove next (we do not even know yet that the group at infinity of the limit is the same for the limit as for the sequence).

Let $\Psi:X\setminus K\rightarrow \RR^4/\Gamma$ be an ALE coordinate of order $O(r^{-\mu})$ for $(X,J,g_0,x_0)$,
where $-2<-\mu<-1$. Recall as discussed in Section~\ref{versal} above, 
there exist holomorphic functions
$u^1,\ldots,u^N$ satisfying certain polynomial relations that determine the birational type of $(X,J)$.
To prove the convergence of the birational structure, we will need to show convergence of $u^j$ in a strong sense 
after the uniform Cheeger-Gromov diffeomorphism is applied. 
\begin{proposition}
\label{uniform_rate}
Let $(X,J,g_i)$ be the sequence of ALE SFK metrics as in Theorem \ref{compactness} with group $\Gamma$ at infinity. Then 
there exist base points $x_i\in X$ such that the following holds:
\begin{enumerate}
\item[(1)]
Up to a subsequence, $(X,J,g_i,x_i)$ 
pointed Cheeger-Gromov converges with a uniform ALE asymptotic rate of order $O(r^{-\mu})$
to an ALE SFK  orbifold $(X_\infty,J_\infty,g_\infty,x_\infty)$. In particular, the group at infinity of the limit is also $\Gamma$. 
\item[(2)]
The limit space $X_\infty$ is birationally equivalent to $X$.
\item[(3)]
There exists a constant $R>0$, such that all holomorphic
curves are contained in geodesic ball $B_R(x_i,g_i)$ when $i$ is sufficiently large.
\end{enumerate}
\end{proposition}
\begin{proof}
By the convergence results discussed in Section~\ref{CGsubsec} above, 
for any sequence of basepoints $x_i \in X$, there exists a 
pointed Cheeger-Gromov limit 
\begin{align}
(X,J,g_i, x_i) \rightarrow 
(X_{\infty}, J_{\infty}, g_{\infty}, x_{\infty}).
\end{align} 
Without loss of generality, we can assume that $x_\infty$ is the only energy concentration point in the limit $X_\infty$, and that $x_i$ is chosen so that 
 $\sup_{X}(\|Rm(g_i)\|_{g_i})$ is obtained at $x_i$.

First, let us assume the sequence $\{(X,J,g_i,x_i)\}$ 
has a uniform ALE energy bound, i.e., that the assumption (3) in Lemma~\ref{uniform_diffeom} is satisfied. 
(We will then prove below that this assumption is necessarily satisfied). 
Under this assumption, by Lemma~\ref{uniform_diffeom}, there exist diffeomorphisms
\begin{align}
\psi_i: X_\infty\setminus B_\delta(x_\infty)\rightarrow X
\end{align}
such that, $\|\psi_i^*g_i-g_\infty\|_{C^{k,\alpha}_{-\mu}(g_\infty)}< \epsilon(i | k,\delta)$,
$\|\psi_i^*J-J_\infty\|_{C^{k,\alpha}_{-\mu}(g_\infty)}<\epsilon(i | k,\delta)$,
for $-2<-\mu<-1$.

Under this assumption, we next analyze the birational structure of the limit space.
Recall that, for each $(X,J,g_i,x_i)$, there exists a harmonic coordinate
$H_i: X\setminus B_R(x_i)\rightarrow \RR^4/\Gamma$, under which ${H_i}_*g_i,{H_i}_*J$ are asymptotic
to $g_{Euc},J_{Euc}$ uniformly of rate $O(r^{-\mu})$.
In the following, we will fix an $R>0$, and consider $({H_i}_*g_i,{H_i}_*J)$ on the fixed space
$A_{R,\infty}(0)\subset \RR^4/\Gamma$. Furthermore, all the norms used in the following are over the space
$A_{R,\infty}(0)$.

Recall the construction in Section~\ref{versal}.
Since $(X,J)$ is a K\"ahler surface with an ALE coordinate $\Psi$, $X$ can be compactified analytically to $\hat{X}$,
and there exist holomorphic functions $u^1,\ldots,u^N$ that determines the birational structure of $X$, obtained from
holomorphic sections on $\hat{X}$.
Define the degree of a function $f$ on $X$ with respect to the coordinate $\Psi$ as
\begin{align}
d_{\Psi}(f) = \lim_{r\to\infty} \left(\frac{\log(\sup_{p\in S_r}|f(\Psi^{-1}(p))|)}{\log(r)}\right) 
\end{align}
where $r$ is the $g_{Euc}$-radius and $S_r$ is the $r$-sphere centered at $\{0\}$ in $A_{R,\infty}(0)$.
For each $u^j$ above, $d_{\Psi}(u^j)$ is finite.
Then we can rearrange $u^1,\ldots,u^N$ in the increasing order of $d_{\Psi}$, and we have positive integers $d_1,\ldots,d_l$,
such that there are $n_j$-th many elements among $u^1,\ldots,u^N$ that have degree of $d_j$, and $d_j<d_{j+1}$,
$\sum_{j=1}^l n_j = N$.
Define $\H$ as the $\CC$-algebra of all holomorphic function on $X$ of finite $d_{\Psi}$-degree.
We can assume $\{u^1,\ldots,u^N\}$ is a minimal set of generators of $\H$.

In a similar fashion,  we can define $d_{H_\infty}$ for holomorphic functions on $X_\infty$, with 
respect to the ALE coordinate $H_{\infty}$ on $X_{\infty}$. 
There exist 
holomorphic functions $u^1_\infty,\ldots, u^{N}_\infty$
on $X_\infty$, which comprises a minimum set of generators of
the $\CC$-algebra of holomorphic functions on $X_\infty$ of finite $d_{H_\infty}$-degree.

We claim that
$l_\infty = l, n^\infty_j = n_j, d^\infty_j = d_j$.
This follows by constructing the deformation to the normal cone for both $X$ and $X_\infty$ as described above in \eqref{defnc}. 
The line bundle $L$ is deformed to $\underline{L}$ along the deformation as $t \to 0$. 
Since $H^0(D,\underline{L}^k)\simeq H^0(D,L^k))$, there exists $\underline{s_j}\in H^0(C(D),\underline{L}^k)$ that corresponds with $s_j$.
The normal cone $C(D)$ admits a flat conical metric $g_C$, so we can define the degree $\underline{d}_j$ for each $\underline{u}^j = \frac{\underline{s}_j}{\underline{s}_0}$
in a similar way. The metric cone $(C(D),g_C)$ 
is the tangent cone at infinity of $(X,g_0)$
and $\underline{u}^j$ is the scale-down limit of $u^j$,
so it follows that $\underline{d}_j = d_j$, and consequently $\underline{n}_j = n_j, \underline{l} = l$.
Applying the same argument to $(X_\infty, J_{\infty})$ proves the claim. 

Next, we study the convergence of the generating holomorphic functions.
Let $u$ be a holomorphic function on $X$ with $d_{\Psi}(u) = d_1$, which is the lowest degree of a non-constant holomorphic function.
Since $\|{H_i}_*u\|_{C^{0,\alpha}_{d_1}({H_i}_*g_i)}$ is finite,
there exists a sequence of positive constants $c_i$, such that on $A_{R,\infty}(0)$,
$\|{H_i}_*(c_i u)\|_{C^{0,\alpha}_{d_1}({H_i}_*g_i)} = 1$.
Up to a subsequence, ${H_i}_*(c_i u)$ pointwise converges to a limit function $w$,
because on any annulus $A_{R,2^{k+1}R}(0)$, the usual H\"older norm is uniformly bounded.
We will next use elliptic theory to refine the convergence. 

Choose $\Delta_{{H_\infty}_*g_\infty}$-harmonic functions $h_1,\ldots,h_m$ of $d_{H_\infty}$-degree $d_1$,
such that for any function which is $\Delta_{{H_\infty}_*g_\infty}$-harmonic and of $d_{H_\infty}$-degree $d_1$, its leading term
can be represented as a linear combination of $h_1,\ldots,h_m$.
Since
${H_i}_*g_i$ converges to ${H_\infty}_*g_\infty$ in any $C^{k,\alpha}_{-\mu}$-norm, 
for any $C^2$ function $f$, we have the pointwise bound
\begin{align}
\label{laplace_compare1}
|(\Delta_{{H_\infty}_*g_\infty}-\Delta_{{H_i}_*g_i}) f| &< \epsilon(i)(r^{-\mu}\cdot |\nabla^2_{{H_\infty}_*g_\infty}f|+
r^{-\mu-1}\cdot |\nabla_{{H_\infty}_*g_\infty}f|),
\end{align}
where $\epsilon(i) \to 0$ as $i \to \infty$,
and for any function $f$ with bounded $C^{2,\alpha}_{\nu}({H_\infty}_*g_\infty)$-norm, we have
\begin{align}
\label{laplace_compare2}
\|(\Delta_{{H_\infty}_*g_\infty}-\Delta_{{H_i}_*g_i}) f\|_{C^{0,\alpha}_{\nu-\mu -2}({H_\infty}_*g_\infty)} &< \epsilon(i\,|\,\nu)\cdot \|f\|_{C^{2,\alpha}_{\nu}({H_\infty}_*g_\infty)}.
\end{align}
where $\epsilon(i\,|\,\nu)\to 0$ as $i\to \infty$ for each fixed weight $\nu$. 
By the classical elliptic estimate in weighted norms (see \cite{Bartnik}), 
we have
\begin{align}
\label{f_estimate}
\|f \|_{C^{2,\alpha}_{d_1}({H_\infty}_*g_\infty)} < C\cdot 
(\|f\|_{C^{0,\alpha}_{d_1}({H_\infty}_*g_\infty)} + \|\Delta_{{H_\infty}_*g_\infty}f\|_{C^{0,\alpha}_{d_1 -2}({H_\infty}_*g_\infty)}).
\end{align}
Since ${H_i}_*(c_i u)$ is $\Delta_{{H_i}_*g_i}$-harmonic and
$\|{H_i}_*(c_i u)\|_{C^{0,\alpha}_{d_1}({H_\infty}_*g_\infty)}$ is uniformly bounded,
the above estimates imply 
$\|{H_i}_*(c_i u)\|_{C^{2,\alpha}_{d_1}({H_\infty}_*g_\infty)}<C$ for some uniform $C>0$. 
In particular, by estimate \eqref{laplace_compare2}, and the invertibility 
of the Laplacian on the complement of a ball, 
there exists a function $\xi_i\in C^{2,\alpha}_{d_1-\mu}({H_\infty}_*g_\infty)$,
such that $\Delta_{{H_\infty}_*g_\infty}\xi_i = \Delta_{{H_\infty}_*g_\infty}({H_i}_*(c_iu))$ and
\begin{align}
\label{xi_estimate}
\begin{split}
\|\xi_i\|_{C^{2,\alpha}_{d_1-\mu}({H_\infty}_*g_\infty)} < 
C\cdot \|\Delta_{{H_\infty}_*g_\infty} {H_i}_*(c_i u)\|_{C^{0,\alpha}_{d_1-\mu -2}({{H_\infty}_*g_\infty})} 
< C\cdot \epsilon(i\,|\,d_1).
\end{split}
\end{align}
By existence of harmonic expansions, we have the decomposition 
\begin{align}
{H_i}_*(c_i u) = \xi_i + \sum_{j=1}^{m} a_{i,j} h_{j} + v_i
\end{align}
for some functions $v_i$ on $A_{R,\infty}(0)$. Then by the estimate of ${H_i}_*(c_i u)$ above and 
\begin{align}
\|\xi_i\|_{C^{2,\alpha}_{d_1}({H_\infty}_*g_\infty)} < C\cdot \|\xi_i\|_{C^{2,\alpha}_{d_1-\mu}({H_\infty}_*g_\infty)} < C\cdot \epsilon(i\,|\,d_1),
\end{align}
we have 
$\sum_{j=1}^m |a_{i,j}| <C$ for some constant $C>0$, 
and there exists finite limit $a_j = \lim_{i\to\infty} a_{i,j}$ for each $1\leq j\leq m$.
Furthermore, 
$v_i$ is a $\Delta_{{H_\infty}_* g_\infty}$-harmonic function with degree $d_{H_\infty}(v_i)<d_1$.
By the elliptic estimate \eqref{f_estimate}, for $0<\epsilon'<1$, we have $\|v_i\|_{C^{2,\alpha}_{d_1-\epsilon'}}<C$ for a uniform $C>0$.
Since $C^{2,\alpha}_{d_1-\epsilon'}({H_\infty}_*g_\infty)$ is compactly embedded into $C^{0,\alpha}_{d_1}({H_\infty}_*g_\infty)$,
we have $v_i$ converges strongly in $C^{0,\alpha}_{d_1}({H_\infty}_*g_\infty)$-norm on $A_{R,\infty}(0)$.
Then by the analysis above, $v_i, \sum_{j=1}^m a_{i,j}h_j, \xi_i$ converge strongly in $C^{0,\alpha}_{d_1}({H_\infty}_*g_\infty)$-norm on $A_{R,\infty}(0)$
as $i\to\infty$.
This implies that
${H_i}_*(c_i u)$ converges to a limiting function  $w$ strongly in $C^{0,\alpha}_{d_1}({H_\infty}_*g_\infty)$-norm, which satisfies
\begin{align}
1-\epsilon < \|w\|_{C^{0,\alpha}_{d_1}({H_\infty}_*g_\infty)}<1+\epsilon
\end{align}
for some small $\epsilon>0$.
By the convergence of the metric and the complex structure, we also have
$w$ is $\Delta_{{H_\infty}_*g_\infty}$-harmonic and ${H_\infty}_*J_\infty$-holomorphic on $A_{R,\infty}(0)$.
Since $X_\infty$ is a one-convex space, $w$ can be extended to a holomorphic function on $X_\infty$.
Recall that $u$ is a non-constant holomorphic function of finite degree on $X$, and the zero locus of 
${H_i}_*(c_i u)$ is a ${H_i}_* J$-analytic subset which intersects with any annulus $A_{r,2r}(0)$ non-trivially for $r$ large enough.
This implies that $\inf_{A_{r,2r}(0)}|w| = 0$. Since $\|w\|_{C^{0,\alpha}_{d_1}({H_\infty}_*g_\infty)}>1-\epsilon$,
we have $w$ is a non-constant ${H_\infty}_*J_\infty$-holomorphic 
function on $A_{R,\infty}(0)$. 
Since $\|w\|_{C^{0}_{d_1}({H_\infty}_*g_\infty)}$ is bounded, and $d_1$ is the lowest possible $d_{H_\infty}$-degree for a non-constant holomorphic function,
we have $d_{H_\infty}(w) = d_1$ and $\sum_{j=1}^m |a_{j}|>0$.

Next, we want to show that there exists some positive constant $C>0$, such that $\frac{1}{C}< |c_i| < C$ for $i$ sufficiently large.
By the convergence of $\xi_i, \sum_{j=1}^m a_{i,j}h_j,v_i$ as above, 
the $d_1$-degree term of ${H_i}_*(c_i u)-w-\xi_i$ can be represented as 
$\sum_{j=1}^m b_{i,j}h_{j}$, where $b_{i,j}\to 0$ as $i\to\infty$.
Then for $i$ sufficiently large, the $d_1$-degree term of ${H_i}_*(c_i u)$ ``approximately'' equals to the $d_1$-degree term of $w$.
Define the ``growth ratio'' for any ${H_\infty}_*J_\infty$-holomorphic function $h$ on $A_{R,\infty}(0)$ with $d_{H_\infty}(h)=d$ by
\begin{align}
I_{H_\infty}(h) = \lim_{r\to\infty}\Big( \sup_{p\in S_r}\frac{|h(p)|}{r^{d}}\Big).
\end{align}
It is not hard to see that $I_{H_\infty}(w)$ is well-defined and $0 <I_{H_\infty}(w) <\infty$ unless $w$ is trivial.
Similarly, we can define $I_{H_i}$ and $I_{\Psi}$ for ${H_i}_*J$-holomorphic functions and ${\Psi}_*J$-holomorphic functions with respect to the corresponding
coordinates.
By the approximation above, $I_{H_\infty}(w) \approx |c_i|\cdot I_{H_\infty}({H_i}_* u)$.
Since $g_i$ is an ALE K\"ahler metric over both the $\Psi$ and $H_i$ coordinates, by \cite[Corollary 3.2]{Bartnik},
\begin{align}
\overline{H_i} = A_i \cdot \overline{\Psi} + \text{ lower order term}
\end{align}
where $\overline{H_i}, \overline{\Psi}$ are the universal covers of the coordinates, and $A_i\in U(2)$. 
It follows that $I_{\Psi}(\Psi_* u) = I_{H_i}({H_i}_*u)$,
and since the harmonic coordinate $H_i$ converges to $H_\infty$, we also have
$I_{H_i}({H_i}_*u) = I_{H_\infty}({H_i}_* u)$.
Then we have $I_{H_\infty}(w) \approx |c_i| I_{\Psi}(\Psi_* u)$.
Since $0< I_{\Psi}(\Psi_* u)<\infty$, there exists a constant $C>0$, such that for $i$ sufficiently large,
$\frac{1}{C}< |c_i| < C$.

As a result, without loss of generality, we can assume $c_i=1$, and up to a subsequence, ${H_i}_* u$ converges to a ${H_\infty}_*J_\infty$-holomorphic function $w$
strongly in $C^{0,\alpha}_{d_1}(A_{R,\infty}(0),{H_\infty}_*g_\infty)$-norm,
and $d_{H_\infty}(w) = d_{\Psi}(u) = d_1$.
Then for generators $u^1,\ldots,u^{n_1}$ of holomorphic functions with $d_{\Psi}$-degree $d_1$, up to a subsequence, the functions 
 ${H_{i}}_* u^1,\ldots, {H_{i}}_* u^{n_1}$
converge to 
${H_\infty}_*J_\infty$-holomorphic functions $w^1,\ldots,w^{n_1}$ 
of $d_{H_\infty}$-degree $d_1$.We claim that $w^1,\ldots,w^{n_1}$ 
are $\CC$-linear independent,
and are therefore generators of ${H_\infty}_*J_\infty$-holomorphic functions of degree $d_1$.
To see this, if there was any linear relation $\sum_{j=1}^m c_j w^j = 0$, 
then for $i$ sufficiently large, 
\begin{align}
\sum_{j=1}^m c_j {H_{i}}_* u^j = {H_i}_* \Big( \sum_{j=1}^m c_j u^j \Big)
\end{align}
would be very small pointwise for all $r$ sufficiently large, 
which is a contradiction to the linear independence of $u^1, \ldots, u^{n_1}$.

Next, let $u$ be a holomorphic function on $X$ with $d_{\Psi}(u) = d_2$.
Without loss of generality, we can assume $u\not\in \CC[u^1,\ldots,u^{n_1}]$.
There is a sequence of constants $c_i>0$ such that on $A_{R,\infty}(0)$, $\|{H_i}_*(c_i u)\|_{C^{0,\alpha}_{d_2}({H_i}_*g_i)} = 1$.
A similar argument to the $d_1$-degree case shows that 
${H_i}_*(c_i u)$ converges to a limit function $w$ strongly in $C^{0,\alpha}_{d_2}({H_\infty}_*g_\infty)$-norm.
Then 
\begin{align}
\label{1mew}
1-\epsilon < \|w\|_{C^{0,\alpha}_{d_2}({H_\infty}_*g_\infty)} < 1+\epsilon
\end{align} 
for some small $\epsilon>0$, which clearly implies that $d_{H_\infty}(w) \leq d_2$.
We claim that $d_{H_\infty}(w) = d_2$. To see this, assume by contradiction 
that $d_{H_\infty}(w) < d_2$.
Since any holomorphic function of $d_{H_\infty}$-degree smaller than $d_2$ is generated by holomorphic functions of $d_{H_\infty}$-degree $d_1$,
there exists a polynomial $F$, such that $w = F(w^1,\ldots, w^{n_1})$, where $w^1,\ldots, w^{n_1}$ are holomorphic functions 
of degree $d_1$ and each $w^j$ is the limit of the sequence ${H_i}_* u^j$ as proved above.
Then we have
\begin{align}
\|c_i\cdot {H_i}_*u - F({H_i}_*u^1,\ldots,{H_i}_*u^{n_1})\|_{C^{0,\alpha}_{d_2}({H_\infty}_*g_\infty)}\to 0
\end{align}
as $i\to \infty$ on $A_{R,\infty}(0)$.
Let $V(u)$ be the zero locus of $u$ on $X$, which is an analytic closed subset and not contained in any compact subset. 
Since $u\not\in \CC[u^1,\ldots,u^{n_1}]$, 
for some small $\epsilon'>0$, the set 
$S = \{x\in V(u): |F(u^1,\ldots,u^{n_1})|>\epsilon'\}$ is non-trivial and not contained in any compact subset.
For a fixed annulus $A_{r,2r}(0)\subset A_{R,\infty}(0)$,  there exists a 
sequence of points $p_i\in H_i(S)\cap A_{r,2r}(0)$, and $p_i\to p_\infty\in A_{r,2r}(0)$.
Then $|{H_i}_*(c_i u-F(u^1,\ldots,u^{n_1}))|(p)>\frac{\epsilon'}{2}$, which contradicts with
$\|{H_i}_*(c_i u - F(u^1,\ldots,u^{n_1}))\|_{C^{0,\alpha}_{d_2}({H_\infty}_*g_\infty)} = 0$ on $A_{R,\infty}(0)$. This contradiction proves that $d_{H_\infty}(w) = d_2$.

Similarly to the degree $d_1$ case above, by analyzing the $d_2$-degree term of $c_i\cdot {H_i}_* u$ and $w$, it follows that there exists a constant $C>0$
such that for $i$ sufficiently large, $\frac{1}{C}< |c_i|< C$.
Without loss of generality we can assume that $c_i=1$, and up to subsequence, ${H_i}_* u$ converges to a holomorphic function $w$ of degree $d_2$.
Then for $u^{n_1+1},\ldots,u^{n_1+n_2}$, which are generators of holomorphic functions of degree $d_2$ on $X$, up to a subsequence, 
${H_i}_* u^{n_1+1},\ldots,{H_i}_* u^{n_1+n_2}$ converge to holomorphic
functions $w^{n_1+1},\ldots, w^{n_1+n_2}$, which are generators of ${H_\infty}_*J_\infty$-holomorphic functions of degree $d_2$ on $A_{R,\infty}(0)$.

By an inductive procedure, the above arguments prove that, 
up to a subsequence, the functions  ${H_i}_*u^1,\ldots,{H_i}_*u^N$ converge to ${H_\infty}_*J_\infty$-holomorphic functions $w^1,\ldots,w^N$ of the corresponding
degrees. Note that for any polynomial relation $F(u^1,\ldots,u^N)=0$, by the convergence of $u^j$, $F(w^1,\ldots,w^N)=0$.
Each $w^j$ can be pulled-back to $X_\infty\setminus B_R(x_\infty)$ and extends to a holomorphic function on the one-convex space $X_\infty$, which
is still denoted as $w^j$.

Define
\begin{align}
R(X) &= \CC[u^1,\ldots,u^N] \simeq \CC[x^1,\ldots,x^N]/ \mathcal{I} \\ 
R(X_\infty) &= \CC[w^1,\ldots,w^N] \simeq \CC[x^1,\ldots,x^N]/\mathcal{I}_\infty,
\end{align}
where $\CC[x^1,\ldots,x^N]$ is the coordinate ring of $\CC^N$. 
By the paragraph above, $\mathcal{I}\subset \mathcal{I}_\infty$,
so there exists a well-defined ring homomorphism from $R(X)$ to $R(X_\infty)$ by mapping each $u^j$ to $w^j$. 
We claim that this ring homomorphism is an isomorphism. To see this, 
assume that $w^1,\ldots,w^N$ satisfy a polynomial relation $F(w^1,\ldots,w^N) = 0$.
Consider the function $F = F(u^1,\ldots,u^N)$, which is 
a holomorphic function on $X$. If $F$ is not identically zero, then 
let $d_{\Psi}(F) = d_F\geq 0$. 
By the strong convergence of $u^j$ proved above, ${H_i}_*F$ converges to a ${H_\infty}_*J_\infty$-holomorphic function $G$ on $A_{R,\infty}(0)$
in $C^{0,\alpha}_{d_F}({H_\infty}_*g_\infty)$-norm.
Since $I_{H_i}({H_i}_*F) = c>0$ is a positive constant, by the $C^{0,\alpha}_{d_F}$-convergence, we have 
$\|G\|_{C^{0,\alpha}_{d_F}({H_\infty}_*g_\infty)}>0$.
However, by the convergence of $u^j$, 
$G = F(w^1,\ldots,w^N) = 0$, which is a contradiction. 
Therefore $u^1,\ldots,u^N$ satisfies the same polynomial 
relation $F$ and $R(X)$ is isomorphic to $R(X_\infty)$.
Since the affine space $Spec(R(X))$ is isomorphic to the image of $X$ in $\CC^N$ under
$u \equiv (u^1, \ldots, u^N)$, the ring isomorphism implies that
$w \equiv (w^1, \ldots, w^N)$ 
embeds $X_\infty\setminus B_R(x_\infty)$ into $\CC^N$ and consequently 
$X_\infty$ is birationally equivalent with $X$.

For the third part of Theorem \ref{uniform_rate}, 
if there exists a holomorphic curve $E$ that is not contained in the geodesic ball $B_R(x_i)$,
then on $E\cap (X\setminus B_R(x_i))$, the holomorphic functions $u^j$ are constant for $1\leq j \leq N$.
However this contradicts with the fact proved above that $u = (u^1,\ldots,u^N)$ embeds $X\setminus B_R(x_i)$ into $\CC^N$.
Thus all holomorphic curves are contained in the geodesic ball $B_R(x_i)$ for each $i$.

To finish the proof of Proposition \ref{uniform_rate}, we need to prove that the assumption (3) in Lemma~\ref{uniform_diffeom} is necessarily satisfied. To prove this, we argue by contradiction. 
Let $r_i$ be the radius such that $\|Rm\|_{L^2(X\setminus B_{r_i}(x_i))}= \frac{\epsilon_0}{2}$.
If  assumption (3) in Lemma~\ref{uniform_diffeom} is not true, then 
$r_i\to\infty$ as $i\to\infty$. 

Consider the rescaled sequence $(X,\frac{1}{r^2_i}g_i, x_i)$. The rescaling preserves the Sobolev constant
and the $L^2$-norm of $Rm$. 
Then by Lemma \ref{uniform_diffeom}, up to a subsequence,
$(X,\frac{1}{r_i^2}g_i,x_i)$ converges to an ALE space $(X'_\infty,g'_\infty,x'_\infty)$
in the sense of pointed Cheeger-Gromov convergence with a uniform ALE asymptotic rate.
In the following, we will first show that the limit space $X'_\infty$ is isomorphic to $\CC^2/\Gamma$ and $g'_\infty$ is a flat metric. Then we will show that 
\begin{align}
\label{L2_continuous}
\|Rm(g'_\infty)\|_{L^2(X'_\infty\setminus B_1(x'_\infty, g_{\infty}'))} = 
\lim_{i \to \infty} \|Rm(\frac{1}{r_i^2}g_i)\|_{L^2(X \setminus B_1(x_i, \frac{1}{r_i^2}g_i))} =  
\frac{\epsilon_0}{2},
\end{align}
which would imply a contraction to flat limit metric.

In order to show that $X'_\infty$ is isomorphic to $\CC^2/\Gamma$,
without loss of generality, we can assume that $x'_\infty$ is the only energy concentration point, since the case of several concentration points is handled by a similar argument.  
Then for each $\delta>0$, there exists a diffeomorphism 
\begin{align}
\psi'_i: X'_\infty\setminus B_{\delta}(x'_\infty)\rightarrow X
\end{align}
such that ${\psi'_i}^*(\frac{1}{r_i^2}g_i)$ converges to $g'_\infty$ smoothly in $X'_\infty\setminus B_\delta(x'_\infty)$.
We also have ${\psi'_i}^*J$ converges to $J'_\infty$ smoothly in $X'_\infty\setminus B_\delta(x'_\infty)$.
Moreover, there exist harmonic coordinates $H'_i$ for $\frac{1}{r_i^2}g_i$, $H'_\infty$ for $g'_\infty$,
and on a fixed annulus $A_{R,\infty}(0)\subset \RR^4/\Gamma$, ${H'_i}_*\frac{1}{r_i^2}g_i$ converges to ${H'_\infty}_*g'_\infty$. 
Consider the rescaled holomorphic functions $r_i^{k_1} u^1,\ldots, r_i^{k_N} u^N$, 
where $k_j = d_{H'_i}(u^j) \leq k_{j+1} = d_{H'_i}(u^{j+1})$.
Note that for the same reason as stated before, $X'_\infty$ has the same spectrum of degrees of holomorphic functions and
each $k_j\in \{d_1,\ldots,d_l\}$.
It is not hard to see that for holomorphic function $u^j$, $I_{H'_i}({H'_i}_*(r_i^{k_j}u^j))$ is a positive constant.
Then following the same argument as used before, we start with the lowest degree $k_1 = d_1$ and
we can show that ${H'_i}_*(c_i r_i^{k_1}u^1)$ converges strongly to a non-zero holomorphic function on $A_{R,\infty}(0)$ 
in $C^{0,\alpha}_{k_1}({H'_\infty}_*g'_\infty)$-norm. 
Then since $I_{H'_i}({H'_i}_*({r_i}^{k_1}u^1))$ is a positive constant and 
$I_{H'_i}({H'_i}_*(c_i r_i^{k_1}u^1))$ converges to a positive limit, there exists a $C>0$ such that $0<\frac{1}{C}< c_i <C $ and we can assume that $c_i=1$.
Then ${H'_i}_*(r_i^{k_1}u^1)$ converges to a holomorphic function ${w'}^1$ of degree $k_1$ on $A_{R,\infty}(0)$, which extends to a holomorphic
function on $X'_\infty$ and will be still denoted by ${w'}^1$.
By a similar iterative argument, we can show that for each holomorphic function $u^j$ of degree $k_j$,
${H'_i}_*(r_i^{k_j}u^j)$ converges to a holomorphic function ${w'}^j$ of $d_{H'_\infty}$-degree $k_j$
in $C^{0,\alpha}_{k_j}({H'_\infty}_*g'_\infty)$-norm.
Let $F(u^1,\ldots,u^N)=0$ be a polynomial relation satisfied by $u^1,\ldots,u^N$.
Denote $F = F'+F''$, where $F'$ is the homogeneous highest-degree term of $F$, and $F''$ is the lower-degree term of $F$.
Then there exist integers $p>p'>0$, such that 
\begin{align}
\begin{split}
0= r_i^{p}F({H'_i}_*u^1,\ldots,{H'_i}_*u^N) &= F'({H'_i}_*(r_i^{k_1}u^1),\ldots, {H'_i}_*(r_i^{k_N}u^N)) \\
&+ r_i^{p'}F''({H'_i}_*(r_i^{k_1}u^1),\ldots,
{H'_i}_*(r_i^{k_N}u^N)).
\end{split}
\end{align}
Letting $i\to\infty$, since $r_i\to 0$, this implies that $F'({w'}^1,\ldots,{w'}^N)=0$.
Next, let $F_1,\ldots,F_m$ be generators of polynomial relations satisfied by $u^1,\ldots,u^N$, and $F'_1,\ldots,F'_m$ be the corresponding
leading terms which are satisfied by $w'_1,\ldots,w'_N$.
Assume $w' \equiv ({w'}^1,\ldots,{w'}^N)$ is not an embedding on $X'_\infty\setminus B_R(x'_\infty)$, where all holomorphic curves contained in
$B_R(x'_\infty)$ for $R$ large enough. 
Then there exists a polynomial relation $P({w'}^1,\ldots,{w'}^N) = 0$ 
but $P({w'}^1,\ldots,{w'}^N)$
is not generated by $\{F'_j({w'}^1,\ldots,{w'}^N)\}_{1\leq j\leq m}$.
Here $P(a_1,\ldots,a_N)$ is a polynomial of degree $q$, where each parameter $a_j$ is a variable of degree $k_j$.
Then by the definition of $F'_j$, $P$ is not the leading term of any polynomial satisfied by $u^1,\ldots,u^N$.
As a result, $P(\Psi_*u^1,\ldots,\Psi_*u^n)$ has non-trivial $d_{\Psi}$-degree $q$ term. 
If not, we have $P'({w'}^1,\ldots,{w'}^N)=0$, and by induction on the lower degree polynomial $P-P'$, it implies that
$P({w'}^1,\ldots,{w'}^N)$ is generated by  $\{F'_j({w'}^1,\ldots,{w'}^N)\}_{1\leq j\leq m}$, which implies a contradiction.
Then we have
\begin{align}
\begin{split}
\inf_{r>R}\sup_{p\in S_r(0)}|&r^{-q}P({H'_i}_*(r_i^{k_1}u^1),\ldots,{H'_i}_*(r_i^{k_N}u^N))| \\
&= \inf_{r>R}\sup_{p\in S_r(0)}|r^{-q}P({\Psi}_*u^1,\ldots,{\Psi}_*u^N)| 
= C>0.
\end{split}
\end{align}
The convergence of ${H'_i}_*(r_i^{k_j} u^j)$ implies the convergence of $P({H'_i}_*(r_i^{k_1}u^1),\ldots,{H'_i}_*(r_i^{k_N}u^N))$
in $C^{0,\alpha}_{q}({H'_\infty}_*g'_\infty)$-norm, which implies that $|P({w'}^1,\ldots,{w'}^N)|>0$ and this gives a contradiction.
Thus $w'$ embeds $X'_\infty\setminus B_R(x'_\infty)$ into $\CC^N$. Since
${w'}^1,\ldots {w'}^N$ satisfy the polynomial relations $F'_1,\ldots,F'_m$, 
$X'_\infty$ is birationally equivalent to $\CC^2/\Gamma$.

For the K\"ahler classes $\kappa_i$ in the statement of Theorem \ref{compactness},
there exists a sequence of smooth ALE K\"ahler background metrics $\omega_{b,i}$, 
where each $\omega_{b,i}\in \kappa_i$, and
$\omega_{b,i}$ converges to a K\"ahler metric $\omega_{b,\infty}\in \kappa_\infty$ smoothly with a uniform ALE asymptotic rate.
Let $W_1,\ldots, W_k$ be smooth $2$-cycles in $X$, and let 
$[W_1],\ldots,[W_k]$ be a basis of $H_2(X,\ZZ)$.
The K\"ahler class of $\omega_i$ can also be parameterized by 
$\int_{W_j}\omega_i$ $(1\leq j\leq k)$.
For the rescaled sequence, as $i\to\infty$,
\begin{align}
\int_{W_j} \frac{1}{r_i^2}\omega_i = \int_{W_j} \frac{1}{r_i^2}\omega_{b,i} \to 0
\end{align}
for each $1\leq j\leq k$.

If $X'_\infty$ is not isomorphic to $\CC^2/\Gamma$, then there exists an effective Weil divisor
$D$ in $X'_\infty$, which may pass through the energy concentration point $x'_\infty$.
Since $D$ is holomorphic, the restriction of $\omega'_\infty$ on $D$ is definite positive, and 
$\int_{D\setminus B_\delta(x'_\infty)}\omega'_\infty>0$.
Let $f:\widetilde{X}_{\infty}'  \rightarrow {X'_\infty}$ be the minimal resolution,
$E'_j$ ($1\leq j\leq r'$) as the exceptional divisors over $x'_\infty$,
and denote $\widetilde{D}$ as the proper transform of $D$. 
Our immediate goal is to find a homology class $[\sigma] \in H_2(\widetilde{X}_{\infty}',\ZZ)$ which is
a nontrivial class in the image of the inclusion map
\begin{align}
\label{istar}
\iota_* :  H_2( \widetilde{X}_{\infty}' \setminus N_{\epsilon}(E') )   \rightarrow  H_2(\widetilde{X}_{\infty}',\ZZ),
\end{align}
where $E' = \cup_{j = 1}^{r'} E'_j$, and $N_{\epsilon}(E)$ denotes a tubular neighborhood of $E$ (with respect to any reference metric), 
which can be identified with a disc bundle in the normal bundle of $E'$, and $\epsilon > 0$ is small.  For simplicity, we can assume that $E'$ is connected and intersects $\tilde{D}$ in a single point, because the following argument will also work in the most general case with minor modifications.  We can assume $D$ is irreducible, so that $\tilde{D}$ is a single rational curve (since we only need to find a single homology class which works). Define the open sets $U = N_{2\epsilon}(E')$, $V =   \widetilde{X}'_{\infty} \setminus N_{\epsilon}(E')$. 
Then $U \cap V$ deformation retracts to $S^3/ \Gamma$ where $\Gamma$ is a finite subgroup of ${\rm{U}}(2)$ 
acting freely on $S^3$. Note that $H_1(S^3/\Gamma) = \Gamma/[\Gamma, \Gamma]$ is a finite abelian group. 
By the universal coefficient theorem, $H^1(S^3/ \Gamma) = Hom( H_1(S^3/\Gamma), \ZZ) = 0$. 
By Poincar\'e duality, $H_2(S^3/\Gamma) = H^1(S^3/\Gamma) = 0$.  Part of the Mayer-Vietoris sequence in singular homology with $\ZZ$-coefficients is then
\begin{equation}
\begin{tikzcd}
\label{cd1}
 0   \arrow[r] &  
\ZZ^j  \oplus H_2(V) \arrow[r, "\beta"] &  H_2(\widetilde{X}'_{\infty})  \arrow[r, "\partial"] & 
H_1(U \cap V) \cong \Gamma/[\Gamma,\Gamma],
\end{tikzcd}
\end{equation}
since $H_2(U) = \ZZ^j$,  $H_2(U \cap V) = H_2( S^3/\Gamma) = 0$, and where $\beta$ is the sum mapping. 
The divisor class $[\tilde{D}]$ is a generator in $ H_2(\widetilde{X}_{\infty}')$. From \eqref{cd1}, 
the class $[m \tilde{D}] = \beta( c_1, c_2)$, where $c_1 \in H_2(U)$, and $c_2 \in H_2(V)$, 
where $m = \big|  \Gamma/ [\Gamma, \Gamma] \big|$. 
We know that the classes $[E_j'] \in H_2(U)$ map to generators in $H_2( \widetilde{X}_{\infty}')$, 
under inclusion, so we have 
\begin{align}
[m \tilde{D}] = \sum_{j} b_j [E'_j] + \beta(0,c_2),
\end{align}
where $b_j \in \ZZ$. Rearranging, we have
\begin{align}
\beta(0,c_2) = [m \tilde{D}]  -  \sum_{j} b_j [E'_j].
\end{align}
The right hand side is therefore the nontrivial homology class we were seeking which is in the image of $\iota_*$. 

The upshot of this discussion is that we can find can find a representative $\sigma$ of the homology class of 
$[m \tilde{D}]-  \sum_{j} b_j [E'_j] $ whose image avoids a tubular neighborhood all the divisors which get blown down. Such a representative is a finite linear combination of $2$-simplices, $\sigma =  \sum a_j \sigma_j,$ where
\begin{align}
\sigma_j : \Delta^2  \rightarrow \widetilde{X}_{\infty}' \setminus N_{\epsilon}(E'),  
\end{align}
with $b_i \in \ZZ$ and, where $\Delta^2$ is a standard $2$-simplex. Note that we can assume that $\sigma_j$ is a smooth mapping since singular homology with continuous chains is isomorphic to singular homology with smooth chains on any smooth manifold.

By the gluing method used in the proof of \ref{tblemma}, there exists a K\"ahler form $\tilde{\omega}$
on $\widetilde{X'_\infty}$, such that
the restriction of $\tilde{\omega}$ on $\widetilde{X}'_\infty\setminus N_{\delta}(E')$ equals to 
$f^*\omega'_{\infty}$, and with respect to which the divisors $E_j'$ have arbitrarily small area.
Note that we can choose $\epsilon$ so that $f(N_{\epsilon}(E'))$ is contained in $B_{\delta}(x_{\infty}')$. 
Then we have
\begin{align}
\int_{\sigma}\tilde{\omega} = \int_{ [m \tilde{D}]-  \sum_{j} b_j [E'_j]}\tilde{\omega} \geq \frac{m}{2} 
\int_{\tilde{D}} \tilde{\omega} 
> \frac{m}{2} \int_{D\setminus B_\delta(x'_\infty)}\omega'_\infty>0.
\end{align}
The diffeomorphism $\psi'_i$ embeds $X'_\infty\setminus B_\delta(x'_\infty)$ into $X$.
Also, by the Mayer-Vietoris sequence, $H_2(X_{\infty}' \setminus  B_\delta(x'_\infty), \QQ)$ embeds into $H_2(X, \QQ)$. 
Therefore we can view the class $[(\psi'_i)_* f_* \sigma] $ as a class in $H_2(X,\QQ)$, which is independent of $i$ when $i$ is sufficiently large. Then 
\begin{align}
[(\psi'_i)_*f_* \sigma] = \sum_{1\leq j\leq k} q_j [W_j]
\end{align}
where each $q_j\in \QQ$, and $[W_1],\ldots,[W_k]$ is the basis of $H_2(X,\ZZ)$ as defined above.
Then we have
\begin{align}
\label{rescaled_homology_vanish}
\int_{(\psi'_i)_*f_* \sigma} \frac{1}{r_i^2}\omega_i = \sum_{1\leq j\leq k}q_{j} \int_{W_j} \frac{1}{r_i^2}\omega_{b,i} \to 0
\end{align}
However, by the pointed Cheeger-Gromov convergence, 
we have
\begin{align}
\int_{(\psi'_i)_*f_* \sigma} \frac{1}{r_i^2}\omega_i = \int_{f_* \sigma} {\psi'_i}^*(\frac{1}{r_i^2}\omega_i)\xrightarrow{i\to\infty}
\int_{f_* \sigma}\omega'_\infty = \int_{\sigma} f^*\omega'_\infty = \int_{\sigma}\tilde{\omega} >0
\end{align}
which contradicts with (\ref{rescaled_homology_vanish}).
This implies that $X'_\infty$ is isomorphic to $\CC^2/\Gamma$.

The Hirzebruch signature theorem for an ALE SFK orbifold
with group $\Gamma$ at infinity, and a single orbifold point $p$ with 
group $\Gamma'$,  
\begin{align}
\label{taum4}
\tau(Y)= - \frac{1}{12\pi^2}\int_{Y} \Vert W^- \Vert^2 dV_{g}  + \eta(S^3/\Gamma)
- \eta(S^3/\Gamma').
\end{align}
In our case $Y = X_{\infty}' = \CC^2/ \Gamma$, so $\tau(Y) = 0$, and since $\Gamma = \Gamma'$
this implies that $W^-(g_{\infty}') \equiv 0$. The Chern-Gauss-Bonnet theorem in this setting states that
\begin{align}
\label{chim4}
\chi(Y)= \frac{1}{8\pi^2}\int_{Y}\Big( \Vert W^- \Vert^2-\frac{1}{2} |E|^2\Big)dV_{g}  + \frac{1}{|\Gamma|}  + 1 - \frac{1}{|\Gamma'|} 
\end{align}
Again, since $Y = X_{\infty}' = \CC^2/ \Gamma$, we have $\chi(Y) = 1$, and this implies that $E \equiv 0$.  Consequently, $g_{\infty}'$ is a flat metric. 

To finish the proof,  we will next show the convergence \eqref{L2_continuous}.
If there is no smooth energy concentration point in $X'_\infty$, then the sequence of highest curvature points $x'_i$ converges to 
the only singular point, which is the vertex of the cone. As a result, the metrics converge smoothly on $X'_\infty\setminus B_{\frac{1}{2}}(x'_\infty,g'_\infty)$,
and \eqref{L2_continuous} is a direct consequence of this.
\begin{lemma}
\label{no_smooth_energy_concentration}
There exists no smooth energy concentration point in $X'_\infty$.
\end{lemma}
\begin{proof}
Assume on the contrary that there exists a smooth energy concentration point $p\in X'_\infty$. Then there exists a sequence of points $p_i\in X$ that converges to $p$ in the Gromov-Hausdorff topology. 
For $i$ sufficiently large, there exists a $\delta>0$, such that the the geodesic ball $B_\delta(p_i,\frac{1}{r_i^2}g_i)$ is homeomorphic
to the bubble tree that ``bubbles-off'' at $p$. 
Since $p$ is a smooth energy concentration point, by choosing $\delta>0$ small enough, $B_\delta(p,g'_\infty)$ is diffeomorphic to the standard $4$-ball. 
Then when $i$ is sufficiently large,
there exists a smooth function $\rho_i$ which is close to the radius function of the geodesic ball $B_\delta(p_i,\frac{1}{r_i^2}g_i)$,
such that $B(i,\delta):=\{\rho_i<\delta\}\subset B_{2\delta} (p_i,\frac{1}{r_i^2}g_i)$,
the boundary $\partial B(i,\delta)$ is diffeomorphic to the standard $3$-sphere, and $\rho_i^2-\delta^2$ is a strictly plurisubharmonic function near the boundary.
Then $B(i,\delta)$ is a strictly pseudoconvex relative open subset in $X$.
By \cite[Theorem 1]{Narasimhan}, there exists a Remmert reduction that maps $B(i,\delta)$ to a Stein space $B'(i,\delta)$,
which contracts a compact analytic subset to isolated points in $B'(i,\delta)$.
By the Stein factorization theorem \cite{GrauertRemmert}, since $B(i,\delta)$ is a normal complex space, $B'(i,\delta)$ is also a normal
complex space. Then by \cite[Theorem a]{Narasimhan1962}, any local holomorphic function in $B'(i,\delta)$ can be extended to a global
function in $B'(i,\delta)$. As a direct consequence, $B'(i,\delta)$ can be embedded into a Euclidean space. 
Furthermore, the boundary sphere $\partial B(i,\delta)$ together with its CR-structure $I$ induced by the complex structure $J$ 
can be embedded into $B'(i,\delta)$. 
Then $(\partial B(i,\delta),I)$ is a CR-embeddable $3$-sphere and $I$ is a small perturbation of the standard CR-structure on $3$-sphere. 
Then by \cite[Section 5]{Lempert94}, the Stein space enclosed by $\partial B(i,\delta)$ is smooth and is diffeomorphic to standard ball in $\CC^2$. As a result, 
$B(i,\delta)$ is obtained by iterative blowups of a $4$-ball.
Since $p$ is a smooth energy concentration point, by Lemma \ref{tblemma}, the second Betti number of the first bubble must be positive.
Then the topology of $B(i,\delta)$ is nontrivial, and there exists at least one $(-1)$-curve
in $B(i,\delta)$. However, this contradicts with the assumption that $X$ is minimal.
\end{proof}

It follows from the above that \eqref{L2_continuous} holds,
which is a contradiction since $g_{\infty}'$ is a flat metric. 
This contradiction finishes the proof of Proposition \ref{uniform_rate}.

\end{proof}

\section{Compactness II. The limit is birationally dominated by $X$}
\label{remove}
Recall that, $(X,J)$ is a minimal complex surface, and 
$g_0$ is a fixed background K\"ahler ALE metric, with K\"ahler form $\omega_0$. 
Without loss of generality we can assume that there is a fixed ALE coordinate system for $g_0$,
\begin{align}
\Psi: X \setminus K\rightarrow (\RR^4\setminus \overline{B})/\Gamma, 
\end{align}
with $g_0$ ALE of order $-2 < -\mu < -1$ and 
$J - J_0 \in C^{\infty}_{-\mu}$.

As a result of Proposition \ref{uniform_rate}, we have
\begin{align}
\label{CGC}
(X,g_i,J,x_i)\xrightarrow{\mathrm{pointed\; Cheeger-Gromov}} (X_\infty,g_\infty,J_\infty,x_\infty).
\end{align}
with uniform ALE asymptotic rate $-2<-\mu<-1$, i.e.,
the sequence convergence in the pointed Gromov-Hausdorff pseudo-distance, 
and for any $\delta>0$, there exists a diffeomorphism 
$\psi_i: X_\infty\setminus B_{\delta}(x_\infty)\rightarrow X_i$, such that 
${\psi_i}^* g_i  \xrightarrow{C^\infty_{-\mu}} g_\infty $, 
${\psi_i}^* J \xrightarrow{C^{\infty}_{-\mu}} J_\infty$, 
and $(X_\infty,J_\infty)$ is birationally equivalent to $(X,J)$. Furthermore, 
as can be seen in the proof of Proposition \ref{uniform_rate},
$\Psi$ is common ALE coordinate
\begin{align}
\label{uniform_ALE_coordinate}
\Psi: X\setminus K\rightarrow (\RR^4\setminus B_R)/\Gamma
\end{align}
where $K$ is a compact subset of $X$, and $B_R$ is a Euclidean ball of radius $R$ centered at $0$,
such that for any $i\geq 1$,
$x_i\in K$, and there exists some constant $C(k)>0$ independent of $i$ such that
$\|{\Psi}_*g_i-g_{Euc}\|_{C^{k,\alpha}_{-\mu}(g_{Euc})} < C(k)$,
$\|{\Psi}_*J-J_{Euc}\|_{C^{k,\alpha}_{-\mu}(g_{Euc})}<C(k)$.

\begin{remark}
Without loss of generality, we may assume for the rest this section that
there is only one energy concentration point $x_\infty\in X_\infty$. 
It is a straightforward generalization to the case of 
multiple energy concentration points. 
\end{remark}
Before giving the proof, we will first demonstrate the no singularity result in the case when $X$ is Stein
by a simple topological argument.
\begin{proposition}If $(X,J)$ is moreover assumed to be Stein then Theorem~\ref{compactness} is true. 
\end{proposition}
\begin{proof}
By Proposition \ref{uniform_rate}, $X_\infty$ is birationally equivalent to $X$.
Let $\widetilde{X_\infty}$ be the minimal resolution of $X_\infty$. Blowdown all $(-1)$-curves 
in $\widetilde{X_\infty}$ to obtain
a Stein surface $Z$. By Proposition \ref{minimal_ALE_moduli}, $Z$ is biholomorphic to $X$. 
Clearly, we have $b_2(\widetilde{X_\infty})\geq b_2(X_\infty)\geq b_2(Z) = b_2(X)$, 
with  equality if and only if $X_\infty \simeq Z$.
From \eqref{b2sum}, $b_2(X)\geq b_2(X_\infty)$.
Then $b_2(X) = b_2(X_\infty) = b_2(Z)$.
This implies that $X_\infty$ is isomorphic to $Z$, and thus $X_\infty$ is smooth.
If $x_{\infty}$ is an energy concentration point, then the first bubble $Y_1$ there is 
an AE SFK orbifold. But by the above inequalities and \eqref{b2sum}, we would 
have $b_2(Y_1) = 0$. Lemma \ref{tblemma} implies that $Y_i$ is biholomorphic to $\CC^2$ with the flat metric, but this is a contradiction, since any bubble must have a point with non-zero curvature. 
Since there are no energy concentration points, Theorem \ref{compactness} follows
(see Section~\ref{cpsubsec} below for the remainder of the argument).
\end{proof}

When $X$ is not Stein, the vanishing of holomorphic curves makes the above topological argument fail.
Heuristically, the orbifold singularity in $X_\infty$ is formed by the vanishing (in area) of some (real) $2$-dimensional submanifolds in $X$ which represent some homology classes. When those submanifolds are holomorphic curves,
the vanishing of their areas implies the degeneracy of the K\"ahler form, which leads to a contradiction.
The difficulty is, a priori, the diffeomorphisms in the pointed Cheeger-Gromov convergence
could be far from being holomorphic. They could map some submanifold in $X$ which is far from being
holomorphic to a holomorphic curve in $X_\infty$.
As a result, the integral of K\"ahler form over those submanifolds could be much smaller than their areas
and one could conclude nothing about the  degeneracy of the K\"ahler form.
Our strategy is to ``chase'' the submanifolds in $X$ that homologically contract to form the singularity in $X_\infty$, and show that
they are ``very close'' to being holomorphic. The fact that $X$ is birationally equivalent with $X_\infty$ plays
an important role in our proof.
Our first theorem in this section deals with this difficulty. 
Roughly, it says that, when $i$ is sufficiently large,
the error between the diffeomorphism $\psi_i$ in 
the pointed Cheeger-Gromov and a holomorphic map is very small.
\begin{theorem}
\label{CGCB}
Consider the convergent subsequence in Theorem \ref{compactness}, 
where $X$ is assumed to be minimal,
\begin{align}
(X,g_i,J,x_i)\xrightarrow{\mathrm{pointed \; Cheeger-Gromov}} (X_\infty,g_\infty,J_\infty,x_\infty)
\end{align}
with uniform ALE asymptotic rate $-2<-\mu<-1$.
For any $\delta>0$, there exists a diffeomorphism $\psi_i: X_\infty\setminus B_{\delta}(x_\infty)\rightarrow X$,
with $\psi_i^*g_i\to g_\infty$, $\psi_i^*J\to J_\infty$.
Then there exists a surjective bimeromorphism $\Phi: X\rightarrow X_\infty$, i.e., $X$ is the minimal resolution
of $X_\infty$,
such that on $X_\infty\setminus B_\delta(x_\infty)$
\begin{align}
\label{convest}
\|\Phi\circ\psi_i - Id\|_{C^{k,\alpha}_{d_N}(g_\infty)} < \epsilon(i\;|\;\delta,k)
\end{align}
where $d_N$ is the highest degree among holomorphic functions $u^1,\ldots,u^N$.
\end{theorem}
\begin{proof}
In the following proof we will denote $E = \cup_j E_j$ as the union of exceptional divisors in $(X,J)$
and $E_{\infty} = \cup_j E_{\infty,j}$ as the  union of exceptional divisors in 
$(X_{\infty},J_{\infty})$.

From Section~\ref{versal}, we know the complex structure $J$ is determined by holomorphic functions $u^1,\ldots,u^N$ with 
polynomial growth rate on $X$ that satisfy certain polynomial relations. Therefore we have a mapping
\begin{align}
\pi_X : X \rightarrow Z,
\end{align}
where $Z \subset \CC^N$ is a Stein space given by the image of the mapping 
$\pi_X(p) = (u^1(p), \dots u^N(p))$. 
Note that $\pi_X$ is the contraction of $E$.  

Furthermore, by Proposition \ref{uniform_rate},
$(\psi_i^*u^1,\ldots,\psi_i^*u^N)$ converge to holomorphic functions $(u^1_\infty,\ldots,u^N_\infty)$
on $X_\infty\setminus B_\delta(x_\infty)$, which satisfy the same polynomial relation(s) as $u^1,\ldots,u^N$. 
Since $X_\infty$ is one-convex,
$(u^1_\infty,\ldots,u^N_\infty)$ can be extended to holomorphic functions on $B_\delta(x_\infty)$.
Then we have a holomorphic map:
\begin{align}
\label{map_infty}
\pi_{X_\infty} : X_{\infty} \rightarrow Z, 
\end{align}
where $\pi_{X_\infty}(p)= (u^1_\infty(p),\ldots,u^N_\infty(p))$.
The image is exactly $Z$ because outside of a large ball 
the mappings $\pi_{X_\infty}(X_{\infty} \setminus B_R(x_{\infty}) \subset Z$
and the image of $\pi_{X_\infty}$ must be isomorphic to $Z$ by the 
 proof of Proposition \ref{minimal_ALE_moduli}.
Note that $\pi_{X_{\infty}}$ is the contraction of $E_{\infty}$.

Denote $\widetilde{X_\infty}$ as the minimal resolution of $X_\infty$ with the projection map
$\pi: \widetilde{X_\infty}\rightarrow X_\infty$. 
Since $X$ is minimal, and $\widetilde{X_{\infty}}$ is smooth and in the same birational class,  
Proposition \ref{minimal_ALE_moduli} implies this 
existence of a surjective bimeromorphism 
\begin{align}
f: \widetilde{X_\infty} \rightarrow X.
\end{align}
We summarize all of the maps in the following diagram

\begin{equation}
\begin{tikzcd}[column sep=large]
\label{fcd}
\widetilde{X_{\infty}} \arrow[d, "f"]   & \arrow[l, hook'] \widetilde{X_{\infty}} \setminus
\pi^{-1}( B_{\delta}(x_{\infty}))  \arrow[d,"\pi"]  \\
 X    \arrow[d,"\pi_X"] & \arrow[l,"\psi_i"{below}] X_{\infty} \setminus B_{\delta}(x_{\infty}) 
\arrow[dl,"\pi_{X_\infty}"] \\
Z &  \\
\end{tikzcd}.
\end{equation}
Consider the mapping $ A= \pi_X\circ f \circ (\pi_{X_\infty} \circ\pi)^{-1} : Z \rightarrow Z$.
It is easy to see this mapping is invertible, and thus is an automorphism of $Z$.
Since $X$ is minimal, any automorphism of $Z$ can be lifted up to an automorphism of $X$.
Then there exists an automorphism $B: X\rightarrow X$, such that $A^{-1}\circ \pi_X = \pi_X \circ B$.
Redefining $f$ to be
\begin{align}
B\circ f: \widetilde{X_\infty}\rightarrow X
\end{align}
We then have
\begin{align}
\label{fcond}
\pi_X\circ f \circ (\pi_{X_\infty} \circ \pi)^{-1} = Id :Z \to Z.
\end{align}
Denote $(E_\infty)_\eta$ as the $\eta$-tubular neighborhood of $E_\infty$
in $X_\infty$ with respect to $g_\infty$.
Restrict ${f}$ on $X_\infty\setminus E_\infty$, then we have a biholomorphic map
$f: X_\infty\setminus E_\infty$ onto its image in $X$.
As a result of this, by part (3) in Proposition \ref{uniform_rate} we can
choose a radius $R>0$ sufficiently large so that the composite 
\begin{align}
\tau_i:=\pi\circ f^{-1}\circ\psi_i: X_\infty\setminus B_R(x_\infty) \rightarrow X_\infty
\end{align}
is well-defined, 
since by the uniform ALE asymptotic rate, when $R$ is sufficiently large,
any holomorphic curve contracted by $f$ is contained in $\widetilde{B_R(x_\infty)}$.

By \eqref{fcond}, we have 
\begin{align}
\tau_i' \equiv \pi_{X_\infty} \circ \tau_i \circ \pi_{X_\infty}^{-1} 
= \pi_X \circ\psi_i \circ  \pi_{X_\infty}^{-1}  : Z \setminus U_R
\rightarrow Z
\end{align}
where $U_R : = \pi_{X_\infty}(B_R(x_\infty))$.

We then have
\begin{align}
\lim_{i \to \infty} \tau_i' = \lim_{\i \to \infty}  (\pi_X \circ\psi_i)  
\circ  \pi_{X_\infty}^{-1}
= \pi_{X_{\infty}} \circ \pi_{X_\infty}^{-1} = Id,
\end{align}
which implies that 
\begin{align}
\lim_{i \to \infty} \tau_i = Id : X_{\infty} \setminus B_R(x_{\infty}) \rightarrow X_{\infty},
\end{align}
where the convergence is in any $C^{k,\alpha}_{d_N}$-norm on $X_\infty\setminus B_R(x_\infty)$, 
since any $\psi_i^*u^j$ converges in $C^{k,\alpha}_{d_N}$-norm, which implies
$\|\pi_X\circ\psi_i-\pi_{X_\infty}\|_{C^{k,\alpha}_{d_N}(X_\infty\setminus B_R(x_\infty))}$ converges to $0$.

We next want to show that $\tau_i$ converges to the identity away from $B_{\delta}(x_{\infty})$.  For this, we need a surjective bimeromorphism from $X_i$ to $X_\infty$.
Since $(X,J)$ is minimal, such a mapping does not exist precisely when there is a $(-1)$-curve in $\widetilde{X_\infty}$.
The following lemma shows that this cannot happen.
\begin{lemma}
\label{no}
There exists no $(-1)$-curve in $\widetilde{X_\infty}$.
\end{lemma}
\begin{proof}
Without loss of generality, 
assume there exists a single  $(-1)$-curve $\tilde{E}_{\infty, -1} \subset \widetilde{X_\infty}$ which is not in the image of any birational map from $X$ to $\widetilde{X_\infty}$
(the argument for multiple $(-1)$-curves is similar).
Denote the image of $\pi(\tilde{E}_{\infty,-1})$ in $X_\infty$ as $E_{\infty,-1}$.
Since $ \widetilde{X_\infty}$ is the minimal resolution of $X_{\infty}$ and
$\tilde{E}_{\infty,-1}$ is not contracted by $\pi$, $E_{\infty,-1}$ has its regular part non-empty.

Denote $q = f(\tilde{E}_{\infty,-1})$, which is a single point since $X$ is a minimal K\"ahler surface with no $(-1)$-curve in it.  
If $\tau_i = \pi\circ f^{-1}\circ \psi_i$ can not be extended to a map on $X_\infty\setminus B_\delta(x_\infty)$, then for any $i$ sufficiently large,
$q\in \psi_i(X_\infty\setminus B_\delta(x_\infty))$.
Denote $p_{i} = \psi_i^{-1}(q)$.
Then as $i\to\infty$, up to a subsequence, 
$p_{i}$ converges to a point $p_{\infty}$ in the closure of $X_\infty\setminus B_\delta(x_{\infty})$.
Without the loss of generality, we can assume that $p_{\infty}\in X_\infty\setminus B_{2\delta}(x_\infty)$,
since we can always shrink $\delta$ to $\frac{\delta}{2}$.
Let $c>0$ be a positive number which can be chosen to be arbitarily small, and $B_{c\cdot\delta}(p_{\infty})$ 
be a geodesic ball centered at $p_{\infty}$ with radius of $c\cdot\delta$.

Then $\tau_i$ can be extended to a mapping
\begin{align}
\tau_i = \pi\circ f^{-1}\circ \psi_i: W \rightarrow X_\infty,
\end{align}
where $W =  X_\infty\setminus \{ B_\delta(x_\infty) \cup B_{c\delta}(p_{\infty}) \}.$

Let $(E_{\infty})_{\eta}$ denote the tubular neighborhood of $E_{\infty}$ with respect to $g_\infty$.
On $W\setminus (E_\infty)_\eta$, by the convergence of complex structure, $\tau_i^*u^1_\infty,\ldots,\tau_i^*u^N_\infty$
converge to some holomorphic functions $v^1,\ldots,v^N$. Since we have shown that $\tau_i$ converges to
$Id$ on $X_\infty\setminus B_R(x_\infty)$, $v^j = u^j_\infty$ outside of $B_R(x_\infty)$.
Then by the unique extension of holomorphic functions, $v^j = u^j_\infty$ on $W\setminus (E_\infty)_\eta$.
Since $(u^1_\infty,\ldots,u^N_\infty)$ embeds $W\setminus (E_\infty)_\eta$ into $\CC^N$, this implies that for any $\eta > 0$, 
\begin{align}
\lim_{\i \to \infty} \tau_i =  Id : W \setminus (E_{\infty})_{\eta} \rightarrow X_{\infty}.
\end{align}
Let $p\in E_{\infty,-1}\cap W$ be a point in the regular part of 
$E_{\infty,-1}$ (which is non-empty for $\delta$ sufficiently small)
and such that $B_{C \delta}(p)$ does not intersect any 
other exceptional curve in $E_{\infty}$ for some $C > 0$.   
Near $p$, we have a holomorphic coordinate $\phi =  (z^1,z^2): U \to \CC^2$ 
 of $(X_\infty, J_{\infty})$ 
with the property that $E_{\infty,-1} \cap U = \{x \in U \ | \ z^1(x) =0\}$, and $p = (0,0)$.
Define $T$ as a small polydisc neighborhood of $p$ by 
$T := \{x \in U \ | \ |z^1(x)|\leq C\cdot \delta, |z^2(x)|\leq C\cdot \delta\}$,
such that $T\subset W$.

By a result of Greene-Krantz \cite[Theorem 1.13]{GreeneKrantz},
there exists a diffeomorphism $\nu_i: \phi(T)\rightarrow \CC^2$, such that, $|\nu_i-Id|<\epsilon(i \;|\; \delta)$,
and $\nu_i^*(J_{Euc}) = \phi_* \psi_i^*(J)$ on $T$.
We can choose $\eta = \frac{1}{2}C\cdot\delta$.
On the annulus $A := \{ (z^1,z^2) : \frac{3}{4}C\cdot\delta<|z_1|<\frac{5}{4}C\cdot\delta\}$, 
$\phi \circ \tau_i \circ \phi^{-1}$ converges to $Id$.
The mapping 
\begin{align}
\zeta_i := \phi \circ \tau_i\circ \phi^{-1} \circ \nu_i^{-1}: T  \rightarrow \CC^2
\end{align}
is biholomorphic to its image since
\begin{align}
\begin{split}
(\zeta_i)_* J_{Euc} &= (\phi \circ \tau_i)_* ( \nu_i \circ \phi)^{-1}_* J_{Euc}\\
& = \phi_* (\tau_i \circ \psi_i^{-1})_* J
= \phi_* (\pi \circ f^{-1})_* J = \phi_* J_{\infty} = J_{Euc}. 
\end{split}
\end{align}
Therefore $\zeta_i$ can be represented as a pair of holomorphic functions
$(\zeta^1_i,\zeta^2_i)$,
and by the maximum principle, we have $|\zeta_i- Id| < \epsilon(i\;|\;\delta)$ on $T$.
We must therefore have
\begin{align}
f \circ \pi^{-1}  := \psi_i\circ \phi^{-1} \circ \nu_i^{-1}\circ\zeta_i^{-1} \circ \phi 
\end{align}
on $T$, since both sides are holomorphic function which agree on $\phi^{-1}(A)$.
By the estimates of $\psi_i,\nu_i,\zeta_i$ above, we have $|f \circ \pi^{-1}-\psi_i|<\epsilon(i\;|\;\delta)$ on $T$.

Choose another point $p'\neq p \in E_{\infty,-1}\cap T$ such that the distance $d_{g_\infty}(p,p')>C'\cdot \delta$ for some $C'>0$.
Recall that $f$ contracts $\tilde{E}_{\infty, -1}$ to a point,
so $f \circ \pi^{-1}$ maps $E_{\infty,-1} \cap T$ to a point, 
therefore $f(p) = f(p')$. 
However, by the estimates above
\begin{align}
| (f \circ \pi^{-1})^*g_i - g_\infty|\leq C\cdot|\psi_i^*g_i-g_\infty|\leq \epsilon(i\;|\;\delta),
\end{align}
on $T$ so we must have $d_{g_i}(f(p),f(p'))> \frac{C'\cdot \delta}{2}$
when $i$ is sufficiently large.
This implies a contradiction, and thus there is no such $(-1)$-curve in $\widetilde{X_\infty}$ as assumed at the beginning of the proof.  
\end{proof}
We now complete the proof of Theorem \ref{CGCB}.
By Lemma \ref{no}, 
the mapping $f: \widetilde{X_{\infty}} \rightarrow X$, which we can assume satisfies \eqref{fcond}, is an isomorphism. 
Consider the bimeromorphism 
\begin{align}
\label{bimeromorphism}
\Phi \equiv \pi \circ f^{-1}  :X \rightarrow X_\infty
\end{align}
which satisfies $\pi_{X_\infty}\circ \Phi\circ {\pi_X}^{-1} = Id$ on $Z$.
By a similar argument as in the analysis above, the composite
\begin{align}
\tau_i := \Phi\circ \psi_i: X_\infty \setminus B_\delta(x_\infty)\rightarrow X_\infty
\end{align}
converges to $Id$.  Clearly, the estimate \eqref{convest} is satisfied on $X_{\infty}  \setminus B_\delta(x_\infty)$.
\end{proof}
\begin{remark}
In the case when $X$ is asymptotic to $\CC^2/\Gamma$ and $\Gamma$ is a finite subgroup of $SU(2)$, i.e., the case
of gravitational instantons, by \cite{Bando}, the limit $X_\infty$ is an Einstein orbifold. It is shown that the bubble-tree must be diffeomorphic to a cyclic quotient of a hyperk\"ahler ALE manifold. 
It is a direct consequence of this that there is no $(-1)$-curve in $\widetilde{X_\infty}$.
This illustrates that the singularity of ALE SFK orbifold limit could be much more complicated than in the Ricci-flat case. 
\end{remark}

We end this section with the following direct consequence of Theorem \ref{CGCB}. 
\begin{corollary}
\label{nosmooth}
There are no smooth energy concentration points in $X_\infty$.
\end{corollary}
\begin{proof}
Without loss of generality, assume $x_\infty\in X_\infty$ is a smooth energy concentration point and there is no other energy concentration point
in $X_\infty$. 
Then the bimeromorphism $\Phi$ from $X$ to $X_\infty$ is moreover an isomorphism.
Then by \eqref{b2sum} and Lemma~\ref{tblemma}, the bubble that degenerates at $x_\infty$ is $\CC^2$ with the flat metric, which is a contradiction.  
\end{proof}
\section{Compactness III. Bubbles are resolutions}
\label{bubbles}

Our first goal is to show that each bubble in the bubble tree is a resolution of the corresponding singularity 
in the previous bubble.
Here are some notations and facts.
Denote the rescaled sequence $(B_\delta(x_i),\frac{1}{r_i^2} g_i,x_i)$
as $(Y_i,g'_i,y_i)$, where $B_\delta(x_i)$ is a geodesic ball of radius $\delta$ with respect to $g_i$, 
and the scaling factor $r_i$ is to be determined below. By Theorem \ref{CGCB}, there exists a $\delta>0$ such that
$B_\delta(x_i)$ contains and only contains holomorphic curves that are contracted to $\{x_\infty\}$ in the limit.
Specifically, there exists a bimeromorphism $\Phi$, which maps $X$ onto $X_\infty$,
and $\Phi\circ\psi_i$ converges to $Id$ on $X_\infty\setminus B_\delta(x_i)$.
Then we also have $\Phi^{-1}(x_\infty)\subset B_\delta(x_i)$,
where $\Phi^{-1}(x_\infty)$ is a union of exceptional divisors ${E'}_{1}\cup\ldots\cup {E'}_{m}$.

The natural scale of $r_i$ to choose is the ``energy scale'', i.e., choose $r_i$ such that
\begin{align}
\label{energy_scale}
\int_{Y_i\setminus B_1(y_i)} \|Rm(g'_i)\|^2 dV_{g'_i} = \frac{\epsilon_0}{2}
\end{align}
where $\epsilon_0$ is the energy threshold introduced in Section~\ref{volume}.
The naturality is in the sense that, the ``energy scale'' preserves the topology,
i.e., after gluing the ``bubble tree'' to the limit space, we will acquire the topology of the
original manifold \cite{Bando}. We begin with the following lemma, which says that the diameter of the exceptional divisors is controlled on the ``energy scale''.
\begin{lemma}
\label{scale_bound}
Let $(Y_i,g'_i,y_i)$ be the rescaled sequence defined above, with the scaling factor chosen to be the
``energy scale'', i.e., the property \eqref{energy_scale} is satisfied.
Then there exists a constant $R_{en}>0$ independent of $i$, such that, when $i$ is sufficiently large,
each holomorphic curve in $Y_i$ is contained in the geodesic ball $B_{R_{en}}(y_i)$.
\end{lemma}
\begin{proof}
By the choice of $r_i$ as in \eqref{energy_scale} and the $\epsilon$-regularity theorem of \cite{TV},
there exists a constant $C>0$ independent of $i$ when $i$ is sufficiently large, such that
$\|Rm(y)\|_{g'_i}<C\cdot r^{-2}$ for $y\in Y_i\setminus B_1(y_i)$.
Then for $i$ sufficiently large, there exists a radius $R>1$, such that on $Y_i\setminus B_R(y_i)$,
$r^2$ is a plurisubharmonic function.
If Lemma \ref{scale_bound} is false, then there exists a holomorphic curve $E$ that intersects with 
$Y_i\setminus B_R(y_i)$ non-trivially for infinitely many $i$.
Let $p_i$ be the point in $E$ where $r^2$ achieves its maximum value.
Since $r^2$ is a plurisubharmonic function, its restriction on $E$ is a subharmonic function. By the maximum principle,
$r^2$ is constant on $E\cap Y_i\setminus B_R(y_i)$, which contradicts with the fact that $E \not\subset Y_i\setminus B_R(y_i)$.
Then we can set $R_{en} = R$, and the lemma is proved. 
\end{proof}

We next need a more precise estimate connecting the bubbles in the ``energy scale'' to the birational structure. 
Before we state and prove this, we next summarize some results in \cite{Lempert91},\cite{Lempert94} with mild modifications under our setting which are the crucial ingredient for this step.

\subsection{Summary of Lempert's results}

Let $(S_1/\Gamma',I)$ be the unit sphere centered at $\{0\}$ in $\RR^4/\Gamma'$ associated with a CR-structure $I$,
where $\Gamma'$ is a finite subgroup of $U(2)$ with no complex reflection.
We have the lifting of the CR-structure in the universal cover $S_1\subset \RR^4$ still denoted as $I$.
Assume $(S_1,I)$ is embeddable, i.e., there exists a diffeomorphism compatible with the CR-structure $I$,
that embeds $S_1$ into $\CC^m$ for some integer $m$.
Let $(S_1,I_{std})$ be the CR-structure induced from the standard complex structure in $\CC^2$.
Denote $(W,J_{std})$ as the analytic compactification of $\CC^2\setminus B_1(0)$ constructed by attaching a divisor $D\simeq \PP^1$
analytically to its end, with $J_{std}$ the analytic extension of the complex structure $J_{Euc}$ on $\CC^2$. 
Then $(W,J_{std})$ is a compact strictly pseudoconcave manifold.
\subsubsection*{\em{L.1.}}
There exist $\epsilon_1>0$, a positive integer $k$, such that if
$\|I - I_{std}\|_{C^k(S_1)} < \epsilon_1$, there exist $\epsilon_2 = \epsilon_2(\epsilon_1 \,|\, k)$, $0<k'<k$,
a complex structure $J$ on $W$ such that
$\|J-J_{std}\|_{C^{k'}(W)} < \epsilon_2$, $J|_{S_1} = I$, and 
$D$ is also holomorphic with respect to $J$.
The norm $C^{k'}(W)$ is defined by using the restriction of Fubini-Study metric on $W$.

Since $J$ is a small perturbation of $J_{std}$, by formula (4.6) in \cite{HanViaclovsky},
$J = E_{J_{std}}(\phi)$, where $\phi$ is a section of $\Lambda^{0,1}\otimes T^{1,0}$ with a small norm.
Since $(W,J_{std}), (S_1,I)$ are $\Gamma'$-equivariant, 
we can have $J$ to be $\Gamma'$-equivariant by averaging $\phi\in \Gamma(W, \Lambda^{0,1}\otimes T^{1,0})$ with the $\Gamma'$-action.
\subsubsection*{\em{L.2.}}
The divisor $D$ is associated with a holomorphic line bundle $L$ on $(W,J_{std})$.
There exists a basis $s_0,s_1,s_2$ of $H^0(W,L)$, where $s_0|_D = 0$ is the defining section of $D$.
When $\epsilon_1$ is small enough, the divisor $D$ also induces a line bundle $L'$ on $(W, J)$, 
which is holomorphic with respect to the complex structure $J$. There exists a smooth bundle isomorphism
$\Pi: L\rightarrow L'$, where $\Pi|_D = Id$.
Since $J_{std},J$ are $\Gamma'$-equivariant, we can require $\Pi$ to be $\Gamma'$-equivariant,
i.e., for any $\gamma\in\Gamma'$, $\gamma^*\circ \Pi = \Pi \circ \gamma^*$.
(This is because, we can choose a set of open charts $\{U_j\}_{1\leq j\leq r}$, such that 
$U_{j_1}\neq U_{j_2}$ if $j_1\neq j_2$, and $\{\sigma(U_j): \sigma\in \Gamma', 1\leq j\leq r\}$ is a covering of $W$. 
Applying the construction of $\Pi$ in \cite{Lempert94}[Lemma 4.2] on each $U_j$,
and apply the $\Gamma$-action to construct $\Pi$ on other charts of the same orbit.)
There exist sections $\sigma_0,\sigma_1,\sigma_2\in H^0(W,L')$,
such that for each $j=0,1,2$,
$\|\Pi^{-1}(\sigma_j) - s_j\|_{C^{k''}(W)}<\epsilon_3$ for some $\epsilon_3 = \epsilon_3(\epsilon \,|\, k)$, $0<k''<k$.
\subsubsection*{\em{L.3.}}
Denote $\sigma^{(1)}_j$ as the first-order truncation of $\sigma_j$ over $D$ 
(which is the projection of $\sigma_j$ to the normal bundle of $D$).
We have $\sigma^{(1)}_j = s_j^{(1)} \in H^0(D,L)$.
Each $\sigma_j$ is determined by $\sigma^{(1)}_j$.
Specifically, since $s_0$ is $\Gamma'$-invariant and $\Pi$ is $\Gamma'$-equivariant, $\sigma_0$ is $\Gamma'$-invariant.
Let
\begin{align}
s_0^{d_{0}}P_0(s_1,s_2),\ldots,s_0^{d_{N}}P_N(s_1,s_2)
\end{align} 
be generators of $\Gamma'$-invariant elements in $H^0(W,L^k)$, where each $P_j(a,b)$ is a homogeneous polynomial of degree
$k-d_j$, and specifically, $s_0^{d_0}P_0(s_1,s_2) = s_0^{k}$.
Since $s_0$ is $\Gamma'$-invariant,
each $P_j(s_1,s_2)$ is also $\Gamma'$-invariant.
As
$P_j({\sigma_1^{(1)}},{\sigma_2^{(1)}}) = 
P_j({s_1^{(1)}},{s_2^{(1)}})$ on $D$,
and 
$P_j({\sigma_1},{\sigma_2})$ is determined by
$P_j({\sigma_1^{(1)}},{\sigma_2^{(1)}})$,
then $P_j({\sigma_1},{\sigma_2})$ is $\Gamma'$-invariant. 
As a result,
${\sigma_0}^{d_{j}}P_j({\sigma_1},{\sigma_2})\in H^0(W,{L'}^k)$ is also $\Gamma'$-invariant.
\subsubsection*{\em{L.4.}}
Let $(v^1,v^2) = (\frac{\sigma_1}{\sigma_0}, \frac{\sigma_2}{\sigma_0})$.
Then $v = (v^1,v^2)$ embeds $(W\setminus D,J)$ into $\CC^2$, and the image of $S_1$ is close to the unit sphere
centered at $\{0\}$ in $\CC^2$.
For each $1\leq j\leq N$, Let
\begin{align}
u^j = \frac{{\sigma_0}^{d_{j}}P_j({\sigma_1},{\sigma_2})}{\sigma_0^k} = P_j(v^1,v^2)
\end{align}
Then $u = (u^1,\ldots,u^N)$ embeds $\mathcal{N} = (W\setminus D)/\Gamma'$ into $Z\subset \CC^N$,
under which $\mathcal{N}$ is biholomorphic to an open subset of the cone
$Z\subset \CC^N$, where $Z\simeq \CC^2/\Gamma'$, $\{0\}\in Z$ is the quotient singularity of the cone.

\subsection{The first bubble $Y_\infty$ is a resolution}
\label{first_bubble}

From now on, we will choose $r_i$ as the ``energy scale'' as defined in \eqref{energy_scale}.
Up to a subsequence
$(Y_i,g'_i,y_i)$ converges to $(Y_\infty,g'_\infty,y_\infty)$ in the pointed Cheeger-Gromov sense,
where $Y_\infty$ is an ALE SFK orbifold with an end asymptotic to $\RR^4/\Gamma'$. 
Without loss of generality, we can assume $y_\infty$ is the only energy concentration point in $Y_\infty$. 
By Lemma \ref{scale_bound} above, there exists a constant $R_{en}>0$ independent of $i$,
such that each holomorphic curve in $Y_i$ is contained in the geodesic ball
$B_{R_{en}}(y_i)$. Without loss of generality, we can assume $R_{en} = 1$.
\begin{lemma}
\label{Lempert_method}
$Y_\infty$ is birationally equivalent to $\CC^2/\Gamma'$, where $\CC^2/\Gamma'$ is the corresponding
quotient singularity at $x_\infty\in X_\infty$.  Furthermore, there are no smooth energy concentration points in $X_{\infty}$. 
\end{lemma}
\begin{proof}
In the following, the Cheeger-Gromov convergence will always be understood up to picking a subsequence.
Consider the sequence $(Y_i,g'_i,y_i)$ that converges to $(Y_\infty,g'_\infty,y_\infty)$ in the pointed
Cheeger-Gromov sense. 
Denote $A_{a,b}(y_i)$ as the closed annulus in $Y_i$ between the geodesic balls $B_{a}(y_i), B_{b}(y_i)$, $a<b$.
Denote $A_{a,b}(0)$ as the annulus in $\CC^2/\Gamma'$ centered at the origin between the radius $a < b$.
In the next several paragraphs, we will follow the idea of Lempert's method in \cite{Lempert94} to show
that when the radius is large,
the annulus is very close to the standard annulus (up to a diffeomorphism that is close to the identity map).

Let $R>1$ be fixed with its value to be determined later.
By Lemma \ref{scale_bound}, all holomorphic curves that degenerate at $Y_\infty$ are contained in $B_R(y_i)$
for each $i$ sufficiently large.
Denote $V_{3R}(v_i)$ as the image of $B_{3R}(y_i)$ after contracting the exceptional divisors in $B_{3R}(y_i)$
to the point $v_i$.
Let $\overline{V_{3R}(v_i)}$ be the orbifold universal cover of $V_{3R}(v_i)$ with a single orbifold point $v_i$, 
which has a strictly pseudoconvex boundary.
$\overline{V_{3R}(v_i)}$ can be embedded into $\CC^2$. The reason is, 
for $i$ sufficiently large, the bimeromorphism $\Phi$ in \eqref{bimeromorphism} maps $B_\delta(x_i)$ to a subdomain of
$B_{2\delta}(x_\infty)$. By possibly shrinking $\delta$ even smaller, we have 
$B_{2\delta}(x_\infty)$ is biholomorphic to a strictly pseudoconvex domain in $\CC^2/\Gamma'$.
As a result of Theorem \ref{CGCB}, $V_{3R}(v_i)$ can be mapped into $B_{2\delta}(x_\infty)$, henceforth 
can be mapped into $\CC^2/\Gamma'$. Then $\overline{V_{3R}(v_i)}$ can be embedded into $\CC^2$.
The embeddability implies that there exists a pair of holomorphic coordinate functions, which determines
the complex structure of $\overline{V_{3R}(v_i)}$ as $J'_i$.

On the limit $(Y_\infty,g'_\infty,y_\infty)$, there is an ALE coordinate
\begin{align*}
\Psi: Y_\infty\setminus B_{\frac{1}{16}R}(y_\infty)\rightarrow (\RR^4\setminus K)/\Gamma'
\end{align*}
where $K$ is a compact subset contained in $B_{\frac{1}{8}R}(0)$ with respect to $g_{Euc}$.
For any $\delta>0$, we also have a diffeomorphism $\psi'_i: B_{4R}(y_\infty)\setminus B_{\delta}(y_\infty)\rightarrow
B_{5R}(y_i)$, such that ${\psi_i'}^*g'_i$ converges to $g'_\infty$, ${\psi_i'}^*J'_i$ converges to $J'_\infty$
on $B_{4R}(y_\infty)\setminus B_\delta(y_\infty)$.
In order to simplify our symbols,
we will 
use $J'_i, J'_\infty$ to denote complex structures $({\psi_i'}\circ \Psi^{-1})^*J'_i$ and ${\Psi}_*J'_\infty$ respectively
on $A_{\frac{1}{2}R, 3R}(0)$ and also on its universal cover $\overline{A_{\frac{1}{2}R, 3R}(0)}\subset \RR^4$;
denote $S_r$ as the boundary of $\overline{B_r(0)}$.
Our goal is to find a diffeomorphism close to the identity map  that perturbs the complex structure $J'_i$ to the standard
one on $\overline{A_{R,2R}(0)}$. Henceforth, a sequence of the ``perturbed'' coordinate functions will converge as holomorphic
functions , which implies that $Y_\infty$ is a resolution.

We will define the normalized annulus
\begin{align}
(A_{a,b}(0), g''_i, J''_i) = (A_{R\cdot a,R\cdot b}(0), \frac{1}{R^2}\cdot g'_i, J'_i)
\end{align}
and similar for $(A_{a,b}(0),g''_\infty,J''_\infty)$.
We can choose $R$ to be large enough, such that for
any $k>0$ and any sufficiently small $\epsilon(k)>0$, when $i$ is sufficiently large, 
$\|J''_\infty - J_{Euc}\|_{C^{k}(A_{\frac{1}{2},3}(0))}<\frac{\epsilon(k)}{2}$,
$\|J''_i - J''_\infty\|_{C^{k}(A_{\frac{1}{2},3}(0))}<\frac{\epsilon(k)}{2}$,
and consequently
\begin{align}
\label{estimate_3}
\|J''_i - J_{Euc}\|_{C^{k}(A_{\frac{1}{2},3}(0))}<\epsilon(k)
\end{align}
Next we will apply Lempert's results L1-L4 on $\overline{A_{1,2}(0)}$.
We will consider $(\overline{A_{1,2}(0)},J_{Euc})$ as a standard annulus domain in $\CC^2$.
In the following paragraphs, each norm is defined based on the standard metrics, i.e.,
either the Euclidean metric or the Fubini-Study metric on the ``compactification''. 

We can compactify $\CC^2$ to $\PP^2$ by adding a divisor $D = \PP^1$ at the infinity analytically.
The standard complex structure $J_{Euc}$ on $\CC^2$ extends to the standard complex structure on $\PP^2$
which is denoted by $J_{std}$, and $\CC^2$ is embedded into $\PP^2$ by
$(z^1,z^2)\rightarrow (z^1,z^2,1)$.
Denote $W_r = \PP^2\setminus B_{r}(0)$.
By choosing $R$ to be large enough, we can assume $\epsilon(k)<\epsilon_1$,
where $\epsilon_1,k$ as in L.1.
Then for $i$ sufficiently large,
$\|J''_i-J_{Euc}\|_{C^{k}(\overline{A_{1,2}(0)})}< \epsilon_1$. 
By L.1, in the pseudoconcave manifold $W_2$,
there exists a $\Gamma'$-equivariant complex structure $J'''_i$ on $W_2$, such that 
$\|J'''_i-J_{std}\|_{C^{k'}(W_2)}<\epsilon_2$, $D$ is holomorphic with respect to $J'''_i$
and $J'''_i = J''_i$ as CR-structures on the boundary $S_2$.
Since $J''_i$ and $J'''_i$ are compatible on $S_2$, there exists a complex structure, denoted as $J_i$, on 
the pseudoconcave manifold $W_{1}$, such that,
$J_i = J'''_i$ on $W_2$, $J_i = J''_i$ on $\overline{A_{1,2}(0)}$, 
and $J_i$ is close to $J_{std}$ on $W_{1}$ under $C^{k}$-norm, and 
is $\Gamma'$-equivariant.

By L.4, we have $v_i = (v^1_i,v^2_i)$ on $(W_1\setminus D,J_i)$. Restrict $v_i$ on 
$\overline{A_{1,2}(0)}$, then we have a map
\begin{align}
{v}_i: \overline{A_{1,2}(0)} \rightarrow \RR^4
\end{align}
which is a diffeomorphism into its image, and
where ${v}^1_i,{v}^2_i$ are holomorphic functions with respect to $J''_i$, and there exists a small number $\lambda$ that depends on $\epsilon(k)$,
such that 
\begin{align}
\label{estimate_1}
\|v^j_i - z^j\|_{C^{k''}(\overline{A_{1,2}(0)})} < \lambda
\end{align}
Also by L.4, there exists a diffeomorphism defined by $u_i = (u^1_i,\ldots,u^N_i)$
\begin{align}
u_i:  \mathcal{N}  = (W_1\setminus D)/\Gamma'\rightarrow \CC^N
\end{align}
where $u^1_i,\ldots,u^N_i$ are holomorphic functions on $(\mathcal{N}, J_i)$, and 
there exists a small number $\lambda'>0$ that depends on $\epsilon(k)$, 
such that
\begin{align}
\label{estimate_2}
\|u^j_i - P_j(z^1,z^2)\|_{C^{k''}(\overline{A_{1,2}(0)})} < \lambda'.
\end{align}
The geodesic ball $B_{R}(y_i)$ can be attached to $\mathcal{N}$ analytically along the boundary $S_1/\Gamma'$.
Denote the glued manifold as $M_i$.
Since $M_i$ is one-convex,
each holomorphic function $u_i^j$ can be extended to a holomorphic function on $M_i$, which is still denoted as $u_i^j$.
Then $u_i = (u^1_i,\ldots,u^N_i)$ maps $M_i$ onto $Z\subset \CC^N$. 
Each holomorphic curve in $M_i$ is mapped to $\{0\}\in \CC^N$ for the reason given below.
Restrict $u_i$ on $\mathcal{N}$, it can be lifted up to a map on the universal cover
$u_i: W_1\setminus D\rightarrow \CC^N$,
which can be decomposed as
\begin{align} 
W_1\setminus D\xrightarrow{(v_i^1,v_i^2)}\CC^2 \xrightarrow{(P_1(z^1,z^2),\ldots,P_N(z^1,z^2)))} \CC^N
\end{align} 
where $P_1,\ldots,P_N$ are homogeneous polynomials as in L.3 and
$\{0\}\in \CC^2$ is mapped to the vertex of the cone by the latter map. Then the singularity point of 
$u_i(M_i)$ in $\CC^N$ is $\{0\}$,
and holomorphic curves are mapped to $\{0\}\in \CC^N$ by $u_i$.

When $i\to \infty$, up to a subsequence, $v^j_i$ $(j=1,2)$ converges to $v^j_\infty$,
and 
\begin{align}
v_\infty = (v^1_\infty,v^2_\infty): \overline{A_{1,2}(0)} \rightarrow \CC^2
\end{align}
is an embedding, and is holomorphic with respect to $J''_\infty$.
This implies that the inner boundary $S_1$ with CR-structure induced by $J''_\infty$ is embeddable.

Now we will construct holomorphic coordinate functions on $\overline{A_{1,\infty}(0)}$(as the universal cover of the ALE end of the limit space). 
Since $(\overline{A_{1,\infty}(0)},g''_\infty)$ 
has an ALE asymptotic rate of $O(r^{-\mu})$ for some $\mu>1$,
we can compactify $\overline{A_{1,\infty}(0)}$ analytically to a strictly pseudoconcave space $W_{1,\infty}$
by attaching a divisor $D\simeq \CC P^1$ to its end, and extend $J''_\infty$ to a complex structure on $W_{1,\infty}$
such that $D$ is holomorphic with respect to $J''_\infty$. 
By choosing the scaling factor $R$ sufficiently large,
we have $\|J''_\infty-J_{std}\|_{C^{k'}(W_{1,\infty})}<\epsilon_2$.
Since $(S_1,J''_\infty)$ (as the boundary of $W_{1,\infty}$) is embeddable as shown above,
then by applying Lempert's result L.2, L.4, there exists a pair of holomorphic functions $(w^1_\infty,w^2_\infty)$
on $A_{1,\infty}\simeq W_{1,\infty}\setminus D$,
which induces an embedding
\begin{align}
w_\infty = (w^1_\infty,w^2_\infty): A_{1,\infty}\rightarrow \CC^2
\end{align}
Then $(w^1_\infty,w^2_\infty)$ is a pair of coordinate function on the universal cover of the end of the limit space.
Thus $Y_\infty$ is birationally equivalent to $\CC^2/\Gamma'$.

 Smooth energy concentration points can be ruled out using the same argument in the proof of Corollary \ref{nosmooth}.
\end{proof}

\subsection{Each deeper bubble is a resolution}
\label{deeper_bubble}
We are going to apply an induction argument to show that each deeper bubble is a resolution to the corresponding
singularity in the previous bubble.
By Lemma \ref{Lempert_method}, the geodesic ball $B_R(y_i)$ is birational to an open neighborhood of $y_\infty\in Y_\infty$.
As in the proof of Lemma \ref{Lempert_method}, $A_{1,2}(0)$ (associated with the complex structure
$(\psi_i'\circ\Psi^{-1})^*J'_i$) is a subset of $\mathcal{N}$. By L.4, 
$u_i$ maps $\mathcal{N}$ to a subset of the cone $Z\subset \CC^N$ .
Recall that in the proof of Lemma \ref{Lempert_method}, we can obtain a manifold $M_i$ by attaching $B_{R}(y_i)$
to $\mathcal{N}$ analytically. Each holomorphic function $u^j_i$ extends over $M_i$ by one-convex property.
Then the map $u_i: \mathcal{N}\rightarrow Z$ can be extended to:
\begin{align}
\pi_{i}: M_i \rightarrow Z 
\end{align}
Since $u^j_i$ converges and extends to a holomorphic function $u^j_\infty$ on $Y_\infty$ for each $1\leq j\leq N$,
there exists a map:
\begin{align}
\pi_{\infty}: Y_\infty\xrightarrow{(u^1_\infty,\ldots,u^N_\infty)} Z 
\end{align}
where $\pi_{i},\pi_{\infty}$ are surjective holomorphic maps that contract the holomorphic curves.
Let $\widetilde{Y_\infty}$ be the minimal resolution of $Y_\infty$, with the projection map
\begin{align} 
\pi: \widetilde{Y_\infty}\rightarrow Y_\infty
\end{align} 
Following the same argument that proves \eqref{fcond},
for each $i$, there exists a holomorphic map
\begin{align} 
f_i: \widetilde{Y_\infty}\rightarrow M_i
\end{align} 
which is surjective to its image, and such that 
$\pi_i\circ f_i \circ (\pi\circ\pi_\infty)^{-1} = Id$
on the subset of $Z$ where it is defined.
Define
\begin{align}
\tau_i = \pi\circ f_i^{-1} \circ \psi_i': A_{\delta, R}(y_\infty) \rightarrow A_{\delta,2R}(y_\infty)
\end{align}
By a similar procedure as we did in the proof of Proposition \ref{CGCB},
we can show that $\tau_i$ converges to the identity map from $A_{\delta,R}$ to itself.
Henceforth, we can show that there exists no $(-1)$-curve in $\widetilde{B_R(y_\infty)}$,
and there exists a surjective bimeromophism from $B_R(y_i)$ to its image in $B_{2R}(y_\infty)$.
Furthermore, this implies that, for a sufficiently small $\delta>0$,
$B_\delta(y_\infty)$ is isomorphic to a neighborhood of the singularity in $\CC^2/\Gamma''$, where $\CC^2/\Gamma''$ is type of 
the quotient singularity at $y_\infty$.
Then we can continue our iteration step, and analyze the next bubble as we did for the first one.
Since for each step, the energy $\|Rm\|^2_{L^2}$ loses a definite value which is $\geq \frac{\epsilon_0}{2}$,
where $\epsilon_0$ is the energy threshold, the iteration could last for at most finite steps.
By doing the induction after finite steps, we can show that each bubble is a resolution to the corresponding
singularity in the previous bubble. Finally, exactly as in the previous steps, there are no smooth energy concentration points at any stage in the bubble tree.

\subsection{Completion of proof of Theorem \ref{compactness}: ruling out bubbling}
\label{cpsubsec}
Since each bubble is a resolution, the bubble tree is diffeomorphic to a sequence of resolutions.
A priori, the bubble tree could have more than one branch. But without the loss of generality, we can assume that
the bubble tree has only one branch, and is diffeomorphic to $Y_1 \# Y_2 \#\ldots \# Y_r$, where
$Y_1$ is the first bubble, $Y_r$ is the deepest bubble, each bubble $Y_{j+1}$ is a resolution of the corresponding singularity
in $Y_j$. 
Since $Y_r$ is smooth and is a resolution, and $b_2(Y_r)$ is nontrivial, there exists a holomorphic curve
$E^r\subset Y_r$. By Laufer's Theorem 2.1, \cite{Laufer1979}, $E^r$ is homologous to a positive cycle $E^{r-1}$ in $\widetilde{Y_{r-1}}$.
Since $\widetilde{Y_{r-1}}$ is a resolution of the singularity in $Y_{r-2}$, $E^{r-1}$ is again homologous to a 
positive cycle $E^{r-2}$ in $\widetilde{Y_{r-2}}$. By induction, finally, $E^r$ is homologous to a nontrivial
positive cycle $E^1$ in $\widetilde{Y_1}$.
Then there exists a rational combination $[E^1] = a_1 [E'_{1}]+\ldots+ a_m [E'_{m}]$ that converges to $[E^1]$, where 
$a_j$ are non-negative rational numbers with at least one larger than $0$,
$E'_{1},\ldots,E'_{m}\subset \Phi^{-1}(x_\infty)$. 
However, by the assumption, $\int_{E^1} \omega_i^2 > C >0$. This implies a contradiction.

Recalling Corollary \ref{nosmooth}, there can be no energy concentration points in the limit, so $X_{\infty}$ must be a smooth manifold, and there exist diffeomorphisms 
\begin{align}
\psi_i: X_\infty\rightarrow X_i
\end{align}
such that $\psi_i^*g_i\xrightarrow{C^{k,\alpha}_{-\mu}} g_\infty$, $\psi_i^*J \xrightarrow{C^{k,\alpha}_{-\mu}} J_\infty$,
where $-2<-\mu<-1$, $k$ is any non-negative integer, $0<\alpha<1$.
Since $X_\infty$ is biholomorphic to $X$, the gauging map $\Phi$ in Theorem \ref{CGCB} can be considered as
an automorphism of $X$, which preserves the rate of ALE coordinate.
By the proof of Lemma \ref{uniform_diffeom}, 
away from a compact subset of $X$, the diffeomorphism $\psi_i$ is constructed by using harmonic coordinates,
and the convergence in Theorem \ref{CGCB} can be improved to
$\|\psi_i - Id\|_{C^{k+1,\alpha}_{-\mu+1}} < \epsilon(i\,|\, k)$.
Then $g_i$ converges to $g_\infty$ in $C^{k,\alpha}_{-\mu}(g_\infty)$-norm.
Without the loss of generality, we can choose $-\mu<\delta_0$.
Then $g_\infty$ is also an ALE metric with respect to the fixed ALE coordinate $\Psi$ of rate $O(r^{\delta_0})$. By a standard bootstrapping argument, $\omega_{\infty}
\in \mathcal{P}(X,J,\omega_0, \delta_0)$, 
and this finishes the proof of Theorem \ref{compactness}.

\section{Existence results}
\label{non-Artin}
In this section, we prove Corollary~\ref{corexist},
Theorem~\ref{background_metric},  
Theorem~\ref{classification}, and Corollary \ref{cyclicthm1} 

\subsection{Proof of Corollary~\ref{corexist}}
For any K\"ahler class $\kappa\in \KK(J)$, let 
$g_{b,1} \in \mathcal{P}(J)$ with $[g_{b,1}] = \kappa$. 
Consider the family
of background ALE K\"ahler metrics $g_{b,t} = (1-t)g_0+tg_{b,1}$ for $t\in [0,1]$.
We want to construct a family of ALE SFK metrics $g_t$ for $t\in[0,1]$, 
and $[g_t] = [g_{b,t}]$,
with $g_t - g_0 \in C^{k,\alpha}_{\delta}(g_0)$. 
Let $S\subset [0,1]$ be the subset where such ALE SFK metric exists. 
By the openness result in \cite{HanViaclovsky}, $S$ is an open subset. By Theorem~\ref{compactness}, $S$ is closed, so $S = [0,1]$ and 
the desired ALE SFK metric exists, which completes the proof. 

\subsection{Theorem~\ref{background_metric}: construction of background ALE K\"ahler metrics}
\label{setup}
Let $(X,J_1)$ be a complex surface, where $J_1\in \JJ^M_k$ ($k>0$).  Our goal is to 
construct an ALE K\"ahler metric $g_1$ on $(X,J_1)$.
Outside of a compact subset $K\subset X$, $X\setminus K$ has a universal cover $\overline{X\setminus K}$,
which can be compactified analytically to an open surface $S$ by attaching a divisor $D\simeq \PP^1$ to its end.
By Pinkham \cite{Pinkham78},
the surface $(S,J_1)$ is a deformation of $(S,J_{std})$ (which is a subset in $\PP^2$),
and the deformation fixes the divisor $D$.
The $k$th-order formal infinitesimal neighborhood of $D$ is defined as $\oo_S^{(k)}=\oo_S/\mathcal{I}^k$, where $\mathcal{I}$
is the ideal sheaf of $D$.
By \cite{Pinkham78}, we know that
$D$ has the same first-order infinitesimal neighborhood in $(S,J_1)$ and $(S,J_{std})$,
i.e., $\oo_S^{(1)}$ is identical with respect to different complex structures.
(Indeed, $\oo_S^{(3)}$ is identical with respect to different complex structures.)
The divisor $D$ is associated with a line bundle $L$ over $(S,J_{std})$,
and a line bundle $L'$ over $(S,J_1)$.
There exists a defining section of $D$ $\sigma_0 \in H^0(S, \oo(L'))$ with $\sigma_0|_D = 0$,
and smooth sections $\zeta_1,\zeta_2\in \Gamma(S,\mathcal{A}(L'))$, of which the restriction of
$\zeta_1,\zeta_2$ on $D$ are generators of $H^0(D,\oo(L))$.
We can use $(\sigma_0,\zeta_1,\zeta_2)$ to map $S$ into $\PP^2$,
and denote the pull-back of the complex structure on $\PP^2$ by $J_{std}$ on $S$.
Since $\bar\partial \zeta_{j} = O(|\sigma_0|)$ for $j=1,2$, where $\bar\partial$ is with respect to $J_1$,
this implies that
\begin{align}
\label{asymptotic_rate_J}
J_1 \sim J_{std} + O(|\sigma_0|)
\end{align}
The functions $\frac{\zeta_1}{\sigma_0},\frac{\zeta_2}{\sigma_0}$ are well-defined smooth functions on $S\setminus D$.
We use 
\begin{align}
x^1 = Re\Big(\frac{\zeta_1}{\sigma_0}\Big), 
x^2 = Im\Big(\frac{\zeta_1}{\sigma_0}\Big),
x^3 = Re\Big(\frac{\zeta_2}{\sigma_0}\Big),
x^4 = Im\Big(\frac{\zeta_2}{\sigma_0}\Big)
\end{align}
as coordinate functions of $\overline{X\setminus K}$.
Be aware that $(x^1+\sqrt{-1}x^2,x^3+\sqrt{-1}x^4)$ are holomorphic functions with respect to $J_{std}$.
Then
\begin{align}
\omega_{Euc} = \frac{\sqrt{-1}}{2} \partial_{std}\bar\partial_{std}
\Big(|\frac{\zeta_1}{\sigma_0}|^2+|\frac{\zeta_2}{\sigma_0}|^2 \Big)
\end{align}
defines a positive $(1,1)$-form on $(S\setminus D,J_{std})$, which is the K\"ahler form associated to the Euclidean metric
under the coordinate $(x^1,x^2,x^3,x^4)$.

Moreover by \eqref{asymptotic_rate_J},
\begin{align}
\bar\partial - \bar\partial_{std} = O(|x|^{-1})
\end{align}
Then
\begin{align}
\frac{\sqrt{-1}}{2}\partial\bar\partial \Big(|\frac{\zeta_1}{\sigma_0}|^2 + |\frac{\zeta_2}{\sigma_0}|^2\Big) = \omega_{Euc} + O(|x|^{-1})
\end{align}
By taking $|x|$ sufficiently large, we can assume 
$\frac{\sqrt{-1}}{2}\partial\bar\partial (|\frac{\zeta_1}{\sigma_0}|^2 + |\frac{\zeta_2}{\sigma_0}|^2) $
is positive definite, therefore a K\"ahler form.
Averaging with the $\Gamma'$-action, we can assume 
$|\frac{\zeta_1}{\sigma_0}|^2+|\frac{\zeta_2}{\sigma_0}|^2$ is $\Gamma'$-invariant,
and can be pushed down to $X\setminus K$.

After contracting all exceptional divisors on $X$, there exists a Stein space $X'$.
Without loss of generality, assume $p\in X'$ is the only singular point. 
We will also identify $X$ with $X'$ away from the
exceptional divisors and $p$.
Furthermore, there exists an integer $k'>0$, such that ${L'}^{k'}$ 
can be extended to a line bundle on the analytic compactification
$\hat{X}$ (which is an orbifold), $\oo({L'}^{k'})$ is globally generated, and there exists a
basis $s_0,\ldots,s_N\in H^0(\hat{X}, \oo({L'}^{k'}))$ which embeds $X'$ into $\CC^N$.
We have
\begin{align}
\varphi = (1 + |u^1|^2+\ldots+|u^N|^2)^\alpha,
\end{align}
where $u^j = s_j/s_0$, $0<\alpha<1$, and $\varphi$ is a strictly pluri-subharmonic function on $X'\setminus p$.

Let $K'\subset X$ be a compact subset and $K\subset K'$. 
Let $\chi$ be a smooth cutoff function defined on $X$, such that $\chi = 0$ on $K$, and $\chi = 1$ on $X\setminus K'$.
Define the $(1,1)$-form $\omega'_1$ as:
\begin{align}
\label{omega'_1}
\omega'_1 =  
A\cdot \sqrt{-1}\partial\bar\partial \varphi + \frac{\sqrt{-1}}{2}\partial\bar\partial\Big(\chi\cdot \Big(\Big|\frac{\zeta_1}{\sigma_0}\Big|^2
+\Big|\frac{\zeta_2}{\sigma_0}\Big|^2\Big)\Big) 
\end{align}
By choosing $A$ to be sufficiently large, $\omega'_1$ is positive definite on $X'\setminus \{p\}$.
By choosing $0<\alpha<1$ to be sufficiently small, $\omega'_1$ is an ALE K\"ahler form with asymptotic rate of at least $O(r^{-\nu})$,
for any $0<\nu<1$,
with respect to the coordinate: 
\begin{align}
\Psi: \overline{X\setminus K} \xrightarrow{(x^1,x^2,x^3,x^4)} \RR^4
\end{align}
By using the gluing argument used in the proof of Lemma \ref{tblemma} locally near $p$, we can modify $\omega'_1$ to be an ALE K\"ahler metric
$\omega_1$ on $X$.

By \cite[(4.7)]{HanViaclovsky}, we have
\begin{align}
J_1 = J_{Euc} + Re(\phi) + Q,
\end{align}
where $\phi\in \Gamma(X,\Lambda^{0,1}\otimes T^{1,0})$ and satisfies the integrablity condition
$\bar\partial\phi + [\phi,\phi] = 0$, where 
$Q \sim \phi * \nabla \phi$ as $|\phi|\to 0$, 
and $\phi\sim O(r^{-\nu})$.
Noting that proof of \cite[Lemma 5.3]{HanViaclovsky} remains valid under the weaker 
assumption that $\delta < 0$,  we may use a sublinear growth vector field
$Y$ in that argument to assume that $\phi$ is divergence free, i.e., 
$\bar\partial^*\phi = 0$.
Then
$(\bar\partial^*\bar\partial+\bar\partial\bar\partial^*)\phi = \bar\partial^*[\phi,\phi] = O(r^{-3+\epsilon})$ for any small $\epsilon>0$.
By standard elliptic estimate, we have $\phi\sim O(r^{-2+\epsilon})$
and $J_1-J_{Euc} \sim O(r^{-\mu})$.
Furthermore, by formula \eqref{omega'_1}, the asymptotic rate of $g_1$ can be improved
to $O(r^{-\mu})$, $-2<-\mu<-1$. 
The argument above completes the proof of Theorem~\ref{background_metric}.

\begin{remark}
\label{KahlerStein}
When $(X,J)$ is a Stein ALE K\"ahler surface, then the K\"ahler cone $\KK(X,J)$ (see Definition \ref{kcdef})
is isomorphic to the entire space $H^2(X,\RR)$. This can be shown by the following. 
Let $\omega_0$ be the fixed background K\"ahler form. 
By weighted Hodge theory, any element in $H^2(X,\RR)$ can
be represented by a harmonic $(1,1)$-form $h = O(r^{-3})$, as $r \to \infty$. 
Clearly, $\omega_0+h$ is a positive $(1,1)$-form outside 
of a compact set. As mentioned above, the function
$\varphi = (1+ |u^1|^2+\ldots+|u^N|^2)^\alpha$ $(0<\alpha<1)$ is a strictly pluri-subharmonic function on $X$, since $(X,J)$ is assumed to be Stein.  Then there exists a constant $C>0$, such that
$\omega_h = \omega_0 + h + C\cdot \sqrt{-1}\partial\bar\partial \varphi$
is a K\"ahler form on $X$. We can choose $\alpha$ to be small enough, such that $\omega_h$ is an ALE K\"ahler metric
of order $O(r^{-\mu})$, $-2<-\mu<-1$.
\end{remark}

\subsection{Smoothing of the M-resolution}
In this subsection, we will construct a deformation 
which will be used in the proof of Theorem~\ref{SFK_non_Artin} below. 
Following the definition in Section~\ref{versal}, we have the deformation to the normal cone $\hat{\X}'\subset Proj(R[z])\times \CC$.
For $t \in \Delta^*$, the punctured unit disc in $\CC$, there is a simultaneous resolution of $\X'$,
$\hat{\X}\rightarrow \Delta^*$, and we identify $\hat{\X}_1$ with $\hat{X}$.
Then we can apply a $\CC^*$-action such that
\begin{align}
(s_0,\ldots,s_N)\rightarrow (t^{k'} s_0, s_1,\ldots, s_N)
\end{align}
which induces a map from $\hat{\X}'_t$ to $\hat{X}'_1$,  which can be lifted to a diffeomorphism:
$f_t:  \X_t\rightarrow \X_1$, which furthermore induces a sequence of ALE K\"ahler metrics:
\begin{align}
\label{pullback_and_rescale}
(\X_t, g_t, J_t) = (\X_t, |t|^2\cdot f_t^* g_1, f_t^* J_1).
\end{align}
Note that $(\X_t,J_t)$ extends to a deformation of complex structure, 
with central fiber isomorphic to $\CC^2/\Gamma$, i.e., 
$\CC^2/\Gamma \hookrightarrow \Y \rightarrow \Delta$. 
Without loss of generality, assume $\Y\rightarrow \Delta$ is in the versal deformation
of $\CC^2/\Gamma$.
Furthermore, as $t \rightarrow 0$, there are basepoints $x_t \in \X_t$
such that  $(\X_t,g_t,J_t,x_t)$ converges to $(\CC^2/\Gamma, g_{Euc}, J_{Euc}, 0)$
in the sense of pointed Cheeger-Gromov convergence with uniform ALE asymptotic rate.
After a base change
\begin{equation}
\begin{tikzcd}
\Y' \arrow{r}{} \arrow[swap]{d}{} & \Y \arrow[swap]{d}{} \\
\Delta \arrow{r}{t\to t^d} & \Delta
\end{tikzcd}
\end{equation}
we have a partial-resolution $\pi: \Z\rightarrow \Y$, such that the central fiber $\Z_0$ is a $M$-resolution,
and $\Z\rightarrow \Delta$ is a $\QQ$-Gorenstein deformation of Type $T$ singularities.

By assumption $\Z_0$ admits an orbifold ALE SFK metric $g_0$, with $\pi$ as the ALE coordinate,
and of ALE asymptotic rate $O(r^{-\mu})$. Without loss of generality, assume that there is
 only one orbifold point in $z_0 \in \Z_0$.
By the convergence above, for each $0<t\leq 1$, there exists a diffeomorphism
\begin{align}
\psi_t: \CC^2/\Gamma\setminus B_\eta(0) \rightarrow \X_t
\end{align}
where $B_\eta(0)$ is with respect to the Euclidean distance,
such that $\psi_t^* g_t$ converges to $g_{Euc}$ under $C^{k,\alpha}_{-\mu}(\CC^2/\Gamma\setminus B_\eta(0), g_{Euc})$-norm for any integer $k>0$,
and $0<\alpha<1$.
Let $U_\eta= \pi^{-1}(B_\eta(0))\subset \Z_0$ be the lifting of the unit ball of $\CC^2/\Gamma$.
Then the map $\psi_t$ can be lifted to a map
\begin{align}
\tilde{\psi}_t: \Z_0\setminus U_\eta\rightarrow \X_t
\end{align}
We can assume that $U_\eta$ is contained in the unit geodesic ball $B_1(z_0,g_0)$ in $\Z_0$.
We have
\begin{align}
\label{J}
\|\tilde{\psi}_t^* J_t - \tilde{J}_0\|_{C^{k,\alpha}_{- \mu} (g_0)}\sim O(|t|^d)
\end{align}
where the norm is taken on the domain on $\Z_0\setminus B_1(z_0)$ as $|t| \to 0$. 
This is because, the family $\Y\rightarrow \Delta$ is a deformation of ALE K\"ahler metrics.
By a standard argument (normalizing each annulus $A_{2^k,2^{k+1}}(z_0)$ to unit size), it is not hard to see
that along this deformation, away from the singularity, the complex structure has a convergence rate of $O(|t|^d\cdot r^{-\mu})$.
The power $d$ comes from the base change.
Exactly as is  \cite[Lemma 15]{BiquardRollin}, 
the estimate \eqref{J} will be needed below to control the perturbation of
the K\"ahler form and complex structure. Moreover, since our base space is non-compact, we also need to control the asymptotic behavior as $r\to\infty$.

\subsection{Smoothing of ALE SFK orbifold metrics}
\label{smoothing_method}
In \cite{BiquardRollin}, Biquard-Rollin use a gluing method to construct 
the smoothing of a CscK orbifold along a one-parameter non-degenerate $\QQ$-Gorenstein deformation. 
We will adapt their proof under the ALE setting, 
which will produce a family of ALE SFK metrics
that degenerate to an orbifold metric at the central fiber.

\begin{theorem}
\label{SFK_non_Artin}
Let $\Z\rightarrow \Delta$ be the $\QQ$-Gorenstein deformation  from above,
where the central fiber $\Z_0$ a $M$-resolution (or a $P$-resolution), and $p\in \Z_0$ is the only singularity in $\Z_0$, which is of type $T_0$ (of type $T$).
Assume there exists an ALE SFK orbifold metric $(\Z_0,J_0,g_0)$. 
Then along this deformation, there exists 
a smooth family of ALE SFK metrics $(\Z_t,J_t,g_t)$ of order $O(r^{-\mu})$ that degenerates to the orbifold metric $(\Z_0,J_0,g_0)$ as $t\to 0$.

\end{theorem}

\begin{proof}
Without loss of generality, 
assume that $t$ real, and let $\epsilon(t) = t^{\frac{d}{2}}$.
Denote $(A_s,g_{A_s})$ as 
a $\ZZ_n$-quotient of a $A_{n-1}$-type gravitational instanton $(\overline{A}_s,g_{\overline{A}_s})$ 
that associated to the type $T_0$ singularity $\{p\}\in B$ of the form $\frac{1}{n^2}(1,na-1)$.
For the family of gravitational instantons $(\overline{A}_s,g_{\overline{A}_s})$, Kronheimer's construction gives the expansion
\begin{align}
F^*_{\epsilon(s)}g_{\overline{A}_s} = g_{Euc} + \xi(s)
\end{align}
where $\xi(s) = O(s^2\cdot R^{-4})$, and
$F$ is a diffeomorphism from $\overline{A}_s$ to the minimal resolution of ${\CC^2/\frac{1}{n}(1,n-1)}$
(see more details in \cite{Kronheimer} and \cite[Section 2]{BiquardRollin}).
In the current setting, $s=t^d$.
The $C^{k,\alpha}_{-\mu}(A_t,g_{A_t})$-norm is defined as in Definition~\ref{weighted_norm} for the weighted H\"older norm
on ALE manifolds.
Let $U\subset \Z_0$ be an open neighborhood of $p$, which is isomorphic to an open neighborhood of $\{0\}\in \CC^2/\frac{1}{n^2}(1,na-1)$.
Let $r\in C^0(\Z_0)\cap C^\infty(\Z_0\setminus \{p\})$ be a function such that $r(p)$ 
is the Euclidean distance
to $p$ in $U$ and coincides with the radius of the ALE metric $g_0$ outside of a compact subset.
Define the weighted H\"older norm $C^{k,\alpha}_{-\mu}(\Z_0,g_0)$ as in Definition \ref{weighted_norm}, where $r$ is defined as above.
For any $u\in C^{k,\alpha}_{-\mu}(\Z_0,g_0)$, when $r\to 0$ or $r\to\infty$, $u = O(r^{-\mu})$.
We can define $C^{k,\alpha}_{-\mu}(\Z_t,g_t)$ in a similar way.

Define the gluing scale $b(t) = \epsilon(t)^\beta$, where $\beta = \frac{2}{2+\mu}$, 
$-2<-\mu<-1$ is the ALE asymptotic rate of the metric constructed such that $\beta$ is close to $\frac{1}{2}$.
Let $\chi: \RR^+\to\RR^+$ be a smooth nondecreasing  function 
\begin{align}
\chi(t) =
\begin{cases}
0 & t < 1/2\\
1 &  t \geq 1.\\
\end{cases}
\end{align}   
Let $H_{\epsilon^{-1}}$ be the homothety that identifies $b\leq r\leq 4b$ in $\Z_t$ with $\frac{b}{\epsilon}\leq R\leq \frac{4b}{\epsilon}$ in 
$A_t$. Attach $A_t$ and $\Z_t$ together by $H_{\epsilon^{-1}}$ to obtain a manifold $\X_t$, which is diffeomorphic to $\Z_t$. Define a Riemannian metric on $\X_t$
\begin{align}
{\tilde{h}_t} = 
\begin{cases}
\epsilon^2\cdot H_{\epsilon^{-1}}^*\big( g_{Euc}+\epsilon^{-2}\xi(\epsilon^2) \big) 
& r\leq b \\
g_t  &r\geq 4 b\\
\epsilon^2\cdot H_{\epsilon^{-1}}^*\big( g_{Euc}+(1-\chi(\frac{\epsilon}{b}R-1))\epsilon^{-2}\xi(\epsilon^2) \big)
& b\leq r\leq 2b.\\
\end{cases}
\end{align}
Define the Hermitian metric $h_t = \frac{1}{2}(\tilde{h}_t+\tilde{h}_t(J_t\cdot, J_t\cdot))$.
Note that as $\epsilon\to 0$, the limit of $g_{Euc}+\epsilon^{-2}\xi(\epsilon^2)$ is called the tangent graviton to the deformation
in \cite{BiquardRollin}.
The weighted H\"older norm $C^{k,\alpha}_{-\mu}(\X_t,h_t)$ can be defined 
by using $\chi$ to separate a function on $\X_t$ into functions supported separately 
on $A_t$ and $\Z_t$, and adding the corresponding norms together.
See more details in \cite[Section~3.3.3]{BiquardRollin}.
Denote $\overline{\omega}_t$ as the $(1,1)$-form corresponding to the Hermitian metric $h_t$.
By the same calculation as done in \cite[Section 3.4]{BiquardRollin}, when $\beta$ is close to $\frac{1}{2}$,
using \eqref{J}, it follows that  
\begin{align}
\|d \overline{\omega}_t\|_{C^{k,\alpha}_{-\mu}(\X_t,h_t)} &\leq C_k \cdot \epsilon^2\\
\|\nabla^{LC}J_t\|_{C^{k,\alpha}_{-\mu-1}(\X_t,h_t)} &\leq C_k\cdot \epsilon^2.
\end{align}
We next employ these estimates to perturb $h_t$ into to a K\"ahler metric.
As in  \cite[Section~3.5]{BiquardRollin}, there is a map of spaces of harmonic $(1,1)$-forms
\begin{align*}
H^{1,1}_{A_t}\oplus H^{1,1}_{\Z_0}\rightarrow K^{1,1}_t,
\end{align*}
where elements in  $K^{1,1}_t$ are very close to harmonic elements in 
$H^{1,1}(\X_t)$.
This implies an $L^2$-``almost orthogonal'' decomposition for
$2$-forms on $\X_t$.
The $\bar\partial_t$-Laplacian $\square_t = \bar\partial_t\bar\partial^*_t+\bar\partial_t^*\bar\partial_t$ is defined by using the background hermitian
form $\overline{\omega}_t$, which is a Fredholm operator with respect to the
$C^{k,\alpha}_{-\mu}$-norm.
Then $H^{1,1}(\X_t)$ is represented by $\bar\partial_t$-harmonic forms in $\H_{-\mu}(\X_t,\Lambda^{1,1})$.
Since $\X_t$ is K\"ahler outside of a compact subset, by a similar proof as in \cite[Proposition~3.5]{HanViaclovsky},
$\H_{-\mu}(\X_t,\Lambda^{1,1})\simeq \H_{-3}(\X_t,\Lambda^{1,1})$, so that the $L^2$-orthogonal decomposition still makes sense
under the ALE setting.

By the perturbation argument in \cite[Section 3.5]{BiquardRollin},  
there exists a $(1,1)$-form $\gamma_t$ which is ``almost orthogonal'' to $K^{1,1}_t$, 
such that $\overline{\omega}_t-\gamma_t$ is $\square_t$-closed, and 
\begin{align}
\|\gamma_t\|_{C^{k+2,\alpha}_{-\mu}(\X_t)}\leq C_k\cdot \|\square_t \overline{\omega}_t\|_{C^{k,\alpha}_{-\mu-2}(\X_t)}.
\end{align}
Exactly as in \cite[Lemma 26]{BiquardRollin}, $\overline{\omega}_t-\gamma_t$ can then 
be perturbed to a $d$-closed $(1,1)$-form, whose real part $\omega_{t'}$,
is a K\"ahler form. The adaptation of Biquard-Rollin's argument to the ALE 
case is entirely analogous to \cite[Section 7]{HanViaclovsky}.

By an implicit function type argument as in \cite[Section~4]{BiquardRollin} 
adapted to the ALE case in \cite[Section~8]{HanViaclovsky},
we can solve the equation $R(\omega_t) = 0$ $(t> 0)$ where each $\omega_t$ is a small perturbation of $\omega_{t}'$. It should be emphasized here that, in the compact case, there is an obstruction to the smoothing of
a CscK orbifold which is given by holomorphic vector fields on $\X_t$ for $t>0$ small. 
However, under the ALE setting, the scalar curvature defines a $4$th-order nonlinear PDE
\begin{align}
R: C^{k,\alpha}_a(\X_t) &\rightarrow  C^{k-4,\alpha}_{a-4}(\X_t)\\
\varphi &\rightarrow R(\omega_{b,t}+\sqrt{-1}\partial\bar\partial\varphi)
\end{align}
where $0<a,\alpha<1$, $k\geq 4$, $t>0$ is sufficiently small.
The cokernel of the linearization of $R$ corresponds to the space of decaying holomorphic vector fields on $\X_t$,
which is trivial as proved in \cite[Proposition~3.3]{HanViaclovsky}. As a result, there is no obstruction in the ALE case. 
We have therefore obtained a family of ALE SFK metrics $\omega_t$, which, by construction, degenerate to the original ALE SFK orbifold metric 
on the $M$-resolution as $t\to 0$.

\end{proof}
\begin{remark}\rm
In case of a $P$-resolution, for Theorem~\ref{SFK_non_Artin}, we require the direction of the deformation $\Z\rightarrow \Delta_\RR$ 
to be away from the discriminant locus (the subset of $\JJ^P_k$ where the Weyl group does not act freely.) See more details in
\cite{BiquardRollin}.
\end{remark}

\subsection{Completion of proof of Theorem \ref{classification}}
\label{completion}
For the proof of (a),
over the Artin component $\JJ_0$,
an initial ALE SFK metric $(X,J_0,g_0)$ on the minimal resolution of $\CC^2/\Gamma$
can be constructed by using \cite{CalderbankSinger} in the cyclic case, 
and \cite{LV14} in the general case. 
By \cite[Theorem~1.4]{HanViaclovsky}, there exists an open neighborhood of $J_0$ in $\JJ_0$, such that
for any complex structure $J$ in this open neighborhood, there exists an ALE SFK metric on $(X,J)$. 
We then apply the $\CC^*$-action on $\JJ_0$. As in~\eqref{pullback_and_rescale}, by the pull-back under the $\CC^*$-action,
and a rescaling of metrics such that the ALE coordinate is fixed, 
we can construct an ALE SFK metric in $\KK(J)$ for any $J$ in $\JJ_0$.

For the proof of (b), 
take $J\in \JJ^M_k$. By the assumption of (b), 
there exists an ALE SFK orbifold metric on the 
associated $M$-resolution $Z_k^M$. Then there exists an open
neighborhood $U\subset \JJ^M_k$ of $\Z_0$, such that for any complex structure 
$J\in U \setminus \{0\}$,
there exists a ALE SFK metric on $(X,J)$, by applying Theorem~\ref{SFK_non_Artin}.
By the pull-back of the $\CC^*$-action, and a rescaling of metrics to fix the ALE coordinate, 
we can also construct a ALE SFK metric for some K\"ahler class in $\KK(J)$,
for all $J\in{\JJ^M_k} \setminus \{0\}$.

For the proof of (c), 
denote ${\JJ^P_k}'\subset \JJ^P_k$ as the subset away from the discriminant locus, with ${\JJ^P_k}'$ is open and dense in $\JJ^P_k$.
Following exactly Case (b), we can construct an ALE SFK metric for some K\"ahler class in $\KK(J)$, for all $J\in{\JJ^P_k}'$.

\subsection{Proof of Corollary~\ref{cyclicthm1}}
The Artin component follows from Case (a) in Theorem~\ref{classification}. 
Next, assume $J \in \JJ^M_k$ with $k > 0$. We can obtain an ALE SFK orbifold metric
on the corresponding $M$-resolution $X_0$ using the Calderbank-Singer construction. 
To see this, notice that the $M$-resolution of $\CC^2/\Gamma$ is toric.
Let $\pi:\widetilde{X_0}\rightarrow X_0$ be its minimal resolution.
In the corresponding moment polygon of $\widetilde{X_0}$, each segment in the boundary represents an exceptional divisor in $\widetilde{X_0}$.
By using Joyce's construction as done in \cite{CalderbankSinger}, there exists a family of ALE SFK metrics on $\widetilde{X_0}$, which is 
parameterized by lengths of boundary segments. 
By decreasing the lengths of segments that correspond to the exceptional divisors contracted by $\pi$
to $0$, the Gromov-Hausdorff limit will be the desired ALE SFK orbifold metric on $X_0$.
Equivalently, these orbifolds can be directly constructed by choosing the lengths of the corresponding boundary 
segments to be exactly zero, in which case the Calderbank-Singer metrics are ALE SFK metrics with orbifold singularities. 
Corollary~\ref{cyclicthm1} is then a consequence of this observation and Case (b) in Theorem \ref{classification}. 

\section{Examples} 
\label{existsec}
In this section, we give the details of the examples in Subsection~\ref{introex} from the Introduction. Namely, we prove Theorems~\ref{theorem1} and~\ref{exthm2}. 
First we recall some important details of cyclic quotient singularities. 

\subsection{Cyclic quotient singularities}

Let $1 \leq q < p$ be relatively prime integers. For a type $\frac{1}{p}(1,q)$-action, let $\widetilde{X}$ be the minimal resolution 
of $\CC^2/ \Gamma(q,p)$.  
 Integers $k$ and $e_i, i = 1 \dots k,$ are defined by the following Hirzebruch-Jung modified Euclidean algorithm:
\begin{align}
\label{modified_EA}
p =e_1q-a_1, \ q =e_2a_1-a_2, \dots, a_{k-3} =e_{k-1}a_{k-2}-1, a_{k-2} = e_k a_{k-1} = e_k,
\end{align}
where the numbers $e_i \geq 2$ and $0 \leq a_i < a_{i-1}, \ i = 1 \dots k$, see \cite{Hirzebruch1953}.  The integer $k$ is called the {\textit{length}} of the modified Euclidean algorithm. 
 This can also be written as the continued fraction expansion
\begin{align}
\frac{q}{p} = \cfrac{1}{e_1 - \cfrac{1}{e_2 - \cdots \cfrac{1}{e_k}}} \equiv[e_1, e_2, \dots, e_k]. 
\end{align}

Recall that exceptional divsor in $\widetilde{X}$ is a string of rational curves, $E_i$ for $i = 1 \dots k$ with $E_i \cdot E_i = - e_i$, and each curve has intersection $+1$ with the adjacent curve, where it has a simple normal crossing singularity. This is represented by the following graph.

\vspace{2mm}
\setlength{\unitlength}{2cm}
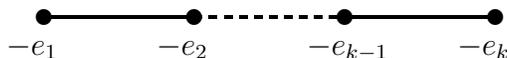
\begin{picture}(0,0)(-1.5,0)
\linethickness{.3mm}
\put(1,0){\line(1,0){1}}
\put(1,0){\circle*{.1}}
\put(.75,-.25){$-e_1$}
\put(2,0){\circle*{.1}}
\put(1.75, -.25){$-e_2$}
\multiput(2,0)(0.1,0){10}
{\line(1,0){0.05}}
\put(3,0){\circle*{.1}}
\put(2.75, -.25){$-e_{k-1}$}
\put(3,0){\line(1,0){1}}
\put(4,0){\circle*{.1}}
\put(3.75, -.25){$-e_k$}
\end{picture}

\vspace{7mm}

\noindent
which we will also denote as $(e_1, \dots, e_k)$. For details on cyclic quotient 
singularities see \cite{R1974}. 

For $\Gamma = \frac{1}{p}(1,q)$, the following formula is proved in \cite{AshiIshi,LockViaclovskyJEMS}
\begin{align}
\label{etacyclic}
\eta(S^3/\Gamma)=\frac{1}{3}\Big(\sum_{i=1}^ke_i+\frac{q^{-1;p}+q}{p}\Big)-k,
\end{align}
where the $e_i$ and $k$ are as defined in \eqref{modified_EA}, and $q^{-1;p}$ denotes the inverse of $q\text{ mod } p$.

\subsection{Artin component examples}
\label{ACE}
In these cases, we will next discuss the topological condition $\mathcal{C}(X) > 0$. 
First, we consider the case that $\Gamma \subset {\rm{SU}}(2)$, 
and $X$ is diffeomorphic to the minimal resolution of $\CC^2/ \Gamma$. 
In this case, we have equality in Nakajima's Hitchin-Thorpe inequality \cite{Nakajima}, 
so we have 
\begin{align}
2 \chi(X) + 3 \tau(X) = \frac{2}{|\Gamma|} + 3 \eta( S^3/\Gamma). 
\end{align}
The left hand side is equal to $2 - b_2(X)$, so we obtain
\begin{align}
\label{csu2}
\mathcal{C}(X) = \frac{4}{|\Gamma|}  > 0.
\end{align}

Next,  consider the cases in Theorem \ref{theorem1}. 
For the Artin component, if $\Gamma$ is cyclic it follows from \eqref{etacyclic}
that 
\begin{align}
\label{cxcyclic}
\begin{split}
\mathcal{C}(X) 
&= 2 - b_2(X) + \frac{2}{p} - 3  \eta\Big(\frac{1}{p}(1,q) \Big)
 = 2 - \sum_{i=1}^k(e_i-2) + \frac{2 - q^{-1;p}-q}{p}.
\end{split}
\end{align}

For $\Gamma = \frac{1}{3}(1,1)$, we have $p =3, q = 1, e_1 = 3, k = 1, 1^{-1;3} = 1$.
If $X$ is in the Artin component of $\Gamma$, then \eqref{cxcyclic} yields
$\mathcal{C}(X) = 1 > 0$. 

For $\Gamma = \frac{1}{5}(1,2)$, the dual graph is $(3,2)$, 
and we have $p =5, q =2, k = 2, 2^{-1;5} = 3$. If $X$ is in the Artin component of $\Gamma$, then \eqref{cxcyclic} yields $\mathcal{C}(X) = \frac{2}{5} > 0$. 

For $\Gamma = \frac{1}{7}(1,3)$, the dual graph is $(3,2,2)$, 
and we have $p =7, q =3, k = 3, 3^{-1;7} = 5$. If $X$ is in the Artin component of $\Gamma$, then \eqref{cxcyclic} yields
$\mathcal{C}(X) = \frac{1}{7} > 0$. 

Below, we will consider various non-Artin components of cyclic quotient 
singularities. For these, we will have $b_2(X) < k$. The 
modification to the formula for $\mathcal{C}(X)$ is simply the following
\begin{align}
\label{csnac}
\mathcal{C}(X) & = 2 + (k - b_2(X)) -  \sum_{i=1}^k(e_i-2) + \frac{2 - q^{-1;p}-q}{p}.
\end{align}

\subsection{Type T cyclic quotient singularities}
We recall the main definition from \cite{KollarShepherd}. 
\begin{definition}
If $\Gamma = \frac{1}{r^2 s} (1, r s d -1)$ where $r \geq 2, s \geq 1, (r,d) = 1$, 
then $\Gamma$ is said to be of type $T_{s-1}$. 
\end{definition}
We will also denote this action by $T(r,s,d)$. 
For type $T$ singularities, there exists non-Artin component such that the corresponding 
space $X$ satisfies $b_2(X) = s -1$. 
Note that this group is covered by the group $\tilde{\Gamma} = \frac{1}{r s} (1, r s -1)$, quotiented by a $\ZZ_r$-action. The spaces $X$ in the non-Artin component admit Ricci-flat metrics which are isometric quotients of an $A_{rs-1}$ hyperk\"ahler metric
\cite{Suvaina_ALE, Wright_ALE}.
We also note that the embedding dimension is $r +3$, and the base of the non-Artin component has dimension $s$ \cite{KollarShepherd, BehnkeChristophersen}. 
The following Proposition gives a useful description of the type T singularity in terms of their dual graphs.

\begin{proposition} \label{TTprop}
If $(e_1, \dots, e_k)$ 
is of type $T_{s-1}$,  then the graphs
$(2, e_1, e_2, \dots, e_{k-1}, e_k + 1)$ 
and 
$(e_1 +1, e_2, \dots, e_{k-1}, e_{k}, 2)$. are also of Type $T_{s-1}$. 
Type $T_0$ are those obtained starting from $(-4)$. 
Type $T_1$ are those obtained starting from $(3,3)$. 
In general, for $s > 2$, type $T_{s-1}$ are those obtained starting from 
$(3, \underbrace{2, \dots, 2}_{s-2}, 3)$ 
and iterating the above procedure $(r-2)$ times. 
\end{proposition}

Using this characterization, we can prove the following. 
\begin{proposition} 
\label{importprop}
Let $\Gamma$ be of type $T(r,s,d)$,  $\ell$ denote the total number of exceptional 
curves in the minimal resolution of $\CC^2/\Gamma$, and $-e_i$ denote the self-intersection number of the $i$th curve, $i = 1 \dots \ell$. Then 
\begin{align}
\ell &= r + s -2\\
\sum_{i =1}^{\ell} e_i &= 3r + 2 s - 4
\end{align}
Furthermore, we have
\begin{align}
\eta(\Gamma) &= \frac{1}{3} \Big( 3 - s - \frac{2}{r^2 s} \Big)\\
\mathcal{C}(X) & = \frac{4}{r^2 s}
\end{align}
\end{proposition}

\begin{proof}The first two formulas follow easily from the description in Proposition \ref{TTprop}. 
Without loss of generality, assume that $1 \leq d \leq r -1$. 
Then the inverse of $r s d - 1$ modulo $r^2s$ is given by $rs ( r -d) -1$.
To see this, 
\begin{align}
(rsd -1) ( rs (r-d) -1) -1 = - r^2 s (1 + d^2 s - rsd) \equiv 0 \mod r^2 s.
\end{align}
Therefore, letting $p = r^2 s$, and $q = r s d -1$, and using \eqref{etacyclic}, we have
\begin{align}
\begin{split}
\eta( \Gamma) &= \frac{1}{3}\Big(\sum_{i=1}^{\ell}e_i+\frac{q^{-1;p}+q}{p}\Big)-k
 =\frac{1}{3} \Big( 3 - s - \frac{2}{r^2 s} \Big).
\end{split}
\end{align}
Finally, by \eqref{csnac}, we have
\begin{align}
\mathcal{C}(X) &= 2 - (s - 1) + \frac{2}{r^2 s} - \Big( 3 - s - \frac{2}{r^2 s} \Big)  = \frac{4}{r^2 s}.
\end{align}
\end{proof}
\begin{remark} Note that $\mathcal{C}(X) = \frac{4}{|\Gamma|}$, something that we already knew 
had to be true from the Nakajima-Hitchin-Thorpe inequality, similarly to \eqref{csu2}. 
\end{remark}

\begin{remark} 
\label{revrem}
Without loss of generality, we can assume that $1 \leq d \leq r -1$. We showed above that 
\begin{align}
\frac{1}{r^2 s} (1, r s d -1) \sim \frac{1}{r^2s} (1, rs(r-d) -1).
\end{align}
This means that $T(r,s,d) \sim T(r,s, r-d)$ are equivalent singularities, but note that the ordering of the
self-intersection numbers $e_i$ is reversed in each case. 
\end{remark}
\subsection{Add a single $(-2)$-curve to a type $T$}
Given  $(e_1, \dots, e_k)$ of Type $T_{s-1}$, 
we consider the graph $(2, e_1, \dots, e_k)$.
Note that, we could also put the $(-2)$ curve on the right hand side. However, this would give an equivalent singularity taking the conjugate Type $T$ singularity (from Remark \ref{revrem}), which  
reverses the order of the self-intersection numbers, and still putting the $(-2)$ curve on the left. 
So let us write the type $T$ string as $T(r,s, r-d)$, and attach the $(-2)$ curve on the left. For this type $T$ singularity, we have 
\begin{align}
\frac{rs(r-d) -1 }{r^2 s} = [e_1, \dots, e_k].
\end{align}
So to determine what the new cyclic singularity is, we have
\begin{align}
\begin{split}
\frac{q}{p}  = [2, e_1, \dots, e_k]
=   \frac{1}{ 2 - \frac{rs(r-d) -1 }{r^2 s}} = \frac{r^2 s}{1 + d r s + r^2 s}.
\end{split}
\end{align}
So this singularity is of type $\frac{1}{1 + d r s + r^2 s} (1, r^2 s)$.
\begin{proposition}
\label{q1p1}
We have 
\begin{align}
q^{-1;p} &= d s r + d^2 s - 1\\
\eta\Big( \frac{1}{p}(1,q) \Big) &=\frac{1}{3}  \frac{ s ( -1 + d^2 + 2 d r + 2 r^2 - r(d+r)s}{1 + d r s + r^2 s}.
\end{align}
\end{proposition}
\begin{proof}
A simple computation shows that 
\begin{align}
r^2 s( d s r + d^2 s -1) - 1 = (-1 + d r s ) (1 + d r s + r^2 s).
\end{align}
Note also that $1 \leq d s r + d^2 s -1 < r^2s + d r s  + 1$. 

Next, using Proposition \ref{importprop} we have 
\begin{align}
\begin{split}
\eta( \Gamma) 
& =  \frac{1}{3}\Big( 3r + 2 s - 4 + \tilde{2}  +\frac{r^2 s  + d s r + d^2 s -1}{1 + d r s + r^2 s }\Big)- (r + s -2 + \tilde{1})\\
& =\frac{1}{3}  \frac{ s ( -1 + d^2 + 2 d r + 2 r^2 - r(d+r)s}{1 + d r s +  r^2 s}.
\end{split}
\end{align}
Note that the $\tilde{2}$ and $\tilde{1}$ terms are there because we added a single $(-2)$ curve.
\end{proof}

Next, we will blow-down the Type $T$ singularity, and let $X$ denote the corresponding $\QQ$-Gorenstein smoothing,
which exists by \cite{BehnkeChristophersen, KollarShepherd}.
\begin{proposition}
We have
\begin{align}
\mathcal{C}(X) = \frac{4 - d^2 s}{1 + d r s + r^2 s}.
\end{align}
\end{proposition}
\begin{proof} 
The $\eta$-invariant was determined in the previous proposition, since the group at infinity is the same. 
Note also that $b_2(X) = s -1 + 1 = s$, since the smoothing of the type $T$ singularity contributes $s-1$ and
the $(-2)$ curves donates another $1$ to this. 
We then have 
\begin{align}
\begin{split}
\mathcal{C}(X) &= 2 - s + \frac{2}{1 + d r s + r^2 s} - \frac{ s ( -1 + d^2 + 2 d r + 2 r^2 - r(d+r)s}{1 + d r s + r^2 s}
=  \frac{4 - d^2 s}{1 + d r s + r^2 s}.
\end{split}
\end{align}
\end{proof}
Clearly, for this to be positive, we require $d = 1$, in which case we have
\begin{align}
\mathcal{C}(X) = \frac{4 - s}{1 +  r s + r^2 s},
\end{align}
which is positive for $s = 1, 2, 3$. 
Note that from Proposition~\ref{q1p1}, the group at infinity is equivalent to 
\begin{align}
\Gamma = \frac{1}{1 + rs + r^2 s} (1, s(r+1) -1),
\end{align}
which yields the following.
\begin{theorem}
\label{xt1}
Let $\Gamma \subset {\rm{U}}(2)$ be any of the following groups for $r\geq 2$
\begin{align}
\tag{1} \Gamma &= \frac{1}{r^2+r+1}(1,r)\\
\tag{2} \Gamma &= \frac{1}{2 r^2+2r+1}(1,2r+1)\\
\tag{3} \Gamma &= \frac{1}{3r^2 + 3r + 1} (1, 3 r +2).
\end{align}
There is a non-Artin component $\JJ_i$ of the versal deformation space $\CC^2/\Gamma$ with $b_2 = i$ in Case $(i)$, $i = 1 ,2, 3$ which has $\mathcal{C}(\JJ_i) > 0$. 
\end{theorem}
Note the first case is a $M$-resolution, but the second and 
third cases are $P$-resolutions, but not $M$-resolutions. 
The dual graphs of the minimal resolutions in these cases look like the following.

For $s = 1, r \geq 2 : ( \overbrace{2, \dots, 2}^{r-1}, r+2)$.

For $s = 2, r \geq 2: ( \overbrace{2, \dots, 2}^{r-1}, 3, r+1)$.

For $s = 3, r \geq 2: ( \overbrace{2, \dots, 2}^{r-1}, 3, 2,  r+1)$.

\subsection{Add two $(-2)$-curves to a type $T$}
We will 
write the type $T$ string as $T(r,s, r-d)$, with dual graph $(e_1, \dots, e_k)$,  
and attach the two $(-2)$ curves on the left. 
To determine $p$ and $q$, we have
\begin{align}
\begin{split}
\frac{q}{p}  &= [2, 2, e_1, \dots, e_k]
=  \cfrac{1}{ 2 - \cfrac{1}{ 2 - \cfrac{rs(r-d) -1 }{r^2 s}}} = \frac{1 + d r s + r^2 s}{2 + 2 d r s + r^2 s}.
\end{split}
\end{align}
So this singularity is of type $\frac{1}{2 + 2d  r s + r^2 s} (1, 1 + d r s + r^2 s)$.
\begin{proposition}
\label{q1p2} 
We have 
\begin{align}
q^{-1;p} &= d s r + 2 d^2 s -1 \\
\eta\Big( \frac{1}{p}(1,q) \Big) & =\frac{1}{3}  \frac{ s ( -2 + 2 d^2 + 2 d r + r^2 - r(2 d+r)s)}{2 + 2 d r s +  r^2 s}.
\end{align}
\end{proposition}
\begin{proof}
A simple computation shows that 
\begin{align}
(1 + d r s + r^2 s)( d s r + 2 d^2 s -1) - 1 = (-1 + d^2 s + d r s ) (2 + 2 d r s + r^2 s).
\end{align}
Note also that $1 \leq d s r + 2 d^2 s -1 < r^2s + 2 d r s  + 2$. 

Next, using Proposition \ref{importprop}, we have 
\begin{align}
\begin{split}
\eta( \Gamma) 
& =  \frac{1}{3}\Big( 3r + 2 s - 4 + \tilde{4}  +\frac{r^2 s  + d r s + 1 + drs + 2 d^2 s -1 }{2 + 2 d r s + r^2 s }\Big)- (r + s -2 + \tilde{2})\\
& =\frac{1}{3}  \frac{ s ( -2 + 2 d^2 + 2 d r + r^2 - r(2 d+r)s)}{2 + 2 d r s +  r^2 s}.
\end{split}
\end{align}
Note that the $\tilde{4}$ and $\tilde{2}$ terms are there because we added a two $(-2)$ curves.
\end{proof}

Next, we will blow-down the Type $T$ singularity, and let $X$ denote the corresponding $\QQ$-Gorenstein smoothing,
which exists by \cite{BehnkeChristophersen, KollarShepherd}.
\begin{proposition}
We have
\begin{align}
\mathcal{C}(X) = \frac{4 - 2 d^2 s}{2 + 2 d r s + r^2 s}.
\end{align}
\end{proposition}
\begin{proof} 
The $\eta$-invariant was determined in Proposition \ref{q1p2}, since the group at infinity is the same.
Also, $b_2(X) = s -1 + 2 = s +1 $, since the smoothing of the type $T$ singularity contributes $s-1$ and
the $(-2)$ curves donate another $2$ to this. Then 
\begin{align}
\begin{split}
\mathcal{C}(X) &= 2 - (s + 1) + \frac{2}{2 + 2 d r s + r^2 s} 
-  \frac{ s ( -2 + 2 d^2 + 2 d r + r^2 - r(2 d+r)s)}{2 + 2 d r s +  r^2 s}\\
& =  \frac{4 - 2 d^2 s}{2 + 2 d r s + r^2 s}.
\end{split}
\end{align}
\end{proof}
Clearly, for this to be positive, we require $d = 1$, in which case we have
\begin{align}
\mathcal{C}(X) = \frac{4 - 2 s}{2 + 2 r s + r^2 s},
\end{align}
which is only positive for $s = 1$. 
Also, by Proposition \ref{q1p2}, the group at infinity is 
\begin{align}
\Gamma = \frac{1}{2  + 2 r + r^2 s} (1, r+1  ),
\end{align}
which yields the following.
\begin{theorem}
\label{xt2}
Let $\Gamma \subset {\rm{U}}(2)$ be any of the following groups for $r\geq 2$
\begin{align}
 \Gamma = \frac{1}{r^2+2r+2}(1,r+1).
\end{align}
Then there is a non-Artin component $\JJ_k$ of the versal deformation space $\CC^2/\Gamma$ with $b_2 = 2$ 
which has $\mathcal{C}(\JJ_k) > 0$. 
\end{theorem}
The dual graph of the minimal resolution of the $M$-resolution in these cases looks like the following. 

For $r \geq 2 : ( \overbrace{2, \dots, 2}^{r}, r+2)$.

\subsection{Add three $(-2)$-curves to a type $T$}

We write the type $T$ string as $T(r,s, r-d)$, with dual graph $(e_1, \dots, e_k)$, 
and attach the three $(-2)$ curves on the left. 
To determine $p$ and $q$ we have
\begin{align}
\begin{split}
\frac{q}{p}  &= [2,2,2,e_1, \dots, e_k]
=   \cfrac{1}{ 2 - \cfrac{1}{ 2 - \cfrac{1}{2 - \frac{rs(r-d) -1 }{r^2 s}}}} = \frac{2 + 2 d r s + r^2 s}{3 + 3 d r s + r^2 s}.
\end{split}
\end{align}
So this singularity is of type $\frac{1}{3 + 3 d  r s + r^2 s} (1, 2 + 2 d r s + r^2 s)$.
\begin{proposition}
\label{q1p3} 
We have 
\begin{align}
q^{-1;p} &= d s r + 3 d^2 s -1\\
\eta\Big( \frac{1}{p}(1,q) \Big) & =\frac{1}{3}  \frac{ - 2 - s( 3 -3 d^2  + r(3 d+r)s)}{3 + 3 d r s +  r^2 s}.
\end{align}
\end{proposition}
\begin{proof}
A simple computation shows that 
\begin{align}
(2 + 2 d r s + r^2 s)( d s r + 3 d^2 s -1) - 1 = (-1 + 2 d^2 s + d r s ) (3 + 3 d r s + r^2 s).
\end{align}
Note also that $1 \leq d s r + 3 d^2 s -1 < r^2s + 3 d r s  + 3$. 

Next, using Proposition \ref{importprop}, we have
\begin{align}
\begin{split}
\eta( \Gamma) 
& =  \frac{1}{3}\Big( 3r + 2 s - 4 + \tilde{6}  +\frac{ 2 + 2 d r s + r^2 s + dsr + 3d^2s -1}{3 + 3 d r s + r^2 s }\Big)- (r + s -2 + \tilde{3})\\
& =\frac{1}{3}  \frac{ - 2 - s( 3 -3 d^2  + r(3 d+r)s)}{3 + 3 d r s +  r^2 s}.
\end{split}
\end{align}
Note that the $\tilde{6}$ and $\tilde{3}$ terms are there because we added a three $(-2)$ curves.
\end{proof}

Next, we will blow-down the Type $T$ singularity, and let $X$ denote the corresponding $\QQ$-Gorenstein smoothing,
which exists by  \cite{BehnkeChristophersen, KollarShepherd}.
\begin{proposition}
We have
\begin{align}
\mathcal{C}(X) = \frac{4 - 2 d^2 s}{2 + 2 d r s + r^2 s}.
\end{align}
\end{proposition}
\begin{proof} 
The $\eta$-invariant term was determined in Proposition \ref{q1p3}, since the group at infinity is the same.
Also, $b_2(X) = s -1 + 3 = s +2 $, since the smoothing of the type $T$ singularity contributes $s-1$ and
the $(-2)$ curves donate another $3$ to this. Then 
\begin{align}
\begin{split}
\mathcal{C}(X) &= 2 - (s + 2) + \frac{2}{3 + 3 d r s + r^2 s} 
-   \frac{ - 2 - s( 3 -3 d^2  + r(3 d+r)s)}{3 + 3 d r s +  r^2 s}\\
& =  \frac{4 - 3 d^2 s}{3 + 3 d r s + r^2 s}.
\end{split}
\end{align}
\end{proof}
Clearly, for this to be positive, we require $d = 1$, in which case we have
\begin{align}
\mathcal{C}(X) = \frac{4 - 3 s}{3 + 3 r s + r^2 s},
\end{align}
which is only positive for $s = 1$. 
By Proposition \ref{q1p3}, the group at infinity is equivalent to  
\begin{align}
\Gamma = \frac{1}{3  + 3 r + r^2 } (1, r+2  ),
\end{align}
which yields the following.
\begin{theorem}
\label{xt3}
Let $\Gamma \subset {\rm{U}}(2)$ be any of the following groups for $r\geq 2$
\begin{align}
 \Gamma = \frac{1}{r^2+3r+3}(1,r+2).
\end{align}
Then there is a non-Artin component $\JJ_k$ of the versal deformation space $\CC^2/\Gamma$ with $b_2 = 3$ 
which has $\mathcal{C}(\JJ_k) > 0$. 
\end{theorem}
The dual graph of the minimal resolution of the $M$-resolution in these cases is the following.  

For $r \geq 2 : ( \overbrace{2,\dots,2}^{r+1}, r+2)$. 

\subsection{Completion of proof of Theorems \ref{theorem1} and \ref{exthm2}}
\label{cop7}
All of the groups in Theorems \ref{theorem1} and \ref{exthm2} are cyclic 
groups. By Corollary \ref{cyclicthm1}, there exists an ALE SFK 
metric in some K\"ahler class, for any $J \in \JJ^M(i)$ away from the central fiber. 
By Subsection \ref{ACE}, and Theorems \ref{xt1}, \ref{xt2}, and \ref{xt3}, 
all cases in Theorems~\ref{theorem1} and~\ref{exthm2} satisfy
$\mathcal{C}(\JJ^M(i))>0$. By Subsection~\ref{volume}, assumption \ref{Vassump} is 
satisfied. By Corollary~\ref{corexist}, it follows that there exists an
ALE SFK metric in any K\"ahler class. 

\section{Conclusion}
\label{conclusion}
In this section, we give a family of examples which shows that smoothings of 
{\textit{non-minimal}} orbifolds can occur as limits of {\textit{minimal}} ALE scalar-flat K\"ahler surfaces. 
In particular, the moduli space of SFK ALE metrics exhibits new phenomena which do not occur in the hyperk\"ahler case $\Gamma \subset {\rm{SU}(2)}$. 

\begin{theorem} 
\label{non_minimal_examples}
There exists sequences $g_i$ of SFK ALE metrics on $\oo_{\CC P^1}(-n)$
with respect to complex structures $J_i$ in the Artin component 
of $\CC^2/ \Gamma$, where $\Gamma = \frac{1}{n}(1,1)$,  
such that 
\begin{align}
(\oo_{\CC P^1}(-n), g_i, J_i,x_i) \rightarrow (X_{\infty}, g_{\infty}, J_{\infty}, x_{\infty})
\end{align}
in the pointed Cheeger-Gromov sense to a limiting SFK ALE orbifold 
$(X_{\infty}, g_{\infty}, J_{\infty})$ such that the limit  $(X_{\infty}, J_{\infty})$ is 
birational to $(\CC^2/ \ZZ_n, J_{euc})$, but is not dominated by the minimal resolution. 
\end{theorem}

\begin{proof}
For $n \geq 3$, take $\mathcal{O}(-n)$, perform the interated blowup which obtained from $(n-2)$ blow-ups starting on the $(-n)$-curve then blow-down all curves except for the $(-1)$-curve on the end, which yields a type $T_0$ singularity. The dual graphs are as follows. 

\vspace{2mm}
For $n =3 : (-1,-4)$.

\vspace{1mm}

For $n = 4: (-1,-2,-5)$. 

\vspace{1mm}

For $n \geq 5: (-1, \overbrace{2, \dots, 2}^{n-3}, n+1)$. 

\vspace{2mm}

For each $n\geq 3$, denote the blow-down space with a type $T_0$ singularity as $Z$. 
Notice that $Z$ is not an $M$-resolution. However, we will show next that the smoothing of the type $T_0$ singularity is unobstructed. 
The smoothing has $b_2 = 1$, and must lie in the Artin component. This is because
there are no non-Artin components for $n \neq 4$, and for $n=4$, the non-Artin component has $b_2 = 0$. 

Note that $Z$ is obtained by blow-ups of $\oo_{\CC P^1}(-n)$, and then blow-downs.  
Since $\oo_{\CC P^1}(-n)$ is toric, and each blow-up 
is at a point fixed by the torus action, it follows that $Z$ is toric. 
As in Subsection \ref{cop7}, by Calderbank-Singer's construction,
there exists a SFK ALE orbifold metric $g_0$ on $Z$. 
We will to apply the smoothing construction as we did in Section \ref{smoothing_method} to find the desired smooth SFK ALE metrics near this orbifold metric.

First we want to show that there is no local-to-global obstruction for the deformation of the quotient singularity.
Let $X = Z\cup D$ be the analytic compactification of $Z$, where $D$ is a $(+n)$-curve.
We want to smooth out the type $T_0$ singularity in $X$ while fixing the divisor $D$.
Denote $T_X = \sheafhom_{\oo_X}(\Omega^1,\oo_X)$ as the dual sheaf of the $(1,0)$-form sheaf on $X$, and denote
$T_X(-log(D))$ as the subsheaf of $TX$ where near each point of $D$, $T_X(-log(D))$ is generated by $(1,0)$-vectors tangent
to $D$. 
We have the following exact sequence
\begin{align}
\begin{split}
H^1(X,T_X(-log(D))) &\rightarrow Ext(\Omega^1(log(D)),\oo_X)\rightarrow \\
 & H^0(X,\sheafext^1_{\oo_X}(\Omega^1(log(D)),\oo_X))
\rightarrow H^2(X, T_X(-log(D)))
\end{split}
\end{align}
Following the proof of \cite[Theorem~2]{LeePark}, 
the obstruction to the deformation we want lies in $H^2(X,T_X(-log(D)))\simeq H^2(\widetilde{X},T_{\widetilde{X}}(-log(D+E)))
\simeq H^0(\widetilde{X},K_{\widetilde{X}}\otimes \Omega^1_{\widetilde{X}}(log(D+E)))$, 
where $\widetilde{X}$ is the minimal resolution of $X$,
$E = \cup_{j=0}^{n-3} E_j$ is union of the exception divisors resolved from the $T_0$-singularity, 
and the last isomorphism is due to Serre duality. The $E_j$ is ordered from the right to the left in the graph above, with 
$E_0\cdot E_0 = -(n+1), E_j\cdot E_j = -2$ for $1\leq j\leq n-3$. 
Note that $\widetilde{X}$ is obtained by blow-ups of the Hirzebruch surface $F_n$. Denote $F$ as the generic fiber, and $E'$ as the $(-1)$-curve
in the dual graph above. 
The canonical divisor can be represented as 
$K_{\widetilde{X}} = (n-2)F -2D + \sum_{j=1}^{n-3}(j E_j) + (n-2)E'$, and the divisor
$D = nF + \sum_{j=0}^{n-3}E_j + E'$.
By the definition of $\Omega^1_{\widetilde{X}}(log(D+E))$, it is a subsheaf of $\Omega^1_{\widetilde{X}}(D+E)$.
Then 
\begin{align}
\begin{split}
&h^0(\widetilde{X}, K_{\widetilde{X}}\otimes\Omega_{\widetilde{X}}(log(D+E))) \leq h^0(\widetilde{X},K_{\widetilde{X}},K_{\widetilde{X}}\otimes 
\Omega^1_{\widetilde{X}}(D+E))=\\
&h^0(\widetilde{X},((n-2)F-2D+\sum_{j=1}^{n-3}(jE_j)+(n-2)E')\otimes\Omega^1_{\widetilde{X}}(D+\sum_{j=0}^{n-3}E_j)) \\
&= h^0(\widetilde{X},\Omega_{\widetilde{X}}^1(-2F+\sum_{j=1}^{n-3}(jE_j)+(n-3)E'))
\leq h^0(\widetilde{X}\setminus (E\cup E'),\Omega_{\widetilde{X}}^1(-2F)) =0.
\end{split}
\end{align}
The last equality holds because $F$ can be a generic fiber, so the holomorphic section vanishes generically and thus vanishes everywhere.
This implies that there is no local-to-global obstruction for deformations of $X$ which preserve the divisor $D$.
The fixed divisor $D$ can be used to construct the deformation to the normal cone.
As a result,
there exists a deformation $\Z\rightarrow \Delta$, where $\Z_0 \simeq Z$ and $\Delta \subset \CC$, 
and each smooth fiber $\Z_t$ is a Stein manifold diffeomorphic to $\oo_{\CC P^1}(-n)$.
Then by using the argument as in Section \ref{non-Artin}, we can construct a family of SFK ALE metrics which degenerates to the orbifold metric
on $Z$.
\end{proof}

\bibliography{Han_Viaclovsky}
\end{document}